\documentclass[11pt]{amsart}

\usepackage{epigamath}

\usepackage[english]{babel}

\numberwithin{equation}{section}

\usepackage{hyperref}
\usepackage{amsthm,amssymb,latexsym,amsmath}
\usepackage[all]{xy}
\usepackage{tikz}

\usepackage{color}

\usepackage[shortlabels,inline]{enumitem}

\newtheorem{theorem}{Theorem}[section]
\newtheorem{mthm}{Main Theorem}
\newtheorem{proposition}[theorem]{Proposition}
\newtheorem{lemma}[theorem]{Lemma}
\newtheorem{corollary}[theorem]{Corollary}

\theoremstyle{definition}
\newtheorem{definition}[theorem]{Definition}

\theoremstyle{remark}
\newtheorem{remark}[theorem]{Remark}
\newtheorem{example}[theorem]{Example}

\newcommand{\Z}{{\mathbb Z}}
\newcommand{\C}{\mathbb{C}}
\newcommand{\R}{\mathbb{R}}
\newcommand{\HH}{\mathbb{H}}
\newcommand{\p}[1]{{\mathbb{P}^{#1}}}

\newcommand{\op}[1]{{\mathcal O}_{\mathbb{P}^{#1}}}

\newcommand{\ox}{{\mathcal O}_{X}}

\newcommand{\leftgamma}{R^L}
\newcommand{\Mu}{M}  

\newcommand{\ch}{\operatorname{ch}}
\newcommand{\td}{\operatorname{td}}

\newcommand{\cala}{{\mathcal A}}
\newcommand{\calb}{{\mathcal B}}

\newcommand{\calh}{{\mathcal H}}
\newcommand{\cali}{{\mathcal I}}

\newcommand{\calm}{{\mathcal M}}

\newcommand{\calo}{{\mathcal O}}

\newcommand{\dbx}{D^{\rm b}(X)}
\newcommand{\coh}{\mathcal{C}oh}

\newcommand{\cohb}{\mathcal{B}^{\beta}}
\newcommand{\cohab}{\mathcal{A}^{\alpha,\beta}}
\newcommand{\nuab}{\nu_{\alpha,\beta}}
\newcommand{\labs}{\lambda_{\alpha,\beta,s}}
\newcommand{\free}[1]{\mathcal{F}_{#1}}
\newcommand{\tors}[1]{\mathcal{T}_{#1}}

\newcommand{\obeta}{{\overline{\beta}}}
\newcommand{\oalpha}{\overline{\alpha}}
\newcommand{\olambda}{\overline{\lambda}}

\newcommand{\inhom}{{\mathcal H}{\it om}}
\newcommand{\inext}{{\mathcal E}{\it xt}}

\newcommand{\Ext}{\operatorname{Ext}}
\newcommand{\Hom}{\operatorname{Hom}}
\newcommand{\RHom}{\operatorname{R\mathcal{H}\textit{om}}}

\DeclareMathOperator{\Pic}{{Pic}}

\newcommand{\into}{\hookrightarrow}
\newcommand{\onto}{\twoheadrightarrow}

\EpigaVolumeYear{6}{2022} \EpigaArticleNr{22} \ReceivedOn{October 2,
2020}
\InFinalFormOn{May 19, 2022}
\AcceptedOn{August 4, 2022}

\title{Walls and asymptotics for Bridgeland stability conditions\\ on threefolds}
\titlemark{Walls and asymptotics on threefolds}

\author{Marcos Jardim}
\address{IMECC - UNICAMP,  Departamento de Matem\'atica, 
  Rua S\'ergio  Buarque de Holanda, 651, 13083-970 Campinas-SP, Brazil}
\email{jardim@ime.unicamp.br}
\author{Antony Maciocia}
\address{University of Edinburgh, School of Mathematics, The King's Buildings,  Peter Guthrie Tait Road, Edinburgh,  EH9 3FD, UK}
\email{A.Maciocia@ed.ac.uk}

\authormark{M.~Jardim and A.~Maciocia}

\AbstractInEnglish{We consider Bridgeland stability conditions for threefolds conjectured by Bayer--Macr\`i--Toda in the case of Picard rank $1$. We study the differential geometry of numerical walls, characterizing when they are bounded, discussing possible intersections and showing that they are essentially regular. Next, we prove that walls within a certain region of the upper half plane that parameterize geometric stability conditions must always intersect the curve given by the vanishing of the slope function and, for a fixed value of $s$, have a maximum turning point there. We then use these facts to prove that Gieseker semistability is equivalent to a strong form of asymptotic semistability along a class of paths in the upper half plane, and we show how to find large families of walls. We illustrate how to compute all of the walls and describe the Bridgeland moduli spaces for the Chern character $(2,0,-1,0)$ on complex projective $3$-space in a suitable region of the upper half plane.}



\acknowledgement{MJ is partially supported by the CNPq grant number 302889/2018-3 and the FAPESP Thematic Project 2018/21391-1; he thanks the University of Edinburgh for a pleasant and productive visit in 2017, which was supported by the FAPESP grant number 2016/03759-6.}

\begin{document}


\maketitle

\begin{prelims}

\DisplayAbstractInEnglish

\bigskip

\DisplayKeyWords

\medskip

\DisplayMSCclass

\end{prelims}


\newpage

\setcounter{tocdepth}{1}

\tableofcontents


\section{Introduction}

Bridgeland's notion of stability conditions on triangulated categories, introduced in \cite{B07} and \cite{B08}, provides a new set of tools to study moduli spaces of sheaves on smooth projective varieties. Such tools have been successfully applied by many authors first to the study of sheaves on surfaces, see for example \cite{AM, AB, BM14a, BM14b, SF16, SF17, FL, MM, YY12}, and more recently on threefolds (especially $\p3$); see for instance \cite{GHS,MS2,S,S18}. One way to study moduli spaces of sheaves using Bridgeland stability spaces is to restrict attention to the so-called geometric stability conditions parameterized by (a subset of) the upper half plane $\HH$. Once we know that the moduli space of Bridgeland-stable objects is asymptotically given by the Gieseker semistable moduli space along an unbounded path, we can try to locate all the points where the moduli space changes along this path (these isolated points are called walls) and compute the change to the moduli space. Eventually, we might reach a point where the Bridgeland space is empty, and then we can reverse our steps to reconstruct the Gieseker moduli space.

For surfaces, this is a fairly well-understood process. In that case, it is known that the geometric stability space is non-empty, that there only finitely many walls in $\HH$ away from the $\beta$-axis which are nested semicircles centered along the horizontal axis and that Bridgeland stability is asymptotic to (twisted) Gieseker stability. Furthermore, there is an effective algorithm to find all such walls for a given Chern character, and then we can carry out the process above to recover the moduli space of semistable sheaves. One approach to finding walls in this case is to observe that every wall for a given Chern character $v$ intersects a special curve which we will denote by $\Theta_v$ in this paper, given by the vanishing locus of the slope function $\nu_{\alpha,\beta}(v)$, and we can then restrict our attention to finding walls along $\Theta_v$.

The whole process becomes much more complicated for threefolds. We can still use the $2$-dimensional construction, but it does not produce full stability conditions, and it is unable to detect sufficient features of the Gieseker moduli spaces since the latter does not coincide, in general, with the asymptotic moduli space.

The first step to improve this was made possible by a number of results guaranteeing the existence of Bridgeland stability conditions on the derived category of sheaves on different types of threefolds, based on the pioneering work of Bayer, Macr\`i and Toda \cite{BMT}. Their idea is to start with the surface case and tilt again. This provides a full stability condition, and the family of moduli spaces is considerably more refined than the one provided by the first tilt. Even though there is no general result which shows that their construction works for all smooth threefolds, it is known to work for a wide variety of relevant examples: $\p3$ \cite{M-p3}, smooth quadric threefolds, \textit{cf.} \cite{S-q3}, abelian threefolds, \textit{cf.} \cite{BMS,MP}, Fano threefolds with Picard rank $1$, \textit{cf.}  \cite{Li}, more general Fano threefolds, \textit{cf.} \cite{BMSZ,Piy}, and smooth quintic threefolds, \textit{cf.} \cite{Li2}.  More precisely, the geometric stability conditions constructed by Bayer, Macr\`i and Toda via the generalized
Bogomolov--Gieseker inequality proposed in \cite{BMT} depend on three real parameters $(\alpha,\beta,s)\in\R^+\times\R\times\R^+$. For each of these, we have an abelian category $\cohab$ and a slope function $\labs$ which allows us to test the stability of objects of $\cohab$. There are, however, known counterexamples (see \cite{S17} and \cite{MSD}) where the generalized Bogomolov--Gieseker inequality fails.

The goal of this paper is to advance on the other two stages of the process outlined above, namely the understanding of the structure of walls and that of asymptotic stability. We only consider the case where $X$ is a smooth projective threefold of Picard rank $1$ over an algebraically closed field of characteristic $0$. This means that we can view our Chern classes (and their twists) as purely numerical vectors of the form $v=(v_0,v_1,v_2,v_3)$.

In order to study walls, we start by providing a uniform way to define the slope functions and their differences in terms of skew-symmetric functions. We go on to consider a number of general properties of numerical $\lambda$-walls (defined as the locus where two $\lambda$-slopes are equal) in Section~\ref{sec:nu-walls}. They are, in general, quartic curves, possibly unbounded and not connected.

In our first main result, we find a simple characterization of those numerical $\lambda$-walls which are bounded, and we show that unbounded walls satisfy a version of Bertram's nested wall theorem. Given numerical Chern characters $v$, $u$ and $u'$, we define $\delta_{01}(u,v):=u_0v_1-u_1v_0$ and an equivalence relation $u\sim_v u'$ which is essentially that the $\lambda$-walls for $v$ corresponding to $u$ and $u'$ are the same; see \eqref{v-eqv} for a precise definition. We also remark that when $v$ is a numerical Chern character satisfying the Bogomolov--Gieseker inequality $v_1^2-2v_0v_2\ge0$ and $v_0\ne0$, the curve $\Theta_v$ allows us to divide the upper half plane $\HH$ into four regions (see Figure~\ref{fig:theta} for an example and Section~\ref{regions} for details).

\begin{mthm}\label{mthm1} 
Suppose $v_0\neq0$ and $u\not\sim_v u'$.
\begin{enumerate}
\item The numerical $\lambda$-wall for $v$ corresponding to $u$ is bounded if and only if $\delta_{01}(u,v)\ne0$. 
\item If $\,\delta_{01}(u,v)=0=\delta_{01}(u',v)$, then the numerical $\lambda$-walls corresponding to $u$ and $u'$ do not intersect. 
\item If $\,\delta_{01}(u,v)\neq0$ and $\ch_{\le2}(u)=\ch_{\le2}(u')$, then the numerical $\lambda$-walls for $v$ corresponding to $u$ and $u'$ only intersect on $\Theta_v$.
\item An unbounded numerical $\lambda$-wall for $v$ does not intersect $\Theta_v$, and its
  unbounded connected components are contained in $R^0_v$. 
\end{enumerate}
\end{mthm}

The different parts of Main Theorem~\ref{mthm1} are proved in various results contained in Section~\ref{sec:walls}. 


\begin{figure}[ht] \centering
\includegraphics{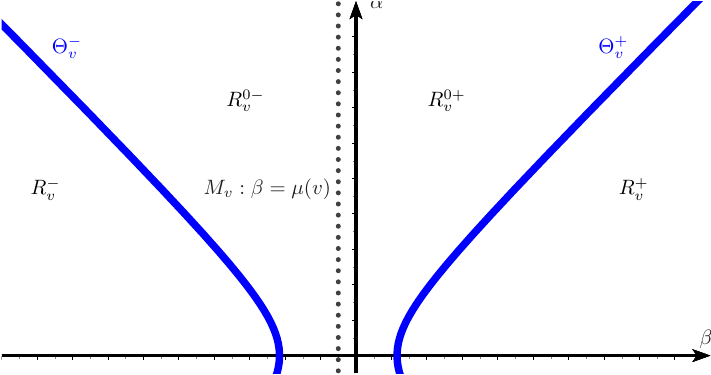}
\caption{The four regions of the plane as defined by the hyperbola $\Theta_v$ and the vertical line $\{\beta=\mu(v)\}$ when $v$ is a numerical Chern character satisfying the Bogomolov--Gieseker inequality $v_1^2-2v_0v_2\ge0$ and $v_0\ne0$. In this picture, we used $\ch_{\le2}(v)=(2,-1,-5/2)$.}
\label{fig:theta}
\end{figure}

The third stage, determining the asymptotics, is not well known. In \cite[Section 6]{BMT}, it is shown that the large volume limit as $\alpha\to\infty$ for $(\alpha,\beta)\in\HH$ gives a polynomial stability condition, but the converse and other directions were not considered. However, unlike the $2$-dimensional case, we cannot assume  the walls are bounded in all directions. In fact, it is easy to check that as $s\to 0$, the walls are unbounded, and it is theoretically possible that the number of walls is infinite. 

This means that the large volume limit is more subtle than for surfaces. We need, therefore, to be more careful about what we mean by asymptotic stability, which we define precisely in Definition~\ref{asym s-st}. We use a strong form of such asymptotic stability which effectively includes finiteness of the number of walls for a given object. We also need to be careful to specify the curve along which we are considering the asymptotics. To this end, we introduce the notion of \textit{unbounded} $\Theta^\pm$-\textit{curve}, which is essentially a curve which is asymptotically either to the left or to the right of all $\Theta$-curves; see Definition~\ref{theta-curve}. 

To help set this up, we also consider what would happen for surfaces, in Section~\ref{nuab sst}. In our case, we look at so-called $\nu$-stability for threefolds, which mimics stability for surfaces given by the first tilt on the category of coherent sheaves, by re-proving results about the large volume limit without the assumption that the walls are bounded.

In order to describe the asymptotics, we need to understand how the stability of an object varies along curves. For $\nu$-stability, it turns out that stable objects can only be destabilized once along inward-moving curves (which cross $\nu$-walls only once). It turns out that this also holds for $\lambda$-stability outside the $\nu$-wall. See Theorem~\ref{one_lambda_wall} for the details.  Accomplishing such a task requires an understanding of the geometry of $\lambda$-walls. For the $\nu$-walls, this was simple because they were circles, and the key property is that they cross the $\Theta_v$-curve at their maximum. To understand the similar properties of $\lambda$-walls, we need to understand their differential properties in a similar way. We do this is Section~\ref{sec:suf-walls}. Unlike $\nu$-walls, $\lambda$-walls need not be regular, and we carefully study when regularity fails. When $s=1/3$, it is straightforward to prove that the walls are regular except if they happen to cross a very special point (on $\Gamma_{u,1/3}$ and its $\nu$-wall), in which case there is a cusp. For other values of $s$, it is much harder. It turns out to be easier to study the differential properties of the $2$-dimensional wall where we allow $s$ to vary, which we call $\Sigma_{u,v}$, regarded as a real algebraic quartic surface in $\mathbb{R}^3$. We show that $\Sigma_{u,v}$ is regular except at some exceptional points and for exceptional $u$ and $v$ (see Theorem~\ref{regular} for the details).

Moreover, there is again a special curve, here denoted by $\Gamma_{v,s}$, which is defined as the vanishing locus of the slope function $\lambda_{\alpha,\beta,s}(v)$. When $s=1/3$, a numerical $\lambda$-wall for $v$ crosses $\Gamma_{v,1/3}$ at its maximum point (just as a $\nu$-wall crosses $\Theta_v$ at its maximum) and the associated $\nu$-wall at its minimum point. This imposes large constraints on the possible numerical $\lambda$-walls when $s=1/3$. We also show that for $s\geq1/3$, any wall existing for one value of $s$ must also exist for all $s$. When $s<1/3$, we show that any wall existing for $s$ exists for all value less than that  of $s$.

The key conclusion is the following; see also Theorem~\ref{RLvs}. The proof uses key differential-geometric information about $\Sigma_{u,v}$ such as its Gauss and mean curvatures.

\begin{mthm}\label{mthm2} 
Suppose a real numerical Chern character $v$ satisfies the Bogomolov--Gieseker inequality and $v_0\neq0$. Any connected bounded component of a numerical $\lambda$-wall in $R^{-}_{v}$ for some $s\geq1/3$ intersects $\Gamma^-_{v,s}$.
\end{mthm}

Although the same statement is not true for unbounded walls, we can describe the explicit conditions $u$ must satisfy so that the wall corresponding to $u$ intersects $\Gamma_{v,s}$.

We are then finally in position to prove, in Section~\ref{labs sst -}, that strong asymptotic stability is equivalent to Gieseker stability. Our results can be summarized as follows.

\begin{mthm}\label{mthm3} 
Let $v$ be a numerical Chern character with $v_0\ne0$ and satisfying the Bogomolov--Gieseker inequality; fix $s\geq1/3$.
\begin{enumerate}
\item If $\gamma$ is an unbounded $\Theta^-$-curve, then an object $E\in D^b(X)$ is asymptotically $\lambda_{\alpha,\beta,s}$-$($semi\/$)$stable along $\gamma$ if and only if $E$ is a Gieseker $($semi\/$)$stable sheaf.
\item If $\gamma$ be an unbounded $\Theta^+$-curve, then an object $E\in D^b(X)$ is asymptotically $\lambda_{\alpha,\beta,s}$-$($semi\/$)$stable along $\gamma$ if and only if $E^\vee$ is a Gieseker $($semi\/$)$stable sheaf.
\end{enumerate}
\end{mthm}

Duals of Gieseker semistable sheaves can be described via a technical lemma of independent interest which characterizes duals of torsion-free sheaves on threefolds (see Proposition~\ref{a=dual} for the details). We also emphasize that more is true for the special curve $\Gamma_{v,s}$, which is an example of an unbounded $\Theta^-$-curve: the first part of the previous statement holds for every $s>0$.

The study of the differential geometry of $\lambda$-walls is also useful to help locate them. Here again there is a complication. For $\nu$-walls (or walls for surfaces), a point on an actual $\nu$-wall corresponds to an actual destabilizing sub-object $F$ of an object $E$ which is $\nu$-stable on one side of the wall and $\nu$-unstable on the other. This is then always the case at all points along the numerical $\nu$-wall. In other words, if a portion of a numerical $\nu$-wall is an actual $\nu$-wall, then the whole numerical $\nu$-wall is an actual $\nu$-wall. It follows that to locate actual $\nu$-walls, it suffices to look along, say, a vertical line from the limiting centre or along $\Theta_v$; we do this in two examples, one in Section~\ref{ideal line} and the other in Lemma~\ref{N1_no_tilt}. 

The situation for $\lambda$-walls is quite different. While we have proved that numerical $\lambda$-walls to the left of $\Theta_v$ must intersect $\Gamma_{v,s}$, that need not be true for actual $\lambda$-walls. However, the largest actual $\lambda$-wall (if it exists) must cross $\Gamma_{v,s}$, and this allows us to find $\lambda$-walls by working along $\Gamma_{v,s}$ from infinity.

We work out a complete example for the case of the Chern character $v=(2,0,-1,0)$ on $\p3$. This is done for the region to the left of $\Theta_v$ in Proposition~\ref{gamma+inst}. The key step is to first find what we call \emph{pseudo-walls}. These are given by destabilizing objects which satisfy the Bogomolov--Gieseker and generalized Bogomolov--Gieseker inequalities. This approach to finding $\lambda$-walls complements Schmidt's method \cite[Theorem 6.1]{S}, which is to observe that, essentially, the $\lambda$-walls crossing $\Theta_v$ must also cross their associated $\nu$-wall at the same point. 

We close the paper with a number of examples for the case of $\p3$ in Section~\ref{sec:examples}. First we consider the case of the ideal sheaf of a line; it illustrates one of two typical situations in which the $\Gamma_{v,s}$ curve intersects~$\Theta_v$. In this case, we show that there must exist a vanishing $\lambda$-wall containing the intersection point (see Theorem~\ref{thm for q}), and we illustrate this by constructing such a wall for this example. We also consider the ideal sheaf of a point for which $\Gamma_{v,s}$ and $\Theta_v$ do not cross and  provide some additional stability information for our final example, which is the null correlation sheaves on $\p3$. In this case, we show there are no $\nu$-walls. We exhibit a number of methods to illustrate the various phenomena which can occur and to locate walls, and we show that there is effectively a single wall to the left of $\Theta_v$.  These cases are also considered from a different point of view using Bridgeland stability in \cite{SS}.

Throughout the paper we make very full use of the triangulated structure of the derived category of coherent sheaves and especially the octahedral axiom. We use a higher-dimensional variant of the octahedral axiom, which is key to understanding how objects vary in the heart $\cohab$ along paths in the upper half plane. We illustrate this in Section~\ref{2nd tilt} using Paul Balmer's description of the objects of $\cohab$.

\subsection*{Acknowledgements}
We thank Victor Pretti for his comments on preliminary versions of this paper and Arend Bayer for helpful conversations. We would also like to thank the referee for their careful reading of the paper and helpful comments to fix a few errors and suggestions to help the clarity of the text in numerous places.
We have made extensive use of the software Maxima, Geogebra and Mathematica for calculations and to produce the various pictures present in this text.

\section{Background material and notation}\label{prelim}

Let $X$ be an irreducible, non-singular projective variety of dimension $3$ over an algebraically closed field of characteristic $0$  with $\Pic(X)=\Z$. Fix an  ample generator $L$ of $\Pic(X)$, and write $\ell=c_1(L)$. Our assumptions mean that each object $A\in\dbx$ has a well-defined numerical Chern character
$$
\ch(A):=(\ch_0(A)\ell^3,\ch_1(A)\cdot \ell^2,\ch_2(A)\cdot \ell,\ch_3(A))\in \Z\times\Z\times\frac12\Z\times\frac16\Z.
$$
Abusing notation, we will simply write $\ch_i(A)$ for $\ch_i(A)\cdot \ell^{3-i}$. We will refer to an element of $\R^4=\R\otimes K_{\text{num}}(X)$ as a \emph{real numerical Chern character} and an element $v$ of $\Z\times\Z\times\frac12\Z\times\frac16\Z$ as a Chern character when there is an object $A\in\dbx$ such that $v=\ch(A)$. We write the components as $v=(v_0,v_1,v_2,v_3)$ corresponding to the Chern characters of objects so that the underlying real Chern character $v$ which is the numerical Chern character $\ch(A)$ of an object of $\dbx$ satisfies $v_i=\ch_i(A)$.

Given $\beta\in\R$, recall the definition of the \emph{twisted Chern character} $\ch^\beta(A):=\exp(-\beta)\cdot\ch(A)= (\ch_0^\beta(A),\ch_1^\beta(A),\ch_2^\beta(A),\ch_3^\beta(A))$. So
\begin{align*}
\ch_0^\beta(A) & := \ch_0(A) ; \\
\ch_1^\beta(A) & := \ch_1(A) - \beta\ch_0(A) ; \\
\ch_2^\beta(A) & := \ch_2(A) - \beta\ch_1(A) + \frac{1}{2}\beta^2\ch_0(A) ; \\
\ch_3^\beta(A) & := \ch_3(A) - \beta\ch_2(A) + \frac{1}{2}\beta^2\ch_1(A) - \frac{1}{6}\beta^3\ch_0(A).
\end{align*}

Recall that the $\mu$-slope of a coherent sheaf $E\in\coh(X)$ is defined as follows:
$$ \mu(E):= 
\begin{cases}
\ch_1(E)/\ch_0(E) &\textrm{if } E \textrm{ is torsion-free}, \\
+\infty& \textrm{otherwise}.
\end{cases} $$
In addition, we also define
$$ \mu^+(E):=\max\{\mu(F) \mid F\into E \text{ is a non-zero subsheaf } \} \quad\text{and} $$
$$ \mu^-(E):=\min\{\mu(G) \mid E\onto G \text{ is a non-zero quotient }  \}. $$
As usual, $E$ is said to be $\mu$-(semi)stable if every non-zero subsheaf $F\into E$ satisfies $\mu(F)<(\le)~\mu(E/F)$. So  $E$ is $\mu$-semistable if and only if 
$\mu^+(E)=\mu(E)$ or, equivalently, $\mu^-(E)=\mu(E)$.

\bigskip

\subsection{$\boldsymbol{\nu}$-stability}

Given $\beta\in\R$, consider the following torsion pair on $\coh(X)$:
$$ \tors\beta := \{E\in\coh(X) ~|~ \textrm{every non-zero quotient } E\onto G \textrm{ satisfies } \mu(G)>\beta \}\quad{\rm and} $$
$$ \free\beta := \{E\in\coh(X) ~|~ \textrm{every non-zero subsheaf } F\into E \textrm{ satisfies } \mu(F)\le\beta \} . $$
Tilting on $(\free\beta,\tors\beta)$, one obtains an abelian subcategory $\cohb:=\langle\free\beta[1],\tors\beta\rangle$ of $\dbx$, which is the heart of a t-structure on $\dbx$.

For $B\in\dbx$, let $\calh^{p}(B)$ denote cohomology with respect to $\coh(X)$. Observe that the objects of $\cohb$ are those $B\in\dbx$ such that: 
\begin{itemize}
\item $\calh^p(B)=0$ for $p\ne-1,0$;
\item $\calh^{-1}(B)\in\free\beta$; and
\item $\calh^{0}(B)\in\tors\beta$.
\end{itemize}
In particular, from the definition of $\free\beta$, the sheaf $\calh^{-1}(B)$ must be torsion-free.

Introducing a new parameter $\alpha\in\R^+$, one considers a group homomorphism 
$$ Z_{\alpha,\beta}^{\rm tilt} \colon K_{\rm num}(X)\to\C, $$
called a \emph{central charge}, given by
\begin{equation}\label{Z tilt}
Z_{\alpha,\beta}^{\rm tilt}(B) := -\left( \ch_2^\beta(B) - \frac{1}{2}\alpha^2\ch_0(B)\right) +
\sqrt{-1}\ch_1^\beta(B),
\end{equation}
whose corresponding slope function is
\begin{equation}\label{nu-slope}
\nu_{\alpha,\beta}(B):=\begin{cases}
\dfrac{\ch_2^\beta(B) - \alpha^2\ch_0(B)/2}{\ch_1^\beta(B)} &\textrm{ if } \ch_1^\beta(B)\ne0, \\ 
+\infty & \text{if }\ch_1^\beta(B)=0.
\end{cases} 
\end{equation}
In addition, we also define
$$ \nuab^+(B):=\max\{\nuab(F) ~|~ F\into B \textrm{ in } \cohb,\ F\neq0\} \quad{\rm and} $$
$$ \nuab^-(B):=\min\{\nuab(G) ~|~ B\onto G \textrm{ in } \cohb,\ G\neq0 \}. $$

An object $B\in\dbx$ is said to be $\nu_{\alpha,\beta}$-(semi)stable if $B\in\cohb$ and every non-zero sub-object $F\into B$ within $\cohb$ satisfies $\nu_{\alpha,\beta}(F)<(\le)~\nu_{\alpha,\beta}(B/F)$. Note that $E$ is $\nu_{\alpha,\beta}$-semistable if and only if $\nu_{\alpha,\beta}^+(E)=~\nu_{\alpha,\beta}(E)$ or, equivalently, $\nu_{\alpha,\beta}^-(E)=~\nu_{\alpha,\beta}(E)$.

Every $\mu$-semistable sheaf and every $\nuab$-semistable object $B\in\cohb$ satisfies the usual Bogomolov--Gieseker inequality, which in our situation is purely numerical, see \cite[Corollary 7.3.2]{BMT}:
\begin{equation}\label{B-ineq}
Q^{\rm tilt}(B) := \ch_1(B)^2-2\ch_0(B)\ch_2(B) \ge 0.
\end{equation}
In addition, for certain choices of $X$, every $\nuab$-semistable object $B\in\cohb$ also satisfies the following \emph{generalized Bogomolov--Gieseker inequality}: 
\begin{equation}\label{GB-ineq}\begin{split}
Q_{\alpha,\beta}(B) &= \alpha^2 Q^{\rm tilt}(B)+4(\ch^\beta_2(B))^2-6\ch^\beta_1(B)\ch^\beta_3(B)\\
&=Q^{\rm tilt}(B) (\alpha^2+\beta^2) + \left( 6\ch_0(B)\ch_3(B)-2\ch_1(B)\ch_2(B) \right)\beta +4\ch_2(B)^2 - 6\ch_1(B)\ch_3(B)\\
&\ge 0,
\end{split}
\end{equation}
originally proposed in \cite[Conjecture 1.3.1]{BMT}. This inequality was proved to hold for all Fano and abelian threefolds with Picard rank $1$, see \cite{Li} and \cite{BMS,MP}, respectively, and for the quintic threefold; see \cite{Li2}. We assume from now on that $X$ is such that the generalized Bogomolov--Gieseker inequality \eqref{GB-ineq} holds for all $\nuab$-semistable objects.

Let $\HH:=\R^+\times\R$, thought of as the upper half plane, with coordinates denoted by $(\alpha,\beta)$. We will want to consider the slope function as a function of $\alpha$ and $\beta$, and to this end it is convenient to define the following function on $(\alpha,\beta)\in\HH$:
\begin{equation}\label{rho function}
\rho_{v}(\alpha,\beta) = v_2 - v_1\beta + \frac{1}{2}v_0(\beta^2-\alpha^2),
\end{equation}
which coincides with the numerator of $\nuab(B)$ when $v=\ch(B)$. To simplify the notation, we define $\rho_B(\alpha,\beta):=\rho_{\ch(B)}(\alpha,\beta)$ for objects $B\in\dbx$. 

Note that the pair $(\cohb,Z^{\rm tilt}_{\alpha,\beta})$ is a \emph{weak stability condition} in $D^{\rm b}(X)$, in the sense of \cite[Section 2]{Tod10}, for all pairs $(\alpha,\beta)\in\R^+\times\R$. In practical terms, this gives the following. 

\begin{proposition}\label{proprank1}
Fix $\beta\in\R$. If $B\in\cohb$, then $\ch^\beta_1(B)\geq0$, with equality only if $\rho_{B}(\alpha,\beta)\geq0$ for all $\alpha>0$. 
\end{proposition}

\bigskip

\subsection{$\boldsymbol{\lambda}$-stability}

The next step is to consider the following torsion pair on $\cohb$:
\begin{gather*}
\tors{\alpha,\beta} := \{E\in\cohb \mid \text{every non-zero quotient } E\onto G \text{ satisfies } \nuab(G)>0 \}\quad  \text{ and} \\
\free{\alpha,\beta} := \{E\in\cohb \mid \text{every non-zero sub-object } F\into E \text{ satisfies } \nuab(F)\le0 \} .
\end{gather*}
Tilting on $(\free{\alpha,\beta},\tors{\alpha,\beta})$, one obtains a new abelian subcategory
$\cohab:=\langle\free{\alpha,\beta}[1],\tors{\alpha,\beta}\rangle$ of $D^{\rm b}(X)$, which
is also the heart of a t-structure on $D^{\rm b}(X)$. 

One then introduces a third parameter $s>0$ in order to define a family of central charges
$$ Z_{\alpha,\beta,s} \colon K_{\rm num}(X)\to\C $$
as follows, for $A\in \cohab$:
\begin{equation} \label{Z_abs}
Z_{\alpha,\beta,s}(A) := -\ch_3^\beta(A) + \big(s+1/6\big)\alpha^2\ch_1^\beta(A) +
\sqrt{-1}\left( \ch_2^\beta(A) - \alpha^2\ch_0(A)/2 \right) ,
\end{equation}
whose corresponding slope function is
\begin{equation}\label{lambda-slope}
\labs(A):=\begin{cases}
\dfrac{\ch_3^\beta(A) - \big(s+1/6\big)\alpha^2\ch_1^\beta(A)}{\ch_2^\beta(A) - \alpha^2\ch_0(A)/2} &\textrm{if } \ch_2^\beta(A) - \alpha^2\ch_0(A)/2\ne0, \\ 
+\infty &\textrm{if } \ch_2^\beta(A) - \alpha^2\ch_0(A)/2=0.
\end{cases}
\end{equation}

\begin{remark}
We could also consider a more general central charge, whose real part is
$$ -\ch^\beta_3(A)+b\ch^\beta_2(A)+a\ch^\beta_1(A) $$
for parameters $b\in\R$ and $a\in\R^+$; see \cite[Lemma 8.3]{BMS} and \cite{Piy}. However, we will only consider the special case where $b=0$, while $a=\alpha^2(s+1/6)$.
\end{remark}

An object $A\in\dbx$ is said to be $\labs$-(semi)stable if $A\in\cohab$ and every non-zero sub-object $F\into A$ within $\cohab$ satisfies $\labs(F)<(\le)~\labs(B/F)$.

For further reference, we define for each real numerical Chern character $v\in\R^4$ the following function on $(\alpha,\beta)\in\HH$:
\begin{equation}\label{tau function}
\tau_{v,s}(\alpha,\beta) = v_3 - v_2\beta + \frac{1}{2}v_1\beta^2 - \frac{1}{6}v_0\beta^3 - \left(s+\frac{1}{6}\right)(v_1-v_0\beta)\alpha^2,
\end{equation}
which coincides with the numerator of $\lambda_{\alpha,\beta,s}(A)$ when $v=\ch(A)$. Again, we define $\tau_{A,s}(\alpha,\beta):=\tau_{\ch(A),s}(\alpha,\beta)$ for objects $A\in\dbx$, and we will also write $\tau_v(\alpha,\beta,s)=\tau_{v,s}(\alpha,\beta)$.
Note that, when non-zero, the denominator of $\labs(A)$ is $\rho_A(\alpha,\beta)$.

Finally, a direct consequence of the generalized Bogomolov--Gieseker inequality \eqref{GB-ineq} is that $(\cohab,Z_{\alpha,\beta,s})$ is a (numerical) stability condition, in the sense of \cite[Definition 2.1.1]{BMT}, for every triple $(\alpha,\beta,s)\in\R^+\times\R\times\R^+$. In practical terms, this yields the following. 

\begin{proposition}\label{proprank2}
Fix $(\alpha,\beta)\in\HH$. If $A\in\cohab$ is non-zero, then $\rho_{A}(\alpha,\beta)\geq0$, with equality only if $\tau_{A,s}(\alpha,\beta)>0$ for every $s\in\R^+$.
\end{proposition}

The generalized Bogomolov--Gieseker inequality can also be used to prove a form of the support property for $\labs$-semistability. In the case we are considering where the Picard rank is $1$, we can state it as follows. 

\begin{proposition}[\emph{cf.} \protect{\cite[Theorem 8.7]{BMS}}]\label{support_property}
Suppose $X$ is a smooth threefold with Picard rank $1$ such that the generalized Bogomolov--Gieseker inequality \eqref{GB-ineq} holds for all $\nuab$-semistable objects. If $E\in\cohab$ is $\labs$-semistable, then $Q_{\alpha,\beta}(E)\geq0$.
\end{proposition}

\begin{remark}
This implies that $(\cohab,Z_{\alpha,\beta,s})$ is a (full) Bridgeland stability condition on  $\dbx$.
\end{remark}

\begin{remark}\label{notbog}
  Observe that $\labs$-semistable objects may not satisfy the usual Bogomolov--Gieseker inequality. For example, on $\p3$, the object $A:=\ox[2]\oplus\ox(1)$ is $\lambda_{1/2\sqrt{1+6s},1/2,s}$-semistable for every $s>0$, but $Q^{\rm tilt}(A)=-1$.
\end{remark}

When we come to do more detailed computations, it will also be useful to have a more uniform notation for the various functions of $v\in K_{\rm num}(X)$ introduced above; more precisely, we define the following:
\begin{equation}\label{chab}
\begin{aligned}
\ch^{\alpha,\beta}_0(v):=v_0,&\quad \ch^{\alpha,\beta}_2(v):=\rho_{v}(\alpha,\beta), \\
\ch^{\alpha,\beta}_1(v):=\ch^\beta_1(v), & \quad
\ch^{\alpha,\beta}_3(v):=\tau_{v,1/3}(\alpha,\beta).
\end{aligned}
\end{equation}

By convention, we set $\ch_i^{\alpha,\beta}(v)=0$ for $i\not\in\{0,1,2,3\}$.
Alternatively, one can also define 
\[\ch^{\alpha,\beta}(v)=\operatorname{Re}\,\bigl(\exp(-\beta-\sqrt{-1}\alpha)\cdot v\bigr).\]

The reason for setting $s=1/3$ will become clearer in Section~\ref{lambdawalls}, but one technical reason is that the partial derivatives of $\ch^{\alpha,\beta}_i(v)$ with respect to $\alpha$ and $\beta$ behave very well; more precisely, 
\begin{equation}
\partial_\alpha\ch_i^{\alpha,\beta}(v)=-\alpha\ch_{i-2}^{\alpha,\beta}(v)\quad \text{and}\quad\partial_\beta\ch_i^{\alpha,\beta}(v)=-\ch^{\alpha,\beta}_{i-1}(v).\label{dervch}
\end{equation}
Note that $\tau_{v,s}(\alpha,\beta)=\ch^{\alpha,\beta}_3(v)-\alpha^2(s-1/3)\ch^{\alpha,\beta}_1(v)$. 

We also introduce
\begin{equation}\label{Delta}
\begin{gathered}
\Delta_{ij}(\alpha,\beta) := \ch^{\alpha,\beta}_i(u)\ch^{\alpha,\beta}_j(v)-\ch^{\alpha,\beta}_j(u)\ch^{\alpha,\beta}_i(v)\quad\text{and} \\
\delta_{ij}(u,v) := \ch_i(u)\ch_j(v)-\ch_j(u)\ch_i(v) = \Delta_{ij}(0,0). \end{gathered}
\end{equation} 
In particular, note that $\Delta_{10}(\alpha,\beta)=\delta_{10}(u,v)$. 

The following is an easy exercise.

\begin{lemma}\label{delta_tech} Fix real numerical Chern characters $u$ and $v$.
\begin{enumerate}
\item\label{delta_identity} We have $\Delta_{01}\Delta_{23}+\Delta_{02}\Delta_{31}+\Delta_{12}\Delta_{03}=0$.
\item\label{delta_dervs} The partial derivatives of $\Delta_{ij}$ are given by
  \begin{gather*}
    \partial_\alpha\Delta_{ij}(\alpha,\beta)=-\alpha\bigl(\Delta_{i-2\,j}(\alpha,\beta)+\Delta_{i\, j-2}(\alpha,\beta)\bigr),\\
\partial_\beta\Delta_{ij}(\alpha,\beta)=-\bigl(\Delta_{i-1\, j}(\alpha,\beta)+\Delta_{i \, j-1}(\alpha,\beta)\bigr).\end{gather*}
\item\label{twovanish} Assume either $\ch^{\alpha,\beta}_i(u)\neq0$ or $\ch^{\alpha,\beta}_i(v)\neq0$. Then, for any $i,j,k\in\{0,1,2,3\}$, if $\Delta_{ij}(\alpha,\beta)=0=\Delta_{ik}(\alpha,\beta)$, then $\Delta_{jk}(\alpha,\beta)=0$.
\item\label{upropv} The following are equivalent: 
\begin{enumerate}[{\rm(a)}]
\item There exist $\alpha,\beta$ and $i\in\{0,1,2,3\}$ such that $\ch_i^{\alpha,\beta}(v)\neq0$ $($or $\ch_i^{\alpha,\beta}(u)\neq0)$ and for all $j\in\{0,1,2,3\}$, $\Delta_{ij}(\alpha,\beta)=0$.
\item For all $\alpha,\beta$ and for all $i,j\in\{0,1,2,3\}$, we have $\Delta_{ij}(\alpha,\beta)=0$.
\item We have $u\propto v$.
\end{enumerate}
\end{enumerate}
\end{lemma}

Note that $u\propto v$ means that $u$ and $v$ are proportional as vectors.

In what follows, we will use $\calm_{\alpha,\beta,s}(v)$ to denote the set of $\labs$-semistable objects with Chern character~$v$; Piyaratne and Toda proved in \cite{PT} that $\calm_{\alpha,\beta,s}(v)$ has the structure of an algebraic stack, locally of finite type over $\C$. Determining whether $\calm_{\alpha,\beta,s}(v)$ also has the structure of a projective variety is an important problem.

\subsection{The second tilt category}\label{2nd tilt}
We will now collect some useful facts about the objects in the second tilt category $\cohab$. Much of the following is well known and is easy to deduce in various \textit{ad hoc} ways, but we give a novel treatment using higher octahedra which is of independent interest (the idea first appeared in 
\cite{Balmer}). We will henceforth drop the $X$ from the notation $\cohab$. Recall that $\calh^i$ denotes cohomology in $\coh(X)$ and $\calh^i_{\mathcal{B}}$ denotes cohomology in $\cohb$.

Suppose $A\in\cohab$ for some $(\alpha,\beta)\in\HH$. We write $A_i=\calh^{-i}_{\mathcal{B}}(A)$ and $A_{ij}=\calh^{-j}(A_i)$. So we have three distinguished triangles:
\begin{equation} \label{3ts}
\begin{gathered}
A_{11}[1]\to A_1\to A_{10} , \\
A_{01}[1]\to A_0\to A_{00} , \\
A_1[1]\to A\to A_0.
\end{gathered}
\end{equation}

Because these triangles intersect, we can arrange them into a diagram as follows:
\begin{equation} \label{big-d}
\xygraph{ 
!~-{@{-}@[|(2.5)]} !{0;/r1.8cm/:,p+/u1cm/::}
*+{A_{01}}="a1"
& *+{A_0[-1]}="a2"
& *+{A_1[1]}="a3"  & *+{A_{10}[1]}="an" \\
& *+{A_{00}[-1]}="a12" & *+{C}="a13"  &
*+{\widetilde A[1]}="a1n" & *+{A_{01}[1]}="a011" \\
&& *+{A}="a23" & *+{C'}="a2n" &
*+{A_0}="a021"  \\
&&&*+{A_{11}[3]}="an1n" & *+{A_1[2]}="a0n11"\\
&&&&*{A_{10}[2].}="a0n1"
"a1":"a2" "a2":"a3" "a3":"an"
"a12":"a13" "a13":"a1n" "a1n":"a011" "a13":"a23" "a23":"a2n"
"a1n":"a2n" "a2n":"an1n" "a021":"a0n11" "an1n":"a0n11" 
"a011":"a021" "a0n11":"a0n1"
"a23":@{~>}"a12" 
"a2":"a12" "a3":"a13" "an":"a1n" "a1n":"a011" "a2n":"a021"  "a12":@{~>}"a1"
"a0n1":@{~>}"an1n" "an1n":@{~>}"a23"
}
\end{equation}

Here the squiggly arrows $X\rightsquigarrow Y$ mean $X\to Y[1]$. The
diagram is meant to repeat infinitely above and to the right by
shifting by $[-n]$ and $[n]$. Every square commutes, and the triangles
along the diagonal are distinguished. Furthermore, each triple of
morphisms formed by composing horizontally and then vertically and
then looping back via the repeated diagram to the right is
distinguished.  The additional objects $C$, $C'$ and $\widetilde A[1]$
are defined as cones on suitable composites. Three of the triangles
are given in display \eqref{3ts}, and the remaining seven are
\begin{equation} \label{7ts}
\begin{gathered}
A_1[1]\to C\to A_{01}[1]\to A_1[2] , \\
A_{10}[1]\to C'\to A_0\to A_{10}[2] , \\
A_{10}[1]\to \widetilde A[1]\to A_{01}[1]\to A_{10}[2] , \\
C\to A\to A_{00}\to C[1] , \\
A_{11}[2]\to A\to C'\to A_{11}[3] , \\
A_{11}[2]\to C\to \widetilde A[1]\to A_{11}[3] , \\
\widetilde A[1]\to C'\to A_{00}\to \widetilde A[2].
\end{gathered}
\end{equation}
The first two tell us what the $\cohb$-cohomologies of $C$ and $C'$ are.
The third tells us that $\widetilde A\in\coh(X)$. The fourth and fifth tell us what the $\cohab$-cohomologies of $C$ and $C'$ are, and the final two tell us what the $\coh(X)$-cohomologies of $C$ and $C'$ are. In particular, we have that
\begin{equation} \label{isom cohomology}
\calh^{-2}(A)\simeq A_{11},\quad\calh^{-1}(A)=\widetilde{A}\quad {\rm and}\quad
\calh^0(A)\simeq A_{00}.
\end{equation}

There are also five distinguished octahedra which are obtained by removing one row and column from the diagram in display \eqref{big-d} (for a more concrete example of an octahedron in this form, see the diagrams \eqref{eq:oct} in the proof of Proposition~\ref{a=dual}). The diagram can also be represented as a $4$-dimensional shape given as a truncated $5$-simplex with five octahedral and five tetrahedral faces.
Another way to express the diagram is that $A$ is filtered in $\dbx$ by
\[A_{00}[-3]\to A_{01}[-1]\to A_{10}\to A_{11}[2]\to A\]
with factors $A_0[-2]$, $\widetilde A$ and $A_1[1]$, respectively.

Observe that $A_{i1}\in\free\beta$, and these must be torsion-free sheaves, while $A_{i0}\in\tors\beta$ and, in particular, $A_{00}\in\cohab$. In fact, a slightly stronger statement is true.

\begin{lemma}\label{h-2 reflexive}
If $A\in\cohab$ for some $(\alpha,\beta)\in\HH$, then $\calh^{-2}(A)$ must be a reflexive sheaf. 
\end{lemma}

\begin{proof}
Assume $\calh^{-2}(A)\simeq A_{11}$ is not reflexive and $T:=A_{11}^{**}/A_{11}\neq0$. Then  $T\to A_{11}[1]$ is a monomorphism in $\cohb$. Since $A_{11}[1]$ is a $\cohb$-sub-object of $A_1$, it must have $\nu_{\alpha,\beta}^+(A_{11})\leq0$. However, $\nu_{\alpha,\beta}(T)=+\infty$, providing a contradiction.
\end{proof}

The following fact will also be useful later on.

\begin{lemma}\label{h-coh-cohb}
If $A\in\cohab$ for some $(\alpha,\beta)\in\HH$ satisfies $\calh^{-2}(A)=\calh^{-1}(\calh^{0}_\beta(A))=0$, then $\calh^{p}(A)\simeq \calh^{p}_\beta(A)$  for $p=-1,0$.
\end{lemma}

\begin{proof}
Chasing through the seven triangles listed in display \eqref{7ts} with $A_{11}=A_{01}=0$, one concludes that $\tilde A\simeq C\simeq A_1$, which is the same as $\calh^{-1}(A)\simeq \calh^{-1}_\beta(A)$ by the isomorphisms in \eqref{isom cohomology}. The vanishing of $A_{01}$ also implies that $A_0\simeq A_{00}$ by the sequences in display \eqref{3ts}; hence $\calh^0(A)\simeq\calh^0_\beta(A)$. 
\end{proof}

We will need to consider short exact sequences in $\cohab$; we look at situations where the
middle term is in either $\cohb$ or $\cohb[1]$.

\begin{lemma}\label{specialize_ses_neg}
If $A\in\cohab\cap\cohb$ and $B\into A\onto C$ is a short exact sequence in $\cohab$, then there are in $\cohab$ a short exact sequence $D\into A\onto C'$, a quotient $C\onto C'$ and an injection $B\into D$ such that $C',D\in\cohb$.
\end{lemma}

\begin{proof}
Apply $\mathcal{H}_{\mathcal{B}}$ to get a long exact sequence in $\cohb$:
\[0\to C_1\to B\to A\to C_0\to 0.\]
Note, in particular, that $B\in\cohb$.
Split this via $D$. Then $\nu^-_{\alpha,\beta}(D)>0$, and so $D\in\cohab$. Set $C'=C_0$. Then we have a short exact sequence $D\to A\to C'$ in $\cohab$ together with an injection $B\into D$ and a surjection $C\onto C'$ also in $\cohab$, as required.
\end{proof}

\begin{lemma}\label{specialize_ses_plus}
If $A\in\cohab\cap\cohb[1]$ and $B\into A\onto C$ is a short exact sequence in $\cohab$, then there are in $\cohab$ a short exact sequence $B'\into A\onto D$, a sub-object $B'\into B$ and a quotient $D\onto C$ with  $B',D\in\cohb[1]$.
\end{lemma}

\begin{proof}
Apply $\mathcal{H}_{\mathcal{B}}$ to get a long exact sequence in $\cohb$:
\[0\to B_1\to A_1\to C_1\to B_0\to 0.\]
In particular, $C\in\cohb[1]$.
Split this via $D[-1]$. Then $\nu^+_{\alpha,\beta}(D)\leq0$, and so $D\in\cohab$. Then we have a short exact sequence $B_1[1]\into A\onto D$ in $\cohab$ together with an injection $B_1[1]\into B$ and a surjection $D\onto C_1[1]=C$ in $\cohab$, as required.
\end{proof}

When the middle term is in $\coh(X)$, we can say a lot more.

{\samepage
  \begin{proposition}\label{sheaf_subobjects}
Suppose $E\in\cohab\cap\coh(X)$ and $D\into E$ is a non-zero monomorphism in $\cohab$. Then $E\in\cohb$ and $D_{01}[2]\in\cohab$. Furthermore, 
there are an $F\in\coh(X)\cap\cohb\cap\cohab$ and monomorphisms $D\overset{\phi}\hookrightarrow F\into E$ in $\cohab$ such that one of the following 
holds for $G:=E/F$ in $\cohab$:  
\begin{enumerate}
    \item\label{sheaf_subobjects1} $G\in\coh(X)\cap\cohb$.
    \item\label{sheaf_subobjects2} $G\in\coh(X)[1]\cap\cohb[1]$, and $\phi$ is the identity.
    \item\label{sheaf_subobjects3} $G\in\coh(X)[1]\cap\cohb$.
\end{enumerate}    
\end{proposition}
}

\begin{proof}
If $A=E\in\cohab\cap\coh(X)$, then the triangles \eqref{3ts} and
\eqref{7ts} imply $C=0$, and then $A_1[1]$ is a sheaf, and so
$A_1=0$. This establishes $E\in\cohb$.

Applying cohomology in $\cohb$ to the triangle $D\to E\to B$, where
$B=E/D$ in $\cohab$, we have that $D_1=0$ and have a $\cohb$ long exact
sequence
\begin{equation}\label{eq:bseq}0\to B_1\to D\to E\to B_0\to0.\end{equation}
So $D\in\cohb\cap\cohab$. Split \eqref{eq:bseq} in the middle via
$Q\in\cohb$, say. From the $\cohb$ short exact sequence $D\to Q\to
B_1[1]$, we see that $Q\in\cohb\cap\cohab$ and $D\to Q$ injects in
$\cohab$. Applying cohomology in $\coh(X)$ to the triangle $Q\to E\to
B_0$, we see that $Q\in\coh(X)$. Applying it to $B_1\to D\to Q$, we get
\[B_{11}\cong D_{01}\quad\text{and}\quad 0\to B_{10}\to D_{00}\to Q\to 0\]
in $\coh(X)$. From this, it follows that $D_{01}\in\cohab[-2]$.

If $D_{01}\neq0$, we set $F=Q$, and then we repeat the above argument
with $Q$ replacing $D$. So we can assume $D_{01}=0$. Suppose
$B_0=0$. Then $Q=E$.  Now set $F=D_{00}=D$. Then we have case~\eqref{sheaf_subobjects2} as
$G:=B_{10}[1]\in\cohb[1]\cap\cohab$.

Otherwise, $B_0\neq0$.  Note that $D\into Q$ in $\cohab$, and so to
find $F$, we may assume $B_1=0$ (by replacing $D$ by $Q$ in the
above). Applying cohomology in $\coh(X)$, we have a long exact
sequence
\[0\to B_{01}\to Q\to E\to B_{00}\to0.\]
If $B_{00}=0$, then we take $F=Q$, and we have case~\eqref{sheaf_subobjects3}. Otherwise, we
split the sequence via $F$. From $Q\to F\to B_{01}[1]$, we see that
$F\in\cohab\cap\cohb$. This is case~\eqref{sheaf_subobjects1}. Note that the triangles $Q\to
F\to B_{01}[1]$ and $B_{01}[1]\to B\to B_{00}$ imply that
$B_{01}[1]\in\cohab$, and so $Q\to F$ injects (and so $D\to F$ also
injects).

The general case can be summarized in a higher octahedron:
\begin{equation}
\xygraph{ 
!~-{@{-}@[|(2.5)]} !{0;/r1.8cm/:,p+/u1cm/::}
*+{B_{00}[-1]}="a1"
& *+{F}="a2"
& *+{\widetilde B[1]}="a3"  & *+{B_{01}[1]}="an" \\
& *+{E}="a12" & *+{B}="a13"  &
*+{B_0}="a1n" & *+{B_{00}}="a011" \\
&& *+{D[1]}="a23" & *+{Q[1]}="a2n" &
*+{F[1]}="a021"  \\
&&&*+{B_{10}[2]}="an1n" & *+{\widetilde B[2]}="a0n11"\\
&&&&*{B_{01}[2].}="a0n1"
"a1":"a2" "a2":"a3" "a3":"an"
"a12":"a13" "a13":"a1n" "a1n":"a011" "a13":"a23" "a23":"a2n"
"a1n":"a2n" "a2n":"an1n" "a021":"a0n11" "an1n":"a0n11" 
"a011":"a021" "a0n11":"a0n1"
"a23":@{~>}"a12" 
"a2":"a12" "a3":"a13" "an":"a1n" "a1n":"a011" "a2n":"a021"  "a12":@{~>}"a1"
"a0n1":@{~>}"an1n" "an1n":@{~>}"a23"
}\vspace{-\baselineskip}
\end{equation}
\end{proof}

\subsection{Duals of semistable sheaves}\label{sst sheaves}

Given an object $A\in\dbx$, we denote its derived dual by $A^\vee:=\RHom(A,\ox)[2]$. For a sheaf $E$, its derived dual $E^\vee$ satisfies $\calh^{j}(E^\vee) = \inext^{j+2}(E,\ox)$ for $j=-2,-1,0,1$ and $\calh^{j}(E^\vee)=0$ otherwise. If $E$ is torsion-free, then we have the following short exact sequence in $\coh(X)$: 
\begin{equation} \label{std sqc}
 0 \to E \to E^{**} \to Q_E \to 0,\quad Q_E:=E^{**}/E,
\end{equation}
where $E^*:=\inhom(E,\ox)$. Note that $\dim Q_E\le1$; letting $Z_E$ be the maximal $0$-dimensional subsheaf of $Q_E$, define $T_E:=Q_E/Z_E$. Let $E'$ be the kernel of the composed epimorphism $ E^{**} \onto Q_E \onto T_E$; it fits into the following short exact sequence in $\coh(X)$:
\begin{equation} \label{e' sqc}
 0 \to E \to E' \to Z_E \to 0 .
\end{equation}
We then have that $\calh^{j}(E^\vee)=0$ for $j\ne-2,-1,0$, and 
$$ \calh^{-2}(E^\vee) = E^*\simeq E'^*,\quad \calh^{-1}(E^\vee) = \inext^1(E,\ox)\simeq\inext^1(E',\ox)\quad {\rm and} $$
$$ \calh^{0}(E^\vee) = \inext^2(E,\ox) \simeq \inext^3(Z_E,\ox) . $$
Moreover, $\inext^1(E,\ox)\simeq\inext^1(E',\ox)$ fits into the short exact sequence
$$ 0 \to \inext^1(E^{**},\ox) \to \inext^1(E,\ox) \to \inext^2(T_E,\ox) \to 0 . $$

Clearly, $E^*\in\free{\beta}$ for $\beta>\mu^+(E^*)$, while $\inext^1(E,\ox)\in\tors{\beta}$ for every $\beta$; it follows that $E'^\vee[-1]\in\cohb$ for every $\beta>\mu^+(E^*)$. In addition, $Z_E^\vee[1]$ is a $0$-dimensional sheaf, so $Z_E^\vee[1]\in\tors{\beta}\subset\cohb$ for every $\beta$. Comparing with the triangle
\begin{equation} \label{tri z_e}
E'^\vee \to E^\vee \to Z_E^\vee[1] 
\end{equation}
obtained from dualizing the sequence in display \eqref{e' sqc},  we have proved the following. 

\begin{proposition}\label{B-cohom-dual}
If $E$ is a torsion-free sheaf and
\[0\to E\to E'\to Z_E\to 0\]
is the sequence \eqref{e' sqc}, where $Z_E$ is the maximal $0$-dimensional subsheaf of $E^{**}/E$, then
\begin{equation} \label{isom-E}
\calh_\beta^{-1}(E^\vee)\simeq E'^\vee[-1] \quad\text{and}\quad
\calh_\beta^{0}(E^\vee)\simeq Z_E^\vee[1] \simeq \inext^2(E,\ox)
\end{equation}
whenever $\beta>\mu^+(E^*)$.
\end{proposition}

We can also formulate a converse to this construction.  Let
$\coh(X)_d$ denote the category of coherent sheaves on $X$ of
dimension at most $d$. We want to be able to characterize when an
object $A\in D^b(X)$ which has cohomology in three consecutive places
$A^i$, $A^{i+1}$ and $A^{i+2}$ such that $A^i$ is reflexive and
$A^{i+j}\in\coh(X)_{2-j}$ for $j=1,2$ is the shift of the dual of a
torsion-free sheaf. We will make use of this when we analyze the
asymptotics for $\beta\gg 0$; we will express them in the form of
lifting properties. First we need a technical lemma as an intermediate
step. The equivalent statement in dimension $2$ is an easy exercise:
if $A$ has cohomology in two places with $A^i$ locally free and
$A^{i+1}\in\coh(X)_0$ such that any subsheaf $S\to A^{i+1}$ does not
lift to $A$, then $A$ is the shift of the dual of a torsion-free sheaf.

\begin{lemma}\label{cohom_dual}
An object $A\in D^b(X)$ satisfies the  conditions
\begin{enumerate}
\item $\calh^i(A)=0$, $i\neq-2,-1,0$,
\item $F:=\calh^{-2}(A)$ is a reflexive sheaf,
\item $G:=\calh^{-1}(A)\in\coh(X)_1$,
\item $S:=\calh^0(A)\in\coh(X)_0$,
\item\label{third_part} the induced map $F^*\to \inext^2(G,\ox)$ surjects, with kernel $K$, say,
\item\label{second_part} the induced map $\inext^1(F,\ox)\to \inext^3(G,\ox)$ is an isomorphism, and
\item\label{first_part} the induced map $f\colon K\to\inext^3(S,\ox)$ surjects
\end{enumerate}
if and only if $E:=A^\vee\cong\ker f$ is a torsion-free sheaf.
\end{lemma}

\begin{proof}
To see where these induced maps come from, consider the triangles
\begin{gather*}F[2]\to A\to \tilde F\to F[3],\\
G[1]\to\tilde F\to S\to G[2].\end{gather*}
Then we have maps
\[\inhom(F,\ox)\to \RHom(\tilde F[-3],\ox)\to\RHom(G[-2],\ox)\cong\inext^2(G,\ox)\]
and $\inext^1(F,\ox)\to \RHom(\tilde F[-4],\ox)\to\inext^3(G,\ox)$. The final one arises because there is a short exact sequence of sheaves
\[0\to\inext^3(S,\ox)\to \inext^3(\tilde F,\ox)\to\inext^2(G,\ox)\to0,\]
and so  $K\to\inext^3(\tilde F,\ox)$ lifts to $K\to \inext^3(S,\ox)$.

The lemma follows immediately from the spectral sequence 
\[\inext^{-p}(\calh^q(A),\ox)\Rightarrow \calh^{q-p+2}(A^\vee),\]
which converges on the third page to a single entry at
$(p,q)=(0,-2)$. Note that $E$ is torsion-free because it is the
subsheaf of a reflexive sheaf (namely $F^*$). We also have that
$E'\cong K$, $Z_E\cong\inext^3(S,\ox)$, $Q_E\cong\inext^2(G,\ox)$ and
$T_E\cong\inext^3(G,\ox)$.
\end{proof}

We can give a more categorical description of the conditions as follows. 

\begin{proposition}\label{a=dual}
Suppose $A\in D^b(X)$. Then $A$ satisfies the  conditions
\begin{enumerate}
\item\label{a=dual1} $\calh^i(A)=0$, $i\neq-2,-1,0$,
\item\label{a=dual2} $F:=\calh^{-2}(A)$ is a reflexive sheaf,
\item\label{a=dual3} $G:=\calh^{-1}(A)\in\coh(X)_1$,
\item\label{a=dual4} $S:=\calh^0(A)\in\coh(X)_0$,
\item\label{a=dual5} no monomorphism $S'\to S$ in $\coh(X)$ lifts to $A$, and
\item\label{a=dual6} no monomorphism $G'\to G$ in $\coh(X)$ lifts to $A[-1]$
\end{enumerate}
if and only if $E:=A^\vee$ is a torsion-free sheaf.
\end{proposition}

\begin{proof}
It is easy to see that the final two conditions are necessary because otherwise we would have maps $E\to \inext^3(S',\ox)[-3]$ or $E\to\inext^2(G',\ox)[-2]$ or $E\to\inext^3(G',\ox)[-3]$, which are impossible.

For the converse, we show that \eqref{a=dual5} implies Lemma~\ref{cohom_dual}\eqref{first_part} while \eqref{a=dual6} implies Lemma~\ref{cohom_dual}\eqref{third_part} and Lemma~\ref{cohom_dual}\eqref{second_part} by considering the contrapositives. For the proof, we consider the octahedron
\begin{equation}\label{eq:oct}
\vcenter{\vbox{\xymatrix@=1.5pc{F[2]\ar@{=}[r]\ar[d]&F[2]\ar[d]\\
B[1]\ar[r]\ar[d]&A\ar[r]\ar[d]&S\ar@{=}[d]\\
G[1]\ar[r]&\tilde F\ar[r]&S}}}\text{\qquad or}
\vcenter{\vbox{\xygraph{ 
!~-{@{-}@[|(2.5)]} !{0;/r1.8cm/:,p+/u1cm/::}
*+{S[-1]}="a1"
& *+{B[1]}="a2"
& *+{G[1]}="a3"  \\
& *+{A}="a12" & *+{\widetilde{F}}="a13"  &
*+{S}="a1n" \\
&& *+{F[3]}="a23" & *+{B[2]}="a2n"  \\
&&&*+{G[2].}="an1n"
"a1":"a2" "a2":"a3" 
"a12":"a13" "a13":"a1n"  "a13":"a23" "a23":"a2n"
"a1n":"a2n" "a2n":"an1n" 
"a23":@{~>}"a12" 
"a2":"a12" "a3":"a13"  "a12":@{~>}"a1"
"an1n":@{~>}"a23"
}}}\end{equation}
Observe that the triangle $F^\vee[-2]\to G^\vee\to B^\vee$ shows that when Lemma~\ref{cohom_dual} holds, $B^\vee=K[1]$.

First suppose $K\to\inext^3(S,\ox)$ fails to surject. Let the quotient be denoted by $T$; this is in $\coh(X)_0$. Let $S'$ be $\inext(T,\ox)$. Then $S'\to S$, and the composite with $S\to B[2]$ vanishes and so lifts to $A$, as required.

Now suppose item~\eqref{third_part} or~\eqref{second_part} of Lemma~\ref{cohom_dual} fails. Applying cohomology to the left vertical triangle of our octahedron and abbreviating $\mathcal{H}^i(B)$ to $B^i$, we have
\[0\to B^{-1}\to F^*\to\inext^2(G,\ox)\to B^0\to\inext^1(F,\ox)\to\inext^3(G,\ox)\to B^1\to0.\]
Then the cone on $B^{-1}[1]\to B$ is the dual of an object $G'$ supported in dimension $1$, and there is a map $G'\to G$ which lifts to $G'\to B$. But there are no morphisms $G'[1]\to T[-1]$, and so this lifts to a non-zero map $G'[1]\to A$, as required.
\end{proof}

\begin{remark}
We can rephrase item~\eqref{a=dual5}  in Proposition~\ref{a=dual} to say that if $K\to A$ is a map from a sheaf in $\coh(X)_0$, then the induced map $K\to S$ must be zero. Similarly, item~\eqref{a=dual6} becomes the statement that if $K[1]\to A$ is a map with $K\in\coh(X)_1$, then the induced map $K\to G$ must be zero. In practice, these are easier to test, but we will make more use of the conditions stated in Proposition~\ref{a=dual}. It is also clear how to extend the statement to higher dimensions.
\end{remark}

\subsection{Refining notions of stability for sheaves}
We complete this section by introducing a stability condition in $\coh(X)$ which interpolates between $\mu$-stability and Gieseker stability; it will play a role in some of the proofs below.

The Hilbert polynomial of a coherent sheaf $E$ on $X$ with respect to the polarization $L$ is
$$ P_E(t) := \chi(E\otimes L^{\otimes t}) = \ch_0(E)\chi(L^{\otimes t}) + \ch_1(E) x_2(t) + \ch_2(E) x_1(t) + \ch_3(E) ,$$
where 
$$ x_2(t) := \left( \frac{1}{2}t^2 L^2 + t \td_1(X)\cdot L + \td_2(X) \right), \quad  
x_1(t):= \left( tL +\td_1(X) \right), $$
and the $\td_j(X)$ denote the Todd classes of $X$. If $F$ is another coherent free sheaf on $X$, we define, following \cite[Section 2]{R}, 
$$ \Lambda(E,F):=\bigl(\delta_{10}(E,F),\delta_{20}(E,F),\delta_{30}(E,F),\delta_{21}(E,F),\delta_{31}(E,F),\delta_{32}(E,F)\bigr), $$
where $\delta_{ij}(E,F):=\delta_{ij}(\ch(E),\ch(F))$ following the notation introduced in display \eqref{Delta}.

We remark that a coherent sheaf $E$ is Gieseker (semi)stable if and only if every proper, non-trivial subsheaf $F\into E$ satisfies $\Lambda(E,F)>~(\geq)~0$ in the lexicographic order. For instance, assume  $E$ is torsion-free, and let $F\into E$ be a proper subsheaf; letting $p_E(t)$ denote the reduced Hilbert polynomial of the sheaf $E$, we have
$$ p_E(t) - p_F(t) = \dfrac{1}{\ch_0(E)\ch_0(F)} \bigl( \delta_{10}(E,F) x_2(t) + \delta_{20}(E,F) x_1(t) + \delta_{30}(E,F)  \bigr), $$
so $E$ is Gieseker (semi)stable if and only if
\[\bigl(\delta_{10}(E,F),\delta_{20}(E,F),\delta_{30}(E,F)\bigr)>~(\geq)~0\] in the lexicographic order. Similarly, a torsion-free sheaf $E$ is Gieseker (semi)stable if and only if every quotient sheaf $E\onto G$ satisfies \[\bigl(\delta_{10}(E,G),\delta_{20}(E,G),\delta_{30}(E,G)\bigr)<~(\leq)~0\]
in the lexicographic order.

In what follows, it will be important to consider another notion of stability for torsion-free sheaves, which is equivalent to the notion of stability in the category $\coh_{3,1}(X)$ in the sense of \cite[Definition 1.6.3]{HL}.

\begin{definition}
A torsion-free sheaf $E$ on $X$ is said to be \emph{$\mu_{\le2}$-$($semi$)$stable} if every proper, non-trivial subsheaf $F\into E$ satisfies $(\delta_{10}(E,F),\delta_{20}(E,F))>~(\geq)~0$ in the lexicographic order.
\end{definition}

For the sheaf $E'$ defined as in display \eqref{e' sqc}, we observe that:
\begin{enumerate}[{\rm(i)}]
\item $E$ is $\mu$-(semi)stable if and only if $E'$ is $\mu$-(semi)stable;
\item $E$ is $\mu_{\le2}$-(semi)stable if and only if $E'$ is $\mu_{\le2}$-(semi)stable.
\end{enumerate}
The first claim follows from the fact that $E'^*=E^*$. For the second claim, note that $\ch_{\le2}(E')=\ch_{\le2}(E)$ since $E'/E$ is $0$-dimensional; in addition, any subsheaf $F'\into E'$ will induce a subsheaf $F\into E$ such that $\ch_{\le2}(F')=\ch_{\le2}(F)$.

Clearly, one has the following chains of implications:
$$ \mu\textrm{-stable} ~~\Rightarrow~~ \mu_{\le2}\textrm{-stable} ~~\Rightarrow~~ \textrm{Gieseker stable} ~~\Rightarrow $$
$$ \Rightarrow~~ \textrm{Gieseker semistable} ~~\Rightarrow~~ \mu_{\le2}\textrm{-semistable} ~~\Rightarrow~~ \mu\textrm{-semistable}. $$

\begin{example}
It is not hard to find explicit examples that show that the reverse implications do not hold in general. Indeed, for $X=\p3$, let $S$ be a rank $2$ reflexive sheaf given as an extension of an ideal sheaf of a line $L\subset\p3$ by $\op3$; note that $\ch(S)=(2,0,-1,1)$. Let $C\subset\p3$ be curve, and consider an epimorphism $\varphi\colon S\onto \mathcal{O}_C(k)$; define $E_\varphi:=\ker\varphi$.
\begin{enumerate}
\item If $C$ is a line not intersecting $L$ and $k>0$, then $E_\varphi$ is $\mu_{\le2}$-semistable but not Gieseker semistable.
\item If $C$ is a conic not intersecting $L$ and $k=0$, then $E_\varphi$ is $\mu_{\le2}$-stable but not $\mu$-stable.
\item If $C$ is a line intersecting $L$ in a single point, then $E_\varphi$ is Gieseker stable but not $\mu_{\le2}$-stable.
\end{enumerate}
\end{example}

Finally, we also recall the following notion of stability for sheaves of dimension $2$; compare with \cite[Definition 1.6.8]{HL}.

\begin{definition}\label{muhat-stab}
A torsion sheaf $T\in\coh(X)_2$ on $X$ is said to be $\hat{\mu}$-(semi)stable if it is pure and every subsheaf $U\subset T$ satisfies
$$ \hat{\mu}(U) := \dfrac{\ch_2(U)}{\ch_1(U)} <~(\leq)~ \dfrac{\ch_2(T)}{\ch_1(T)}=\hat{\mu}(T). $$
\end{definition}

Setting $\Lambda_2(E,F):=(\delta_{10}(E,F),\delta_{20}(E,F),\delta_{21}(E,F))$, we introduce the following version of the Harder--Narasimhan filtration, which will be useful later on.

\begin{lemma}\label{hn-filtration}
Every sheaf $\,E$ admits a filtration $0=E_0\subseteq E_1\subset \cdots \subset E_n=E$ satisfying the following conditions:
\begin{enumerate}
\item $\dim E_1\le1$, and $\dim E_k\ge2$ for every $k\ge2$.
\item Each factor $G_k:=E_k/E_{k-1}$ for $k\ge2$ is either $\hat{\mu}$-semistable 
$($if $\dim E_k/E_{k-1}=2)$ or $\mu_{\le2}$-semistable $($if $ E_k/E_{k-1}$ is torsion-free$)$, and
$$ \Lambda_2(G_k,G_{k+1}) > 0\quad \text{ for } k=1,\dots, n-1 . $$
\end{enumerate}
\end{lemma}

\begin{proof}
Let $E_1$ be the maximal subsheaf of $E$ of dimension at most $1$, so that $F:=E/T$ has dimension at least~$2$. Note that this $E_1$ might be the zero sheaf.

Next, let $F=F_0\twoheadrightarrow F_1\twoheadrightarrow \cdots \subset F_n=0$ be the Harder--Narasimhan cofiltration of $F$ as an object in the category $\coh_{3,1}(X)$, in the sense of \cite[Theorem 1.6.7]{HL}; each $F_k$ is a sheaf of dimension at least $2$, either $\hat{\mu}$- or $\mu_{\le2}$-semistable. Composing with $E\twoheadrightarrow F$ provides our required cofiltration and hence a filtration in the usual way.
\end{proof}

\section{Regions of \texorpdfstring{$\boldsymbol{\HH}$}{H}}\label{regions}

Let $v$ be a real numerical Chern character with $v_0\ne0$ satisfying $Q^{\rm tilt}(v)\ge0$. It follows that the curve
$$ \Theta_v:=\{\rho_v(\alpha,\beta)=0\} $$
is a hyperbola in $\HH$, centered around the vertical line $\Mu_v:=\{\beta=\mu(v)\}$; explicitly,
\begin{equation}\label{hyperbola}
\rho_{v}(\alpha,\beta)=0 ~~ \Longleftrightarrow ~~
\left(\beta-\mu(v)\right)^2-\alpha^2 = \frac{Q^{\rm tilt}(v)}{v_0^2}.
\end{equation}
When $v$ fails the Bogomolov--Gieseker inequality, we occasionally still consider $\Theta_v$, but now it is a single-branch hyperbola cutting the $\alpha$-axis.

The hyperbola $\Theta_v$ divides $\HH$ into three regions $R^-\sqcup R^0\sqcup R^+$ defined as follows:
\begin{align*}
R^-_v &= \{(\alpha,\beta)\mid \text{$\rho_v(\alpha,\beta)>0$ and $\beta<\mu(v)$}\} , \\
R^0_v &= \{(\alpha,\beta)\mid \text{$\rho_v(\alpha,\beta)<0$ or $\rho_v(\alpha,\beta)=0$ and $\beta<\mu(v)$}\},\quad{\rm and}\\
R^+_v &= \{(\alpha,\beta)\mid \text{$\rho_v(\alpha,\beta)\geq0$ and $\beta>\mu(v)$}\}.
\end{align*}
Notice that $R^-_v$ does not include any of the $\Theta_v$, whereas $R^0$ and $R^+_v$ include the branch of $\Theta_v$ to their left. In addition, $R^0_v$ is split in half by the vertical line $\{\beta=v_1/v_0\}$; we shall denote these regions by $R^{0+}_v$ and $R^{0-}_v$ (to include the vertical line) to the left and right, respectively. These regions are illustrated in Figure~\ref{fig:theta}. 
\begin{remark} 
The various cases which arise in Proposition~\ref{sheaf_subobjects} can be rewritten usefully in terms of these regions. In the notation of the statement, we have $F\in R^-_{\ch(F)}$, and then the cases are:  
\begin{enumerate}
\item $G\in R^-_{\ch(G)}$,
\item $G\in R^{0+}_{\ch(G)}$,
\item $G\in R^{0-}_{\ch(G)}$.
\end{enumerate}
\end{remark}

The categories $\cohab$ along the curve $\Theta_v$ satisfy some strong conditions. These form the basis for Theorem~\ref{schmidt} (see also \cite[Lemmas 6.2 and 6.4]{S}).

\begin{proposition}\label{nuzero}
If $A\in\cohab$ with $\ch(A)=v$ and $(\alpha,\beta)\in \Theta_v$, then, for every $s>0$, $A$ is $\labs$-semistable, $\calh^0_\beta(A)\in\coh(X)_0$, and $\calh^{-1}_\beta(A)$ is $\nuab$-semistable.

Conversely, if $B\in\cohb$ with $\ch_{\le2}(B)=\ch_{\le2}(v)$ is $\nuab$-semistable for some $(\alpha,\beta)\in \Theta_v$ and $P\in\coh(X)_0$, then every non-trivial extension $A\in\Ext^1(P,B[1])$ with $\ch(A)=v$ belongs to $\cohab$ and is $\labs$-semistable for every $s>0$.
\end{proposition}

\begin{proof}
Suppose $F\into A$ is a sub-object in $\cohab$. Since $\rho_A(\alpha,\beta)=0$, we have that $\rho_F(\alpha,\beta)=0$ as well; thus $\labs(F)=\infty$. The analogous statement holds for quotient objects. Hence, $E$ is $\labs$-semistable for every $s>0$.

Now write $A$ as 
\[0\to A_1[1]\to A\to A_0\to 0\]
in $\cohab$. As in the argument above, we also conclude that $\rho_{A_0}(\alpha,\beta)=0$. However, this contradicts $\nuab^-(A_0)>0$ unless $A_0\in\coh(X)_0$. Since $A_1\in\free{\alpha,\beta}$, we have $\nuab^+(A_1)\le0=\nuab(A_1)$, and so $A_1$ is $\nuab$-semistable.

For the second statement, note that every sub-object $F\into B$ within $\cohb$ satisfies $\nuab(F)\le \nuab(B)=0$ since $B$ is $\nuab$-semistable and $(\alpha,\beta)\in \Theta_v$. It follows that $B\in\free{\alpha,\beta}$; thus $B[1]\in\cohab$. Since $P\in\coh(X)_0$, then $P\in\tors{\alpha,\beta}$ for every $(\alpha,\beta)\in\HH$; thus any extension $A\in\Ext^1(P,B[1])$ does belong to $\cohab$. For every $s>0$, we have $\labs(A)=+\infty$ because $(\alpha,\beta)\in \Theta_v$; thus $A$ must be $\labs$-semistable.
\end{proof}

\subsection{Variation of $\boldsymbol{\cohab}$ along paths}\label{ss:variation_along_paths}
Fix a numerical Chern character $v$,  a point $(\alpha,\beta)\in R^-_v$ and some object $A$ of $\cohab$ with $\ch(A)=v$. Suppose $\gamma\colon [0,b)\to\HH$ is a path in the upper half plane with $\gamma(0)=(\alpha,\beta)$. Let $\pi_i\colon \HH\to\mathbb{R}$ be the projection maps (so that $\pi_1$ picks out the $\alpha$-value and $\pi_2$ the $\beta$-value). First assume  $\beta(t):=\pi_2(\gamma(t))$ is constant and $\alpha(t):=\pi_1(\gamma(t))$ is monotonic increasing. As before we write $A_0$ and $A_1$ for the $\mathcal{B}^\beta$-cohomology of $A$.  Observe that the Harder--Narasimhan factors of $A_0$ and $A_1$ with respect to $\nu_{\alpha,\beta}$ are locally constant (they change only on $\lambda$-walls corresponding to the factors). Let $A^-_0$ be the Harder--Narasimhan factor of $A_0$ with $\nu_{\alpha,\beta}(A^-_0)=\nu_{\alpha,\beta}^-(A_0)$. Then as we move along $\gamma$ (upwards), $A$ remains in $\cohab$ until we cross $\Theta^-_{\ch(A^-_0)}$ (in a Harder--Narasimhan chamber), at which point $\mathcal{H}^1_{\mathcal{A}}(A)=A^-_0$. But since $\nu^+_{\alpha,\beta}(A_1)<0$ and we are traversing the path upwards, for any Harder--Narasimhan factor $A'_1$ of $A_1$ with $\ch(A')_1=v'$, we start in the region $R^0_{v'}$. Hence we can only cross $\Theta^+_{\ch(A_1')}$, but this does not take $A$ out of $\mathcal{A}^{\gamma(t)}$. The reverse happens if $\alpha(t)$ is monotonic decreasing. In this case, if $A^+_1$ is the Harder--Narasimhan factor corresponding to $\nu_{\gamma(t)}^+(A_1)$, then as we cross $\Theta^+_{\ch(A^+_1)}$, $A$ moves out of $\mathcal{A}^{\gamma(t)}$ with $\mathcal{H}_\mathcal{A}^{-1}(A)=A^+_1$. Otherwise, $A$ remains in $\mathcal{A}^{\gamma(t)}$.

This also applies to horizontal paths, but now we need to consider what happens when we cross vertical lines $M_u=\{\beta=\mu(u)\}$ corresponding to $u$ equal to Chern characters of $\mu$-Harder--Narasimhan factors of $A_{00}$, $A_{10}$, $A_{01}$ and $A_{11}$. In this case, we see that crossing these lines does not affect $A$.  To see this, consider one of the cases where $\mu_-(A_{10})=\mu(B)$ and $\widetilde A_{10}=\ker(A_{10}\to B)$ in $\coh(X)$. Then we have an octahedron
\[\xygraph{ 
!~-{@{-}@[|(2.5)]} !{0;/r1.8cm/:,p+/u1cm/::}
*+{A_1}="a1"
& *+{A_{10}}="a2"
& *+{B}="a3"  \\
& *+{A_{11}[2]}="a12" & *+{Q[1]}="a13"  &
*+{A_1[1]}="a1n" \\
&& *+{\widetilde A_{10}[1]}="a23" & *+{A_{10}[1]}="a2n"  \\
&&&*+{B[1]}="an1n"
"a1":"a2" "a2":"a3" 
"a12":"a13" "a13":"a1n"  "a13":"a23" "a23":"a2n"
"a1n":"a2n" "a2n":"an1n" 
"a23":@{~>}"a12" 
"a2":"a12" "a3":"a13"  "a12":@{~>}"a1"
"an1n":@{~>}"a23"
}\]
for some new object $Q$. From this it follows that $Q\in\cohb$ and that
$\nu^*_{\alpha,\beta}(Q)\leq0$ as $Q$ is a $\cohb$-sub-object of $A_1$. So $Q[1]\in\cohab$.
Now we add the triangle $A_1[1]\to A\to A_0$ to the octahedra by inserting a second column and
second row to form a higher octahedron (see Section~\ref{2nd tilt})
\begin{equation}
\xygraph{ 
!~-{@{-}@[|(2.5)]} !{0;/r1.8cm/:,p+/u1cm/::}
*+{A_0[-2]}="a1"
& *+{A_1}="a2"
& *+{A_{10}}="a3"  & *+{B}="an" \\
& *+{A[-1]}="a12" & *+{C'[-1]}="a13"  &
*+{D[-1]}="a1n" & *+{A_{0}[-1]}="a011" \\
&& *+{A_{11}[2]}="a23" & *+{Q[1]}="a2n" &
*+{A_1[1]}="a021"  \\
&&&*+{\widetilde A_{10}[1]}="an1n" & *+{A_{10}[1]}="a0n11"\\
&&&&*{B[1]}="a0n1"
"a1":"a2" "a2":"a3" "a3":"an"
"a12":"a13" "a13":"a1n" "a1n":"a011" "a13":"a23" "a23":"a2n"
"a1n":"a2n" "a2n":"an1n" "a021":"a0n11" "an1n":"a0n11" 
"a011":"a021" "a0n11":"a0n1"
"a23":@{~>}"a12" 
"a2":"a12" "a3":"a13" "an":"a1n" "a1n":"a011" "a2n":"a021"  "a12":@{~>}"a1"
"a0n1":@{~>}"an1n" "an1n":@{~>}"a23"
}
\end{equation}
with a new object $D$,  where $C'$ is from \eqref{7ts}. The triangle $B[1]\to D\to A_0$ implies that $D\in\cohab$. We also see
that $Q[1]\to A\to D$ is a triangle and so $A\in\cohab$. Hence, $A$ remains in
$\cohab$ through $M_u$. The other cases follow similarly.

In conclusion, we have proved the following. 

\begin{theorem}\label{t:stays_in_A}
Suppose $(\alpha_0,\beta_0)\in\HH$ and $A\in\mathcal{A}^{\alpha_0,\beta_0}$. Let $A^-_0$ be the Harder--Narasimhan factor of $A_0\in\mathcal{B}^{\beta_0}$ with least $\nu_{\alpha_0,\beta_0}$ and $A^+_1$ the Harder--Narasimhan factor of $A_1$ with the greatest $\nu_{\alpha_0,\beta_0}$. Also suppose  $\Theta^+_{\ch(A^+_1)}$ does not contain $(\alpha_0,\beta_0)$. Then there is an open neighbourhood $U$ of $(\alpha_0,\beta_0)$ such that $A\in\cohab$ for all $(\alpha,\beta)\in U$. Furthermore, any boundary of $U$ consists of $\,\Theta^+_w$-curves associated to sub-objects in $\cohb$ or $\Theta^-_w$-curves of quotient objects.
\end{theorem}

\subsection{The $\boldsymbol{\Gamma}$-curve}
The next step is to understand the vanishing locus of $\lambda$-slope for a given numerical Chern vector $v$ and to study its relative position with respect to $\Theta_v$.

\begin{definition} \label{Gammas}
For any real numerical Chern character $v$ and $s\ge 0$, we define the curve
$$ \Gamma_{v,s} := \{(\alpha,\beta)\in\HH ~\mid~
\tau_{v,s}(\alpha,\beta)=0\}. $$
In addition, we set $\Gamma_{v,s}^\bullet:=\Gamma_{v,s}\cap R_v^\bullet$, with $\bullet=-,0,+$.
\end{definition}

The curve $\Gamma_{v,s}$ is given explicitly by 
\begin{equation}\label{eq:alphagamma}
v_0\beta ^3 - 3v_1\beta^2 + \left( 6v_2-(6s+1)v_0\alpha^2 \right)\beta + \left( (6s+1)v_1\alpha^2 - 6v_3 \right) = 0 .
\end{equation}
If $v_0\ne0$, equation \eqref{eq:alphagamma} is cubic on $\beta$; hence $\Gamma_{v,s}$ has at most three connected components. In fact, there exists an $\alpha_0>0$ such that
$$ \Gamma_{v,s} \cap \{(\alpha,\beta)\in\HH ~|~ \alpha>\alpha_0 \} $$
has exactly three connected components. Two of these components are asymptotic to lines $\alpha\sqrt{6s+1}=\pm(\beta-v_1/v_0)$ and therefore lie in the regions $R^\pm_{v}$; these are the curves $\Gamma_{v,s}^\pm$ in Definition~\ref{Gammas} above. The third component, namely $\Gamma_{v,s}^0$, either coincides with or is asymptotic to the vertical line $\beta=v_1/v_0$.

\begin{example}\label{m0n0}
When $v=(m,0,-n,0)$ for some $m,n>0$,  equation \eqref{eq:alphagamma} simplifies to
$$ \left(m\beta ^2 - (6s+1) m\alpha^2 - 6n \right)\beta = 0, $$
so $\Gamma_{v,s}$ consists of exactly three components: $\Gamma^{\pm}_{v,s}$ are the two branches of the hyperbola 
$\beta ^2 -(6s+1)\alpha^2=6n/m$, and $\Gamma^0_{v,s}$ coincides with the vertical axis $\{\beta = 0\}$.
\end{example}

In the limit $s\to0$, also assuming  that the Bogomolov--Gieseker inequality $v_1^2\geq2v_0v_2$ holds for $v$, the curves $\Gamma^{\pm}_{v,s=0}$ are co-asymptotic with the two branches of the hyperbola $\Theta_v$.

It turns out that the position of the different branches of the cubic curve $\Gamma_{v,s}$ relative to the two branches of the hyperbola $\Theta_v$ has an important geometric meaning, which is unravelled in the following two statements.

\begin{proposition}
Let $v$ be a real numerical Chern character with $v_0\ne0$  satisfying the Bogomolov--Gieseker inequality \eqref{B-ineq}.
If, for each $(\alpha,\beta)\in\HH$, there is an object $E\in\cohab$ with $\ch(E)=v$, then $\Gamma_{v,s}$ consists of exactly three connected components:
\[ \Gamma^-_{v,s}\sqcup\Gamma^0_{v,s}\sqcup\Gamma^+_{v,s} \]
and
\begin{equation}\label{newbog}
q(E):=3v_1^2v_2^2-6v_1^3v_3+18v_0v_1v_2v_3-8v_0v_2^3-9v_0^2v_3^2\geq0.
\end{equation}
\end{proposition}

\begin{proof}
The hypotheses imply that the three asymptotic branches $\Gamma^{\bullet}_{v,s}$ remain in their respective regions $R^{\bullet}_v$ for $\bullet=+,-,0$ all the way down to the horizontal axis $\{\alpha=0\}$, and so we have three distinct irreducible components, as given. 

Recall that a cubic equation admits three distinct real roots exactly when its discriminant is positive. Regarding equation \eqref{eq:alphagamma} as a cubic in $\beta$, an easy computation gives its discriminant as ($108$ times)
\begin{equation}
\begin{split}
q_\alpha(E):=&\; \alpha ^6 v_0^4 (6 s+1)^3+9 \alpha ^4 v_0^2(6 s+1)^2 \left(v_1^2-2 v_0v_1\right)\\
&+27 \alpha^2  (6 s+1) \left(v_1^2-2v_0v_1\right)^2\\  &+27\left(3v_1^2v_2^2-6v_1^3v_3+18v_0v_1v_2v_3-8v_0v_2^3-9v_0^2v_3^2\right).
\end{split}
\end{equation}
Setting $\alpha=0$ yields \eqref{newbog}. Finally, we need to know that $q_\alpha$ can only change sign once as $\alpha$ varies. To see that, observe that
\[\frac{d q_\alpha(E)}{d\alpha^2}=3(6s+1)\bigl(v_0^2(6s+1)\alpha^2+3(v_1^2-2v_0v_2)\bigr)^2>0.\]
So as a cubic in $\alpha^2$, $q_\alpha$ has a single point of inflection and so can change sign at most once. Note that, since $v_1^2\geq 2v_0v_1$, the point of inflection does not occur for any $\alpha>0$. 
\end{proof}
We also make use of the following. 

\begin{definition}\label{RL}
Suppose $v_0\neq0$ and $Q^{\rm tilt}(v)\geq0$. We define the region $\leftgamma_{v,s}$ to be
\[\{(\alpha,\beta)\in\HH\mid \alpha<\alpha_0\text{ whenever } (\alpha_0,\beta)\in\Gamma^-_{v,s}\}\cap R^-_v.\]
Alternatively,
\[ \leftgamma_{v,s} = \{(\alpha,\beta)\in\HH\mid v_0\tau_{v,s}(\alpha,\beta)\geq0\}\cap R^-_{v}. \]
\end{definition}

\begin{figure}[ht]
  \centering
\includegraphics{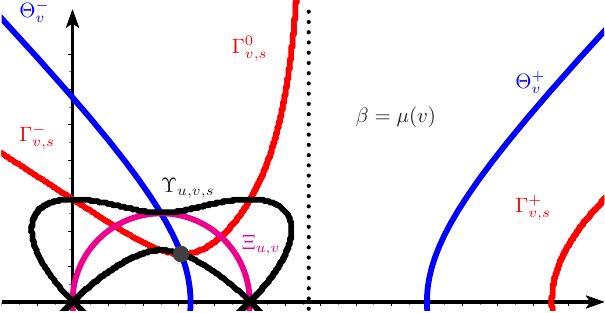}
\caption{This graph contains an example of the situation described in Proposition~\ref{left-cap} for $v=(3,4,2,2/3)$, which is the Chern character of the tangent bundle of $\p3$. We set $s=1/3$.
The curve $\Gamma_{v,s}$ (in red) intersects the left branch of the hyperbola $\Theta_v$ (in blue) at the point $(\alpha,\beta)\approx(0.27,0.62)$. In addition, $\Gamma_{v,s}^0$ is asymptotic to the vertical line $\Mu_v$, while the component $\Gamma_{v,s}^+$ does not intersect $\Theta_v^+$. We also illustrate Theorem~\ref{thm for q} here, with the vanishing $\nu$- and $\lambda$-walls represented by the curves $\Xi_{u,v}$ (in magenta) and $\Upsilon_{u,v,s}$ (in black), respectively, where $u=\ch(\op3(1))$. Finally, Lemma~\ref{l-l} is also represented since both of these walls cross $\Theta^-_v$ at the same point. }
 \label{line-fig-1}
\end{figure}

Note that if $q(v)<0$, then $\Gamma_{v,s}$ intersects $\Theta_{v}$. This is illustrated in Figure~\ref{line-fig-1}. On the other hand, away from $\Theta_v$ and $\Mu_v$, the curve $\Gamma^-_{v,1/3}$ is strictly monotonic. This is because $\partial_\beta\ch^{\alpha,\beta}_3(v)=-\ch^{\alpha,\beta}_2(v)$. It follows that if $q(v)\geq0$, we have $\Gamma_{v,s}\cap\Theta_v=\varnothing$, and otherwise this intersection is a single point at which $\Gamma_{v,s}$ has a minimum turning point, as illustrated in Figure~\ref{line-fig-1}.

\begin{proposition}\label{left-cap}
Let $v$ be a real numerical Chern character with $v_0\ne0$ satisfying the Bogomolov--Gieseker inequality \eqref{B-ineq}.
If $q(v)<0$, then $\Gamma_{v,s}$ intersects the left branch of $\,\Theta_v$ if and only if
\[ v_0^2v_3>v_0 v_1 v_2-\frac{1}{3}v_1^3. \]
\end{proposition}

\begin{proof}
This is easiest to see by considering the sign of $\labs(v)$ for very small $\alpha$ (we can just take $\alpha=0$). Observe that for $\beta\to-\infty$, the sign is positive, and the sign changes when we cross either $\Gamma_{v,s}$ or $\Theta_v$. If $q(v)<0$, then $\Gamma_{v,s}$ crosses the $\beta$ axis exactly once. Since $\rho_v(\alpha,\beta)>0$ for $\beta\to-\infty$, it follows that $\tau_{v,s}(\alpha,\beta)>0$  to the left of this point. Observe that $\lambda_{0,\beta,s}(v)=0$ for $\beta>\mu(v)$ if and only if $\Gamma_{v,s}$ intersects the left branch of $\Theta_v$.  Write $\tau_{v,s}(0,\beta)$ as
\[6v_3-\frac{6 v_1 v_2}{v_0}+\frac{2 v_1^3}{v_0^2}+\left(\beta -\frac{v_1}{v_0}\right) \left(\frac{3 v_1^2}{v_0}-6 v_2\right)-v_0 \left(\beta -\frac{v_1}{v_0}\right)^3.\]
Then the numerical condition is the positivity of the constant term in this cubic in $\beta-v_1/v_0$.
\end{proof}

The following is an easy exercise.

\begin{proposition}
The quartic form $q$ in equation \eqref{newbog} is invariant under $E\mapsto E\otimes\calo_X(k)$ and $E\mapsto E^\vee$.
\end{proposition}

\section{Walls}\label{sec:walls}

\subsection{$\boldsymbol{\nu}$-walls}\label{sec:nu-walls}

Given a real numerical Chern character $v$, a curve $\Xi_{u,v}\subset\R^+\times\R$ is called a \emph{numerical $\nu$-wall} for $v$ if there is a real numerical Chern character $u$ such that
$$ \Xi_{u,v} := \{t_{u,v}(\alpha,\beta)=0\},\quad{\rm where} $$
\begin{equation}\label{defdelta} 
t_{u,v}(\alpha,\beta) := \rho_{u}(\alpha,\beta)\ch_1^\beta(v) - \rho_v(\alpha,\beta)\ch_1^\beta(u) = \Delta_{21}(\alpha,\beta);
\end{equation}
compare with the notation introduced in display \eqref{Delta}. Note that this coincides with the numerator of the difference $\nuab(u)-\nuab(v)$, so its vanishing locus is precisely where it changes sign. Observe that numerical $\nu$-walls are invariant as subsets of the upper half plane under the changes of coordinates $u\mapsto\kappa v+u$, $u\mapsto \zeta u$ and $(u,v)\mapsto (v,u)$ for any real constants $\kappa$ and $\zeta\neq0$.

\begin{lemma}\label{empty_tilt}
Suppose  $v_0>0$,  $v$ satisfies the Bogomolov--Gieseker inequality and $u=(0,1,x,y)$ for real $x$ and $y$. Then the following are equivalent: 
\begin{enumerate}
\item\label{empty_tilt1} $\Xi_{u,v}$ is empty.
\item\label{empty_tilt2} For all real $\kappa$, $\kappa v+u$ satisfies the Bogomolov--Gieseker inequality.
\item\label{empty_tilt3} $(x-\mu(v))^2 \leq \mu(v)^2 - 2v_2/v_0$. 
\item \label{eq:no_tilt} $\delta_{02}(u,v)^2\leq 4\delta_{01}(u,v)\delta_{12}(u,v)$.
\end{enumerate}
\end{lemma}

\begin{proof}
The Bogomolov--Gieseker inequality for $\kappa v+u$ as a polynomial in $\kappa$ is
\begin{equation}\label{bogkvu}
\kappa^2Q^{\mathrm{tilt}}(v)+2\kappa(v_1-v_0x)+1.
\end{equation}
The condition in~\eqref{empty_tilt3} is the discriminant condition for \eqref{bogkvu} divided by $v_0^2$. This shows that~\eqref{empty_tilt2} and~\eqref{empty_tilt3} are equivalent. The condition that $\Xi_{u,v}=\varnothing$ is that $\Delta_{21}(0,\beta)=0$ has at most one solution. Explicitly, this is
\begin{equation}\label{tilt_kappa}
\beta^2 v_0/2-\beta v_0x+v_1x-v_2, 
\end{equation}
and the discriminant condition for this is $v_0^2x^2\leq2v_0(v_1x-v_2)$, which is equivalent to the condition in~\eqref{empty_tilt3}.  Finally, $\delta_{02}(u,v)^2-4\delta_{01}(u,v)\delta_{12}(u,v)=v_0^2x^2-2v_0(v_1x-v_2)$ establishes the equivalence of~\eqref{eq:no_tilt}.
\end{proof}

The structure of numerical $\nu$-walls was described in \cite{M} for the case of projective surfaces, but the same results also hold for projective threefolds; see for instance \cite[Theorem 3.3]{S} and \cite[Remark~9.2]{MS}. Precisely, numerical $\nu$-walls are non-intersecting semicircles centered along the horizontal axis and cross the hyperbola $\Theta_v$ at their maximum point or the vertical line $\Mu_v$. The semicircles tend to a fixed point $(0,C_0)$ on the $\beta$-axis as the radius tends to zero.

\begin{definition}
A numerical $\nu$-wall $\Xi_{u,v}$ is called an \emph{actual $\nu$-wall} if for some $(\alpha_0,\beta_0)\in\Xi_{u,v}$, there are an object $B\in\mathcal{B}^{\beta_0}$ with $\ch(B)=v$ and a $\nu_{\alpha_0,\beta_0}$-semistable sub-object $F\into B$ with $\ch(F)=u$ such that the quotient $E/F$ is also $\nu_{\alpha_0,\beta_0}$-semistable.
\end{definition}

It then follows that the same property holds for every point of $\Xi_{u,v}$; see \cite[Theorem 3.3]{S}. Actual $\nu$-walls are \emph{locally finite} (that is, any compact subset of the upper half plane intersects only finitely many numerical $\nu$-walls), see \cite[Lemma 6.23]{MS}, and there are $C_{\rm max}\in\R$ and $R_{\rm max}>0$ such that every numerical $\nu$-wall is contained in the semicircle centered at $(0,C_{\rm max})$ with radius $R_{\rm max}$. In particular, there is an $\oalpha>0$ for which there are only finitely many actual $\nu$-walls above the horizontal line $\{\alpha=\oalpha\}$. However, actual $\nu$-walls may accumulate towards the point $\Theta_v\cap\{\alpha=0\}$; see \cite{Mea} and \cite{YY12} for examples on abelian surfaces.

Another observation is that, since $\nu$-walls are nested, there can be at most two vanishing $\nu$-walls for a given Chern character $v$, one on each side of the vertical line $\Mu_v$. 

The following technical result, which will be useful in Section 8, follows immediately from Lemma~\ref{empty_tilt}.
For a Chern character $v$, we let 
\[Y_v=\{\gamma\mid \gamma=\delta_{02}(u,v)/\delta_{01}(u,v)\text{ for some actual $\nu$-wall given by $u$}\}.\]

\begin{proposition}\label{no_tilt_condition}
Suppose $\ch(E)=v$ and $\ch(F)$ are two numerical Chern characters which satisfy the Bogomolov--Gieseker inequality, $v_0>0$ and $\delta_{01}(u,v)\neq0$.
If $F\into E\onto G$ is a short exact sequence in $\cohb$ for $\beta<\mu(E)$, then if $\delta_{02}(u,v)/\delta_{01}(u,v)\not\in Y_v$, we have 
\[\delta_{01}(u,v)\delta_{12}(u,v)\geq0\quad\text{and}\quad
|\delta_{02}(u,v)|\leq 2\sqrt{\delta_{01}(u,v)\delta_{12}(u,v)}.\]
\end{proposition}

For example, if a particular Chern character $v$ admits no actual $\nu$-walls, then any Chern character of a sub-object in $\cohb$ must satisfy those inequalities.

\subsection{$\boldsymbol{\lambda}$-walls}\label{lambdawalls}

Given in a real numerical Chern character $v$, a curve $\Upsilon_{u,v,s}\subset\R^+\times\R$ is called a \emph{numerical $\lambda$-wall} for $v$ if there is a real numerical Chern character $u$ such that
$$ \Upsilon_{u,v,s} := \{ f_{u,v,s}(\alpha,\beta)=0 \},\quad{\rm where} $$
\begin{equation} \label{fuvs}
f_{u,v,s}(\alpha,\beta):=\tau_{u,s}(\alpha,\beta)\rho_v(\alpha,\beta)-\tau_{v,s}(\alpha,\beta)\rho_{u}(\alpha,\beta).
\end{equation}
Note that this coincides with the numerator of the difference $\labs(u)-\labs(v)$, so its vanishing locus is precisely where it changes sign. Comparing with the notation introduced in display \eqref{Delta}, we remark that
$$ f_{u,v,s}(\alpha,\beta)=\Delta_{32}-\alpha^2(s-1/3)\Delta_{21}; $$
thus, in particular, $f_{u,v,1/3}(\alpha,\beta)=\Delta_{32}(\alpha,\beta)$.

More explicitly, using the notation from display \eqref{Delta}
$$ \delta_{ij}=\delta_{ij}(u,v)=\ch_i(u)\ch_j(v)-\ch_j(u)\ch_i(v), $$
we have
\begin{equation}\label{num-walls}
\begin{split}
f_{u,v,s}(\alpha,\beta) = & \;\dfrac{6s+1}{12}\delta_{10}\alpha^4  \\
&+ \left( \dfrac{3s-1}{6}\delta_{10}\beta^2 - \dfrac{3s-1}{3}\delta_{20}\beta + \dfrac{6s+1}{6}\delta_{21} - \dfrac{1}{2}\delta_{30}\right)\alpha^2  \\
&+ \left( \dfrac{1}{12}\delta_{10}\beta^4-\dfrac{1}{3}\delta_{20}\beta^3+\dfrac{1}{2}(\delta_{30}+\delta_{21})\beta^2-\delta_{31}\beta+\delta_{32} \right).
\end{split}
\end{equation}

Numerical $\lambda$-walls that are bounded as subsets of $\HH$ have a simple numerical characterization.

\begin{proposition}\label{bounded walls 1}
Let $v$ be a real numerical Chern character with $v_0\ne0$ and satisfying the Bogomolov--Gieseker inequality. A numerical $\lambda$-wall $\Upsilon_{u,v,s}$ for $v$ is bounded if and only if $\delta_{10}(u,v)\ne0$.
\end{proposition}

Notice that if $\Upsilon_{u,v,s}$ is bounded for some value of the parameter $s$, then it is bounded for every value of $s$.

\begin{proof}
Assuming $\delta_{10}(u,v)\ne0$, turn $f_{u,v,s}(\alpha,\beta)$ into a homogeneous polynomial
of degree 4 by adding a new variable $\gamma$; let us denote this new function by
$\overline{f_{u,v,s}}(\alpha,\beta,\gamma)$. The set
$\{\overline{f_{u,v,s}}(\alpha,\beta,\gamma)=0\}$ can be regarded as a hypersurface in $\R\p2$,
and $\Upsilon_{u,v,s}$ is bounded if and only if
$\{\overline{f_{u,v,s}}(\alpha,\beta,\gamma)=0\}$ does not intersect the line at infinity
$\{\gamma=0\}$. 

Comparing with equation \eqref{num-walls}, note that
$$ \overline{f_{u,v,s}}(\alpha,\beta,0)=
\dfrac{\delta_{10}}{12}\left( 6s\alpha^2(\alpha^2+\beta^2)+(\alpha^2-\beta^2)^2\right). $$
Since the right-hand side vanishes if and only if $\alpha=\beta=0$, we conclude that $\{\overline{f_{u,v,s}}(\alpha,\beta,\gamma)=0\}$ does not intersect the line at infinity $\{\gamma=0\}$; hence $\Upsilon_{u,v,s}$ is bounded.

Next, if $\delta_{10}(u,v)=0$ and $\delta_{20}(u,v)\ne0$, then the homogenisation of $f_{u,v,s}(\alpha,\beta)$ yields a polynomial $\overline{f_{u,v,s}}(\alpha,\beta,\gamma)$ of degree 3. Comparing with equation \eqref{num-walls}, note that
\begin{equation}\label{eq:f_unbounded}
\overline{f_{u,v,s}}(\alpha,\beta,0)=
-\dfrac{\delta_{20}}{3}\left( (3s-1)\alpha^2 + \beta^2\right)\beta, 
\end{equation}
so $\overline{f_{u,v,s}}(1,0,0)=0$; thus $\Upsilon_{u,v,s}$ is not bounded.

Finally, if $\delta_{10}(u,v)=\delta_{20}(u,v)=0$, then $f_{u,v,s}(\alpha,\beta)$ becomes a constant multiple of $\rho_u(\alpha,\beta)=\rho_v(\alpha,\beta)$; hence the curve $\Upsilon_{u,v,s}$ coincides with the hyperbola $\Theta_v$, which is not bounded.
\end{proof}

Next, we observe an interesting relation between a numerical $\lambda$-wall $\Upsilon_{u,v,s}$ and its associated numerical $\nu$-wall $\Xi_{u,v}$: if one intersects the hyperbola $\Theta_v$, then so does the other.

\begin{lemma}\label{l-l}
Let $v$ be a real numerical Chern character satisfying the Bogomolov--Gieseker inequality and $v_0\ne0$. If a numerical $\lambda$-wall $\,\Upsilon_{u,v,s}$ intersects $\Theta_v$ at a point $(\alpha_0,\beta_0)$ for which an object $A$ with $\ch(A)=v$ exists in $\mathcal{A}^{\alpha_0,\beta_0}$, then the associated numerical $\nu$-wall $\,\Xi_{u,v}$ also passes through $(\alpha_0,\beta_0)$.

Conversely, if a numerical $\nu$-wall $\,\Xi_{u,v}$ intersects $\Theta_v$ at a point $(\alpha_0,\beta_0)$, then the associated numerical $\lambda$-wall $\Upsilon_{u,v,s}$ also passes through $(\alpha_0,\beta_0)$.
\end{lemma}

\begin{proof}
Since $\rho_v(\alpha_0,\beta_0)=0$, we have $\rho_u(\alpha_0,\beta_0)\tau_{v,s}(\alpha_0,\beta_0)=0$. However, Proposition~\ref{proprank2} guarantees that $\tau_{v,s}(\alpha_0,\beta_0)>0$ for every $s$, thus $\rho_u(\alpha_0,\beta_0)=0$. It follows that $t_{u,v}(\alpha_0,\beta_0)=0$, as desired.

For the second statement, note that $\Theta_v$ never intersects $\Mu_v$ because $v$ is assumed to satisfy the Bogomolov--Gieseker inequality and so $\ch_1^{\beta_0}(v)\neq0$.  So, if $t_{u,v}(\alpha_0,\beta_0)=\rho_v(\alpha_0,\beta_0)=0$, then also $\rho_u(\alpha_0,\beta_0)=0$, thus $f_{u,v}(\alpha_0,\beta_0)=0$.
\end{proof}

\begin{lemma}\label{nodal point}
If the curves $\Theta_v$ and $\Gamma_{v,s}$ intersect, then every numerical $\lambda$-wall for $v$ passes through the point of intersection.
\end{lemma}

\begin{proof}
Let $(\alpha_0,\beta_0)$ be the point of intersection, so that $\rho_v(\alpha_0,\beta_0)=\tau_{v,s}(\alpha_0,\beta_0)=0$. By the expression in display \eqref{fuvs}, it follows that $f_{u,v,s}(\alpha_0,\beta_0)=0$ for every real numerical Chern character $u$ and every $s>0$.
\end{proof}

Let us now analyse unbounded numerical $\lambda$-walls. Note that if $\delta_{10}(u,v)=0$, then the expression in equation \eqref{num-walls} reduces to a cubic polynomial in $\beta$, with coefficients depending on $\alpha^2$, so it always has a real root; in other words, unbounded numerical $\lambda$-walls intersect every horizontal line. Furthermore, unbounded numerical $\lambda$-walls with different values of the parameter $s$ can look very different; see Figure~\ref{vert walls fig}. In particular, an unbounded numerical $\lambda$-wall might not be connected, and one of its connected components may be bounded; see Figure~\ref{un-bounded}. The main properties of unbounded walls are described in the following series of lemmas. First, we consider the case $s>1/3$.

\begin{figure}[ht] \centering
\includegraphics{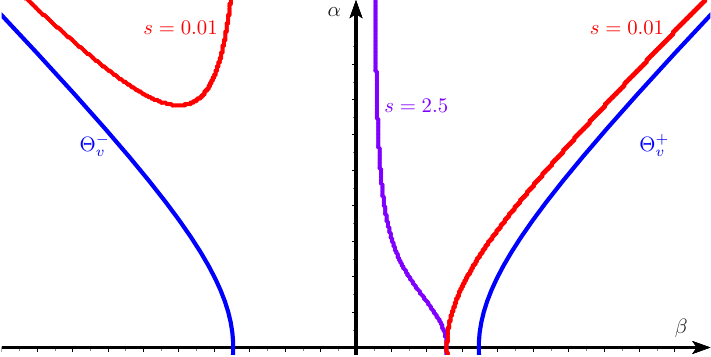}
\caption{The numerical $\lambda$-wall $\Upsilon_{u,v,s}$ for $v=(2,0,-3,0)$ and $u=(0,0,-1,1)$ is plotted for $s=0.01$ (red curve), where the wall has two separate connected components, and for $s=2.5$ (purple curve), where the wall has a single connected component. Both curves are asymptotic to the vertical line $\Mu_v=\{\beta=\mu(v)=0\}$, while the red curve is also asymptotic to both branches of $\Theta_v$ (in blue).} \label{vert walls fig}
\end{figure}

\begin{lemma} \label{bounded walls 2}
Let $s>1/3$. If $u$ and $v$ are real numerical Chern characters satisfying $\delta_{10}(u,v)=0$ and $\delta_{20}(u,v)\ne0$, then there exists a $\beta_{\rm max}>0$ $($depending on $u$, $v$ and $s)$ such that $\Upsilon_{u,v,s}\subset \R^+\times[-\beta_{\rm max},\beta_{\rm max}]$.
\end{lemma}

\begin{proof}
If $\delta_{10}=0$, then equation \eqref{num-walls} reduces, after dividing by $\delta_{20}$, to
\begin{equation}\label{nw-alpha2}
\underbrace{\left( (s-1/3)\beta+\cdots \right)}_{a_2}\alpha^2 + \underbrace{\left( \beta^3/3 +\cdots \right)}_{a_0} = 0
\end{equation}
If $s>1/3$ and either $\beta\ll0$ or $\beta\gg0$, then both coefficients $a_2$ and $a_0$ in equation \eqref{nw-alpha2} have the same sign; hence equation \eqref{nw-alpha2} does not admit any solutions. 
\end{proof}

\begin{lemma}\label{verticalwalls}
Let $s\ne1/3$. If $u$ and $v$ are real numerical Chern characters satisfying $\delta_{10}(u,v)=0$ and $\delta_{20}(u,v)\ne0$, then one of the connected components of $\,\Upsilon_{u,v,s}$ is asymptotic to the vertical line $\{\beta=\beta'\}$, where
$$ \beta' := \dfrac{3}{(6s-2)\delta_{20}(u,v)}\left( \dfrac{6s+1}{3}\delta_{21}(u,v) - \delta_{30}(u,v) \right). $$
\end{lemma}

\begin{proof}
If $s\ne1/3$, then the equation $f_{u,v,s}(\alpha,\beta)=0$ can be rewritten in the following way:
$$ \alpha^2 = \dfrac{\dfrac{1}{3}\delta_{20}\beta^3-\dfrac{1}{2}(\delta_{30}+\delta_{21})\beta^2+\delta_{31}\beta-\delta_{32}}{-\dfrac{3s-1}{3}\delta_{20}\beta + \dfrac{6s+1}{6}\delta_{21} - \dfrac{1}{2}\delta_{30}}. $$
Note that the denominator vanishes at the value $\beta=\beta'$ given above; therefore, $\alpha$ goes to infinity as $\beta$ approaches~$\beta'$. 
\end{proof}

\begin{lemma} \label{bounded walls 3}
Let $s\le1/3$. If $u$ and $v$ are numerical Chern characters satisfying $\delta_{10}(u,v)=0$ and $\delta_{20}(u,v)\ne0$, then there exists an $\alpha_{\rm max}>0$ $($depending on $u$, $v$ and $s)$ such that
$$ \Upsilon_{u,v,s}\cap\{(\alpha,\beta)\in\HH ~|~ \alpha>\alpha_{\rm max}\}\subset \{(\alpha,\beta)\in\HH ~|~ \rho_v(\alpha,\beta)<0\}. $$
\end{lemma}

\begin{proof}
If $s<1/3$, then equation \eqref{nw-alpha2} does admit solutions, which are asymptotically of the form
$$ \alpha\sqrt{1/3-s}=\pm(\beta-\beta') $$
for some $\beta'\in\mathbb{R}$ given explicitly in Lemma~\ref{verticalwalls} above.
In other words, there is an $\oalpha>0$ such that we have $\Upsilon_{u,v,s}\cap\{(\alpha,\beta)\in\HH ~|~ \alpha>\oalpha\}\subset R^0_v$.

If $s=1/3$ and $\delta_{21}-\delta_{30}\ne0$, then the equation $f_{u,v,s}(\alpha,\beta)=0$ reduces to
$$ \alpha^2 + \left(\dfrac{2}{3}\dfrac{\delta_{20}}{\delta_{21}-\delta_{30}}\beta^3 +\cdots\right)=0. $$
Depending on the sign of $\delta_{20}/(\delta_{21}-\delta_{30})$, such an equation does admit solutions either for $\beta\ll0$ or for $\beta\gg0$; in both cases, $\alpha$ grows like $|\beta|^{3/2}$, so $\Upsilon_{u,v,s}$ will lie in the region $R^0_v$ once $\alpha$ is sufficiently large. The same conclusion holds when $\delta_{30}=\delta_{21}$ since then $f_{u,v,s}(\alpha,\beta)$ becomes just a cubic polynomial on $\beta$, not depending on $\alpha$. 
\end{proof}

\begin{remark}
Note that there are no $\nu$-walls corresponding to unbounded $\lambda$-walls, and so by continuity the unbounded component remains in $R^0_{v}$ as it cannot intersect $\Theta_v$.

However, an unbounded $\lambda$-wall for $v$ may have a bounded connected component contained in $R^-_v$, as pictured in Figure~\ref{un-bounded}.
\end{remark}

\begin{figure}[ht] \centering
\includegraphics{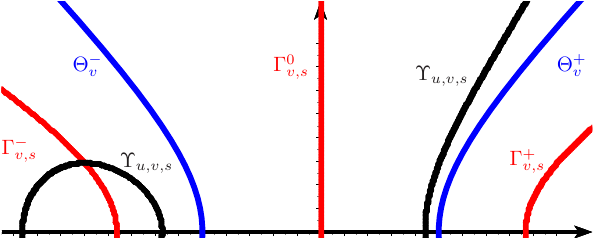}
\caption{The numerical $\lambda$-wall $\Upsilon_{u,v,s}$ for $v=(2,0,-1,0)$ and $u=(1,0,-1,1)$, so that $\delta_{01}(u,v)=0$, is plotted for $s=1/3$ (black curve). It has two connected components:
one is bounded and fully contained in $R_v^-$; the other is unbounded and fully contained in $R_v^{0+}$. The bounded component crosses $\Gamma_{v,s}^-$. The curve $\Theta_v$ and the other connected components of $\Gamma_{v,s}$, namely $\Gamma_{v,s}^0$ and $\Gamma_{v,s}^+$, are also shown, with the former coinciding with the $\alpha$-axis.} \label{un-bounded}
\end{figure}

We now show that intersection between numerical $\lambda$-walls and the curve $\Gamma_{v,s}$ is
independent of $s$. 

\begin{lemma}\label{indep of s}
If a numerical $\lambda$-wall $\Upsilon_{u,v,\overline{s}}$ crosses $\Gamma_{v,\overline{s}}$ for some $\overline{s}\ge0$ away from $\Theta_v$, then $\Upsilon_{u,v,s}$ crosses $\Gamma_{v,s}$ for every $s\ge0$ at points whose $\beta$-coordinate does not depend on $s$.
\end{lemma}

\begin{proof}
  Assume  $(\alpha_0,\beta_{0})\in\HH$ satisfies $f_{u,v,\overline{s}}(\alpha_0,\beta_{0})=\tau_{v,\overline{s}}(\alpha_0,\beta_{0})=0$ and  $\rho_{v}(\alpha_0,\beta_{0})\ne0$; in other words, $(\alpha_0,\beta_{0})\in \Upsilon_{u,v,\overline{s}}\cap\Gamma_{v,\overline{s}}$ lies away from the hyperbola $\Theta_v$. It follows that $\tau_{u,\overline{s}}(\alpha_0,\beta_{0})=0$; \textit{i.e.}
  the curve $\Gamma_{u,s}$ also passes through $(\alpha_0,\beta_0)$, so that
$$ \ch_3^{\beta_0}(u) =
\left( \overline{s} + \dfrac{1}{6} \right) \ch_1^{\beta_0}(u) \alpha_0^2. $$
Since $\tau_{v,\overline{s}}(\alpha_0,\beta_0)=0$ as well, we conclude that
\begin{equation}\label{eqn for b_0}
\dfrac{\ch_3^{\beta_0}(u)}{\ch_1^{\beta_0}(u)} = \dfrac{\ch_3^{\beta_0}(v)}{\ch_1^{\beta_0}(v)} = \left(\overline{s}+\dfrac{1}{6}\right)\alpha_0^2. 
\end{equation}
For any $s\ge0$, there is an $\alpha_s$ such that $(s+1/6)\alpha_s^2=(\overline{s}+1/6)\alpha_0^2$; substituting back into equation \eqref{eqn for b_0}, we conclude that $\tau_{u,s}(\alpha_s,\beta_0)=\tau_{v,s}(\alpha_s,\beta_0)=0$, thus $f_{u,v,s}(\alpha_s,\beta_0)=0$, meaning that $(\alpha_s,\beta_0)\in\Upsilon_{u,v,s}\cap\Gamma_{v,s}\cap\Gamma_{u,s}$,
which is independent of the parameter $s\ge0$.
\end{proof}

Given $v\in K_{\rm num}(X)$, we define an equivalence relation in $K_{\rm num}(X)$ as follows:
\begin{equation}\label{v-eqv}
u\sim_v u'\text{ if there exist real numbers }\psi\neq0\text{ and }\phi\text{ such that }u'=\phi v+\psi u.    
\end{equation}
Note that $u\sim_v u'$ implies $\delta_{ij}(u',v)=\psi\delta_{ij}(u,v)$ and thus $\Upsilon_{u,v,s}=\Upsilon_{u',v,s}$. In other words, the set of numerical $\lambda$-walls for a fixed $v\in K_{\rm num}(X)$ only depend on the equivalence class under $\sim_v$.

\begin{lemma}
Suppose $u$ and $v$ are real numerical Chern characters and $v_0\ne0$. Then:
\begin{enumerate}
\item 
$u\sim_v(0,1,\delta_{02}(u,v)/\delta_{01}(u,v),\delta_{03}(u,v)/\delta_{01}(u,v))$ if $\,\delta_{01}(u,v)\neq0$ and
\item $u\sim_v(0,0,1,\delta_{03}(u,v)/\delta_{02}(u,v))$ if $\delta_{01}(u,v)=0$ and $\delta_{02}(u,v)\neq0$.
\end{enumerate}
\end{lemma}

\begin{proof}
If $\delta_{01}(u,v)\ne0$, take $\phi=u_0/\delta_{01}(u,v)$ and $\psi=-v_0/\delta_{01}(u,v)$. If $\delta_{01}(u,v)=0$, take $\phi=u_0/\delta_{02}(u,v)$ and $\psi=-v_0/\delta_{02}(u,v)$.
\end{proof}

It follows that equivalence classes for $\sim_v$ come in three types when $v_0\ne0$, according to Table~\ref{table-eqv} below ($x$ and $y$ are arbitrary rational numbers). In the degenerate case, the numerical $\lambda$-wall $\Upsilon_{u,v,s}$ coincides with $\Theta_v$.

\begin{table}[ht]
\begin{tabular}{||c|c|c||}
\hline
\textbf{Type} & \textbf{Characterization}                       & \textbf{Canonical element} \\ \hline
Bounded       & $\delta_{01}\neq0$                              & (0,1,x,y)                  \\ \hline
Unbounded     & $\delta_{01}=0$, $\delta_{02}\neq0$             & (0,0,1,x)                  \\ \hline
Degenerate    & $\delta_{01}=\delta_{02}=0$, $\delta_{03}\neq0$ & (0,0,0,1)                  \\ \hline
\end{tabular}
\medskip
\caption{Equivalence classes $\sim_v$ when $v_0\ne0$.} \label{table-eqv}
\end{table}

This gives us a quick way to see that various families of numerical $\lambda$-walls do not intersect each other, and it provides a form of Bertram's nested wall theorem.

\begin{theorem}\label{inter walls}
Suppose $v_0\neq0$ and $u\not\sim_v u'$.
\begin{enumerate}
\item\label{inter walls1} The numerical $\nu$-walls corresponding to $u$ and $u'$ do not intersect away from $\Mu_v$.
\item\label{inter walls2} If $\delta_{01}(u,v)=0=\delta_{01}(u',v)$, then the numerical $\lambda$-walls corresponding to $u$ and $u'$ do not intersect. 
\item\label{inter walls3} If $\delta_{01}(u,v)\neq0$ and $\ch_{\le2}(u)=\ch_{\le2}(u')$, then the numerical $\lambda$-walls for $v$ corresponding to $u$ and $u'$ only intersect on $\Theta_v$.
\end{enumerate}
\end{theorem}

\begin{proof}
We prove~\eqref{inter walls3} and leave~\eqref{inter walls1} and \eqref{inter walls2} as similar exercises. We simply observe that the equation for the numerical $\lambda$-wall $\Upsilon_{u,v,s}$ can be written in the following manner: 
\[ u_3\rho_v(\alpha,\beta) = g(\alpha,\beta,v,\ch_{\le2}(u)) \]
for some function $g$. Hence, if $(\alpha,\beta)$ is not on $\Theta_v$ but lies on $\Upsilon_{u,v,s}\cap\Upsilon_{u',v,s}$, then we must have $u_3=\ch_3(u')$.
\end{proof}

We conclude that distinct unbounded numerical $\lambda$-walls never intersect one another. On the other hand, it is interesting to observe that, in contrast with numerical $\nu$-walls, two distinct bounded numerical $\lambda$-walls for $\Upsilon_{u,v,s}$ and $\Upsilon_{u',v,s}$ can intersect both along $\Theta_v$ (if $\ch_{\le2}(u)=\ch_{\le2}(u')$, as illustrated in item~\eqref{inter walls3} of Theorem~\ref{inter walls}), and away from it; see Figure~\ref{cross num walls}. 

\begin{figure}[ht] \centering
\includegraphics{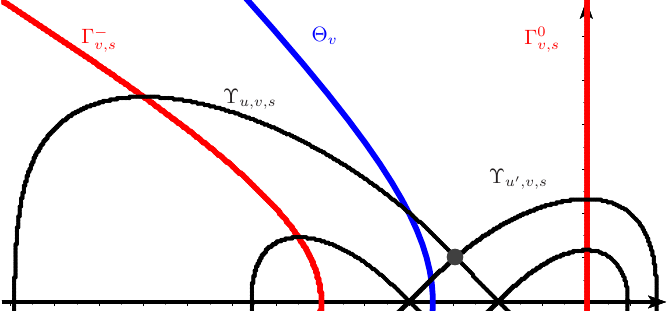}
\caption{The numerical $\lambda$-walls $\Upsilon_{u,v,s}$ and $\Upsilon_{u',v,s}$ for $v=(2,0,-3,0)$ with $u=(1,-1,1/2,-1/6)$ and $u'=(1,-2,2,-4/3)$ are plotted for $s=1/3$; they cross at the point $(\alpha=0.5,\beta=-1.48)$. The curves $\Gamma_{v,1/3}$ (in red) and $\Theta_v$ (in blue) are also shown.} \label{cross num walls}
\end{figure}

Our normal form for $u$ can also be used to give sufficient conditions for when a numerical $\lambda$-wall for $v$ intersects $\Gamma_{v,s}$.

\begin{proposition}
Let $u$ and $v$ be real numerical Chern characters such that $v$ satisfies the Bogomolov--Gieseker inequality, $\ch_0(v)\ne0$ and $\delta_{01}(u,v)\neq0$. If $\delta_{01}(u,v)\delta_{03}(u,v)\ge\delta_{02}(u,v)^2/2$, then the numerical $\lambda$-wall $\Upsilon_{u,v,s}$ intersects $\Gamma_{v,s}^0$. If $\delta_{01}(u,v)\delta_{03}(u,v)\le\delta^2_{02}/2$, then $\Upsilon_{u,v,s}$ intersects both $\Gamma_{v,s}^{+}$ and $\Gamma_{v,s}^-$.
\end{proposition}

\begin{proof}
Writing $u$ in the canonical form $(0,1,x,y)$, observe that $\Gamma_{u,s}$ is given by the hyperbola
\[(\beta-x)^2/2-\alpha^2(s+1/6)=x^2/2-y.\]
The conditions in the statement about $\delta_{ij}$ are equivalent to the hyperbola crossing the $\beta$-axis or not. The hyperbolae are asymptotic to $\beta=x\pm\sqrt{2(s+1/6)}\alpha$, while $\Gamma_{v,s}$ is asymptotic either to $\Mu_v$ or to $\beta=\mu\pm\sqrt{6(s+1/6)}\alpha$.
Then the hypotheses guarantee that $\Gamma_{u,s}$ and $\Gamma_{v,s}$ intersect as stated. 
\end{proof}

\begin{remark}
Observe that an unbounded numerical $\lambda$-wall $\Upsilon_{u,v,s}$ associated to $u\sim_v(0,0,1,x)$ intersects $\Gamma_{v,s}$ (necessarily away from $\Theta_v$ since unbounded $\lambda$-walls for $v$ never cross $\Theta_v$) at a point with coordinates $(\alpha,x)$ for some uniquely determined $\alpha$ so long as $x$ is the $\beta$-coordinate of some point on $\Gamma_{v,s}$. Since the wall $\Upsilon_{u,v,s}$ can only cross $\Theta_v$ at its intersection with $\Gamma_{v,s}$ and its asymptotes are in $R^0_v$, it follows that, for $|x|$ sufficiently large, there is a component of the wall which is bounded and which intersects $\Gamma^{\pm}_{v,s}$.
\end{remark}

\subsection{Actual, pseudo and vanishing $\boldsymbol{\lambda}$-walls}

\begin{definition}
An \emph{actual $\lambda$-wall} $W_{u,v,s}$ is the subset of points $(\alpha,\beta)$ in $\Upsilon_{u,v,s}$ for which there are an object $E\in\cohab$ with $\ch(E)=v$ and a path $\gamma\colon (-1,1)\to \R_{>0}\times \R$ such that $\gamma(0)=(\alpha,\beta)$ and $E\in\cala^{\gamma(t)}$ is $\lambda$-stable for $t<0$ and $\lambda$-unstable for $t>0$. 
\end{definition}

This is a rather stronger condition than the usual one, which asks simply that $A$ be properly $\labs$-semistable at the point $(\alpha,\beta)$. Our definition ensures that actual $\lambda$-walls are $1$-dimensional and avoids degenerate situations like the one described in Remark~\ref{notbog}. In addition, it allows us to state the following. 

\begin{lemma}\label{lem:actual wall}
An actual $\lambda$-wall $W_{u,v,s}$ is a union of segments of arcs within the underlying numerical $\lambda$-wall $\Upsilon_{u,v,s}$ whose endpoints lie on another actual $\lambda$-wall for $u$ or $v$, or on $\Theta_u\cap\Theta_v$, or on the $\beta$-axis.
\end{lemma}

\begin{proof}
Note that the wall will have an endpoint if $\alpha\to0$ along the curve; we take that to be an endpoint by convention.\footnote{In fact, we will see in Section 6 than the wall crosses $\alpha=0$ transversely or is singular there.}

First assume $p\in W_{u,v,s}$ is an isolated point; that is, there exists an open neighbourhood $B$ of $p$ on which there are no other actual $\lambda$-walls. In contrast with the path $\gamma$ going through $p$ along which the moduli space $\calm_{\gamma(t),s}(v)$ changes as explained above, one can find a path $\gamma'\colon [0,1]\to B$ such that $\gamma'(0)=\gamma(0)$ and $\gamma'(1)=\gamma(1)$ which does not cross any actual $\lambda$-wall, implying that $\calm_{\gamma(0),s}(v)=\calm_{\gamma(1),s}(v)$, and thus providing a contradiction.

It follows that $W_{u,v,s}$ is a union of segments of arcs within $\Upsilon_{u,v,s}$. The same argument shows that if an endpoint $(\alpha_0,\beta_0)$ of such an arc does not lie on another actual $\lambda$-wall for a short exact sequence $0\to A\to E\to B\to0$ in $\cohab$  corresponding to the wall (so $\ch(E)=v$ and either $\ch(A)=u$ or $\ch(B)=u$), then $E$, $A$ or $B$ must go out of the category. But if $A$ or $B$ go out of the category, it is because they have a sub-object or quotient object which has a $\Theta$-curve going through that point. But then the $\lambda$-slope goes to
infinity, and this would contradict the existence of the wall in a small neighbourhood of the
point. Consequently, the wall can only end on $\Theta_{\ch(E)}$ and so also on $\Theta_{\ch(A)}$.
\end{proof}

In another direction, if an object is strictly destabilized along a curve, then it can only become stable if  the curve crosses either an actual $\lambda$-wall or a $\Theta$-curve. Note that if an actual $\lambda$-wall $W_{u,v}$ crosses $\Theta_u$, then it must also cross $\Theta_v$.

\begin{proposition}\label{theta_or_wall}
Fix $s>0$.
Suppose $\gamma\colon (-a,b)\to\HH$ is a curve for some $a,b>0$ and $A\in\mathcal{A}^{\gamma(t)}$ for all $-a<t<b$. Suppose $B\into A\onto C$ is a destabilizing sequence for $-a<t<0$.
\begin{enumerate}
\item If $A$ is not $\lambda_{\gamma(t),s}$-semistable for $0<t<b$, then there exist $\epsilon,\epsilon'>0$ and some destabilizing sequence $K\into A\onto Q$ defined for all $-\epsilon'<t<\epsilon$ with $\lambda_{\gamma(t),s}(K)\geq\lambda_{\gamma(t),s}(B)$, respectively $\lambda_{\gamma(t),s}(Q)\leq\lambda_{\gamma(t),s}(C)$, with equalities if and only if $C=Q$, respectively $B=K$.
\item If $A$ is $\lambda_{\gamma(t),s}$-stable for $0<t<b$, then there is an actual $\lambda$-wall corresponding to $B$ containing $\gamma(0)$.
\end{enumerate}
\end{proposition}

\begin{proof}
In either case, we have a long exact sequence for $0<t<\epsilon$ in $\mathcal{A}^{\gamma(t)}$:
\[0\to C^{-1}\to B^0\to A\to C^0\to B^1\to 0.\]
This is because as we cross $t=0$, the phase changes continuously and the  $\mathcal{A}^{\gamma(t)}$-cohomology of any object in $\mathcal{A}^{\gamma(t)}$ for $-\epsilon'<t<0$ in nearby $\gamma(t)$ is concentrated in positions $-1$, $0$ and/or $1$. Split the sequence via $B^0\to K\to A\to Q\to C^0$, and note that, for $-1\ll t<0$, $C^{-1}[1],B^{1}[-1]\in\mathcal{A}^{\gamma(t)}$ and $\lim_{t\nearrow 0}\lambda_{\gamma(t),s}(C^{-1}[1])=+\infty$ and $\lim_{t\nearrow0}\lambda_{\gamma(t),s}(B^{1}[-1])=-\infty$. We also have short exact sequences for $-1\ll t<0$:
\[C^{-1}[1]\to C\to C^0 \text{ and } B^0\to B\to B^{1}[-1].\]
It follows that $K\to A\to Q$ is short exact in $\mathcal{A}^{\gamma(t)}$ for $-\epsilon<t<\epsilon$ and that, for $-1\ll t<0$, $\lambda_{\gamma(t),s}(K)\geq\lambda_{\gamma(t),s}(B)$, with equality only if $C^{-1}=0$. Similarly, $\lambda_{\gamma(t),s}(Q)\leq\lambda_{\gamma(t),s}(C)$, with equality only if $B^1=0$. 

If $A$ is $\lambda_{\gamma(t),s}$-stable beyond $t=0$, then we must have equality in one of these. If both are equalities, then the sequence $0\to B\to A\to C\to0$ provides an actual $\lambda$-wall at $\gamma(0)$. If only one is an equality, then $\gamma(0)\in\Theta_{u}$, where $u=\ch(B)$ or $u=\ch(C)$. But then, $\lambda_{\gamma(0),s}(A)=+\infty$ and so $\gamma(0)\in\Theta_{\ch(A)}$ as well, and so, again, this sequence provides an actual $\lambda$-wall.
\end{proof}

If our aim is to find all actual $\lambda$-walls, then it is easier to first consider a list of necessary numerical conditions in order to reduce the possibilities to a small list of examples (this is exactly what we will do in the example presented in Proposition~\ref{gamma+inst} below).

\begin{definition}
By a \emph{pseudo $\lambda$-wall} $W_{u,v,s}$, we mean the subset of points $(\alpha,\beta)$ of a numerical $\lambda$-wall $\Upsilon_{u,v,s}$ for which:  
\begin{enumerate}
\item\label{pseudowall1} there are objects $E,F,G\in\cohab$ satisfying $\ch(E)=v$, $\ch(F)=u$, $\ch(G)=v-u$;
\item\label{pseudowall2} $Q_{\alpha,\beta}(u)\geq0$, $Q_{\alpha,\beta}(v-u)\geq0$, $Q_{\alpha,\beta}(u)+Q_{\alpha,\beta}(v-u)\leq Q_{\alpha,\beta}(v)$.
\end{enumerate}
\end{definition}

The support property, see Proposition~\ref{support_property}, implies that an actual
$\lambda$-wall is also a pseudo $\lambda$-wall. We shall see an example in \eqref{weird_wall} in Proposition~\ref{gamma+inst} of a pseudo-wall which is not an actual wall but for which
there is a destabilizing sequence.

\begin{remark}
In practice, we would start by replacing \eqref{pseudowall1} with the necessary conditions that $u\in\Z\times\Z\times\Z/2\times\Z/6$ and $\chi(u)\in\Z$. The third inequality in~\eqref{pseudowall2} is an observation of Schmidt (\cite[Lemma 2.7]{S}) and follows from the fact that the bilinear form satisfies $Q_{\alpha,\beta}(v,v)=Q_{\alpha,\beta}(v)\geq0$.
\end{remark}

Unlike actual $\nu$-walls, distinct actual $\lambda$-walls for a given Chern character $v$ do intersect. This was first noticed by Schmidt in \cite{S}, where he establishes a refinement of Lemma~\ref{l-l} providing a relationship between actual $\nu$- and $\lambda$-walls for $v$ along $\Theta_v^-$; more precisely, he proves the following statement.

\begin{theorem}[\emph{cf.} \protect{\cite[Theorem~6.1]{S}}]
\label{schmidt}
Let $v$ be a Chern character satisfying the Bogomolov--Gieseker inequality and $v_0\ne0$, and let $(\alpha_0,\beta_0)$ be a point $\Theta_v^-$.
\begin{itemize}
\item If there is an actual $\lambda$-wall for $v$ containing $(\alpha_0,\beta_0)$ given by the exact sequence $0\to F\to A \to G\to 0$ in $\mathcal{A}^{\alpha_0,\beta_0}$, then there is an actual $\nu$-wall containing $(\alpha_0,\beta_0)$ given by the exact sequence $0\to F[-1]\to A[-1] \to G[-1]\to 0$ in $\mathcal{B}^{\beta_0}$.
\item Conversely, if there is an actual $\nu$-wall for $v$ containing $(\alpha_0,\beta_0)$ given by the sequence $0\to F\to A \to G\to 0$ in $\mathcal{B}^{\beta_0}$, then there is an actual $\lambda$-wall for $v$ which is defined by the same sequence in $\mathcal{A}^{\alpha_0,\beta_0}$.
\end{itemize}
\end{theorem}

We also provide a concrete example of two distinct actual $\lambda$-walls for the Chern character $v:=(2,0,-1,0)$ which intersect away from $\Theta_v$, in Section~\ref{sec:instantons} below.

Furthermore, an actual $\lambda$-wall $W_{u,v,s}$ is called a \emph{vanishing $\lambda$-wall} for $v$ if for each $p\in W_{u,v,s}$ there exists a path $\gamma\colon (-1,1)\to\HH$ crossing $W_{u,v,s}$ transversely at $\gamma(0)=p$ for some $w\in(0,1)$ such that $\calm_{\gamma(t),s}(v)\ne\emptyset$ for $t<0$, while $\calm_{\gamma(t),s}(v)=\emptyset$ and $t>0$.

The meaning of the quartic function $q(v)$ on $K_{\rm num}(X)$ defined in equation \eqref{newbog} can now be expressed in the following theorem using the geometry of the $\Gamma$-curves we defined in Section 3. Recall that $q(E)$ is the discriminant of the cubic defining $\Gamma$ whose roots are the $\beta$-coordinates of  $\Gamma_{v,s}\cap\{\alpha=0\}$.

\begin{theorem}\label{thm for q}
Let $v$ be a Chern character for which there exists a Gieseker semistable sheaf $\,E$ with $\ch(E)=v$. If $q(v)<0$, then there must exist vanishing $\nu$- and $\lambda$-walls for $v$.
\end{theorem}

\begin{proof}
Since $E$ is Gieseker semistable, $v$ satisfies the Bogomolov--Gieseker inequality, and so the curve $\Theta_v$ divides the plane into three regions, as explained in Section~\ref{regions} above. 

Note that $\Gamma_{v,s}$ is a smooth cubic curve, and so when $q(E)<0$,  there is an open neighbourhood of $\alpha=0$ such that also $q_\alpha(E)<0$. It follows that there is some $\alpha_0>0$ such that $\Gamma_{v,s}\cap\{\alpha=\alpha_0\}$ consists of a single point. Thus, either the asymptotic components $\Gamma^-_{v,s}$ and $\Gamma^0_{v,s}$, or $\Gamma^+_{v,s}$ and $\Gamma^0_{v,s}$, must belong to the same connected component of $\Gamma_{v,s}$, which must then cross the hyperbola $\Theta_v$.

Since Proposition~\ref{proprank2} fails for $(\alpha,\beta)\in\Theta_v$ below  the point of intersection $(\tilde{\alpha},\tilde{\beta}):=\Gamma_{v,s}\cap\Theta_v$, there cannot exist any objects with Chern character equal to $v$ in $\cohab$ for $(\alpha,\beta)\in\Theta_v$ with $\alpha\le\tilde{\alpha}$.

Fix $\beta$ and $\alpha$ sufficiently large so that $(\alpha,\beta)$  lies above any $\nu$-wall for $v$. Then $E$ is $\nuab$-semistable, and Proposition~\ref{nuzero} guarantees the existence of $\labs$-semistable objects in $\cohab$ for $(\alpha,\beta)\in\HH$. It follows that there must exist a vanishing $\nu$-wall for $v$ crossing $\Theta_v$ above the point of intersection $\{(\tilde{\alpha},\tilde{\beta})\}:=\Gamma_{v,s}\cap\Theta_v$. But then there is also a vanishing $\lambda$-wall through the same point $(\tilde{\alpha},\tilde{\beta})$.
\end{proof}

We will see an example of the situation described in Theorem~\ref{thm for q} in Section~\ref{ideal line} below, where we study in detail the ideal sheaf of a line in $\p3$.

\section{Asymptotic \texorpdfstring{$\boldsymbol{\nuab}$}{nu\textunderscore\{alpha,beta\}}-stability}\label{nuab sst}

One of the main goals of this paper is to characterize which objects in $\dbx$ are $\nuab$- and $\labs$-semistable \emph{at infinity}, that is, for large values of the parameters $\alpha$, $\beta$. More formally, let $\gamma\colon [0,\infty)\to\HH$ be an unbounded path; we consider the following definition. The dual situation is also considered in \cite[Proposition~3.2]{Piy} but from a different perspective. 

  {\samepage
\begin{definition}\label{asym nu s-st}
An object $A\in\dbx$ is \emph{asymptotically $\nuab$-$($semi$)$stable along $\gamma$} if the following two conditions hold: 
\begin{enumerate}[(i)]
\item There is a $t_0>0$ such that $A\in\mathcal{B}^{\gamma(t)}$ for every $t>t_0$.
\item There is a $t_1>t_0$ such that, for every $t>t_1$, every sub-object $F\into A$ in $\mathcal{B}^{\gamma(t)}$ satisfies $\nu_{\gamma(t)}(F) <$ $(\le)$ $\nu_{\gamma(t)}(A)$.
\end{enumerate}
\end{definition}
}

In this section we characterize asymptotically $\nuab$-semistable objects. More precisely, we establish the following results.

\begin{theorem} \label{a nu sst}
Let $B\in\dbx$ be an object with $\ch_0(B)\ne0$.
\begin{enumerate}
\item\label{a nu sst 1} $B$ is asymptotically $\nuab$-$($semi\/$)$stable along a path $\gamma(t)=(\alpha(t),\beta(t))$ such that $\beta(t)<\mu(B)$ for $t\gg0$ if and only if $B$ is a $\mu_{\le2}$-$($semi\/$)$stable sheaf.
\item\label{a nu sst 2} $B$ is asymptotically $\nuab$-$($semi\/$)$stable  along a path $\gamma(t)=(\alpha(t),\beta(t))$ such that $\beta(t)>\mu(B)$ for $t\gg0$ if and only if $S:=B^\vee[-1]$ is a $\mu_{\le2}$-$($semi\/$)$stable sheaf such that $S^{**}/S$ either is empty or has pure dimension~$1$.
\end{enumerate}
\end{theorem}

Theorem~\ref{a nu sst} is not new, though its statement and the second part, especially, are different from previous versions. The proof of item~\eqref{a nu sst 1} is essentially the same as \cite[Proposition 14.2]{B08}, which covers the case of $K3$ surfaces; also compare  with \cite[Proposition 2.5]{S18}. Part~\eqref{a nu sst 2}  is to be compared with \cite[Proposition~2.3]{AM} in the case of surfaces and \cite[Proposition 5.1.3]{BMT}. Since we will use Theorem~\ref{a nu sst} and the arguments in its proof in the subsequent sections, we include a full proof here. Our proof is longer than previous ones for two reasons. Firstly, we have a stronger form of asymptotic stability, and secondly we do not assume  there are only finitely many $\nu$-walls above a certain horizontal line. Although the latter fact is true, we want to illustrate how it can be avoided, since it is not known for $\lambda$-stability. If we do assume  the $\nu$-walls are bounded above, then it follows that the $t_1$ of the definition of asymptotic $\nu$-stability can be chosen uniformly in $E$, that is, $t_1$ only depends on $\ch_{\le2}(E)$. 

The study of asymptotic $\nuab$-stability is considerably simplified by the following observation. 

\begin{lemma}\label{asymp_B_is_sheaf}\quad
  
\begin{enumerate}
\item If $B\in\cohb$ for all $\beta\ll0$, then $B\in\coh(X)$.
\item\label{one_nu_wall} Suppose $E$ is a $\mu$-semistable sheaf, and suppose there is some $(\alpha,\beta)$ with $\beta<\mu(E)$ which lies on an actual $\nu$-wall so that there are a $\nuab$-semistable object $F$ and a monomorphism $F\into E$ in $\cohb$. Then this is the only $\nu$-wall for $\beta<\mu(E)$ at which $E$ is destabilized. 
\end{enumerate}
\end{lemma}

\begin{proof}
If $\calh^{-1}(B)\ne0$, then
$$ \ch_1^\beta(\calh^{-1}(B)[1]) = -\ch_1(\calh^{-1}(B)) + \ch_0(\calh^{-1}(B))\beta < 0 \quad\text{for } \beta\ll0, $$
so $B$ cannot lie in $\cohb$ for any $\beta\ll0$.

For the second item, the idea is that $E$ must be $\nuab$-stable above the $\nu$-wall defined by the exact sequence. The local finiteness of the actual $\nu$-walls ensures that $E$ is $\nuab$-stable either immediately below or immediately above this wall. But we can show that the latter holds, as follows. First observe that $\delta_{01}(E,F)\leq0$. This is because if we split the $\coh(X)$ sequence
\[0\to\calh^{-1}(G)\to F\to E\to\calh^0(G)\to0,  \]
where $G=E/F$ in $\cohb$, via $K\to E$, then $\mu(K)\leq\mu(E)$ by the $\mu$-semistability of $E$ while $\mu(\calh^{-1}(G))\leq\beta<\mu(F)$ and so $\mu(F)<\mu(K)\leq\mu(E)$.  We use Lemma~\ref{delta_tech}\eqref{delta_dervs} to compute the partial derivatives, which satisfy 
\[\partial_\alpha \bigl(\nuab(E)-\nuab(F)\bigr)=\partial_\alpha\frac{\Delta_{21}(E,F)}{\ch^\beta_1(E)\ch^\beta_1(F)}=-\frac{\alpha\delta_{01}(E,F)}{\ch_1^\beta(E)\ch_1^\beta(F)}\geq0\]
as $E$ is $\mu$-semistable. Consequently, $E$ cannot be destabilized on another wall outside of this $\nu$-wall because it must be stable immediately above any wall and unstable immediately below the wall. But on a path between two adjacent walls, it cannot be both stable and unstable. 

To complete the proof, observe that $F\into E\onto G$ remains a short exact sequence in $\cohb$ at all points inside the $\nu$-wall corresponding to $F$.
\end{proof}

We will see that a similar statement holds for $\beta>\mu(E)$. So, for a given $E$, there is at most one actual $\nu$-wall destabilizing it on each side of the vertical line $\Mu_v$. We can then reduce the proof of Theorem~\ref{a nu sst} to the study of asymptotic $\nuab$-stability along horizontal lines. To this end, we define
\begin{equation} \label{h-lines}
\begin{aligned}
\Lambda_{\oalpha}^-&:= \{ (\alpha,\beta)\in\HH ~|~ \alpha=\oalpha,\beta<0 \}
\quad\text{and}\quad \\
\Lambda_{\oalpha}^+&:= \{ (\alpha,\beta)\in\HH ~|~ \alpha=\oalpha,\beta>0 \}.
\end{aligned}
\end{equation}

\subsection{Asymptotics along $\boldsymbol{\Lambda_{\oalpha}^-}$}

The first part of Theorem~\ref{a nu sst} is proved in two separate lemmas.

\begin{lemma}\label{asymp_nu_is_mu2}
If $B\in\dbx$ is asymptotically $\nuab$-semistable along the horizontal half-line $\Lambda_{\oalpha}^-$, then $B$ is a $\mu_{\le2}$-semistable sheaf.   
\end{lemma}

\begin{proof}
By Lemma~\ref{asymp_B_is_sheaf}, we may assume $B\in\coh(X)\cap\cohb$ for all $\beta\ll0$.

Let $T\into B$ be a torsion subsheaf of $B$; if $\ch_1(T)\ne0$, then
$$ \lim_{\beta\to-\infty}\dfrac{-1}{\beta}(\nu_{\oalpha,\beta}(T)-\nu_{\oalpha,\beta}(B)) = \dfrac{1}{2}, $$
contradicting asymptotic $\nuab$-semistability. If $\ch_1(T)=0$, then
$$ \nuab(B)<\nuab(T)=+\infty $$ 
for every $(\alpha,\beta)\in\HH$, again contradicting asymptotic $\nuab$-semistability. So $B$ must be a torsion-free sheaf.

Let $F$ be a subsheaf of $B$; since $F\in\cohb$ for $\beta<\mu^-(F)$, $F$ is also a sub-object of $B$ within $\cohb$ for $\beta<\min\{\mu^-(F),\mu^-(B)\}$. Note that
$$ \lim_{\beta\to-\infty}\left(\nu_{\oalpha,\beta}(F)-\nu_{\oalpha,\beta}(B) \right) = -\dfrac{1}{2}\dfrac{\delta_{10}(B,F)}{\ch_0(F)\ch_0(B)}\le0 $$
by asymptotic $\nuab$-semistability. It follows that $\delta_{10}(B,F)\ge0$; thus $B$ is $\mu$-semistable.

If $\delta_{10}(B,F)=0$, then 
$$ \lim_{\beta\to-\infty}(-\beta)(\nu_{\oalpha,\beta}(F)-\nu_{\oalpha,\beta}(B)) = -\dfrac{\delta_{20}(B,F)}{\ch_0(F)\ch_0(B)}\le0 $$
by asymptotic $\nuab$-semistability; thus $\delta_{20}(B,F)\ge0$, as desired.
\end{proof}

Before going on to look at the converse, we can make an interesting deduction from this. 

\begin{proposition}
If $E$ is a $\mu$-semistable sheaf which is not $\mu_{\leq2}$-semistable, then $E$ is not $\nuab$-semistable for any $(\alpha,\beta)\in\HH$.
\end{proposition}

\begin{proof}
Suppose otherwise. Then by item~\eqref{one_nu_wall} in Lemma~\ref{asymp_B_is_sheaf}, it follows that $E$ is asymptotically $\nuab$-semistable, and then by Lemma~\ref{asymp_nu_is_mu2} it must be $\mu_{\leq2}$-semistable, which yields a contradiction.
\end{proof}

Now we consider the converse to Lemma~\ref{asymp_nu_is_mu2}. 

\begin{lemma}\label{h nu-ss}
Fix $\oalpha>0$. If $E$ is $\mu_{\le2}$-semistable, then there is a $\beta_0<0$ $($depending only on $\ch_{\le2}(E)$ and $\oalpha)$ such that $E$ is $\nu_{\oalpha,\beta}$-semistable for every $\beta<\beta_0$.   
\end{lemma}

\begin{proof}
First, note that $E$ is $\mu$-semistable; thus $E\in\tors\beta\subset\cohb$ whenever $\beta<\mu(E)$.
Then Lemma~\ref{asymp_B_is_sheaf}\eqref{one_nu_wall} implies that there is at most one actual $\nu$-wall for $E$ containing the point $(\oalpha,\beta_0)$  (set $\beta_0=\mu(E)$ if the wall does not exist). It follows that $E$ is $\nu_{\oalpha,\beta}$-stable for $\beta<\beta_0$.
\end{proof}

This completes the proof of the first part of Theorem~\ref{a nu sst}. 

\begin{remark}\label{h nu-ss rk0}
Recall the definition of $\hat\mu$-stability from Definition~\ref{muhat-stab}.
We will also need the following version of Lemma~\ref{h nu-ss} for torsion sheaves: given any fixed $\oalpha>0$, if $T\in\coh(X)_2$ is $\hat{\mu}$-semistable, then there is a $\beta_0<0$ such that $T$ is $\nu_{\oalpha,\beta}$-semistable for every $\beta<\beta_0$. The proof of this claim is similar to the proof of Lemma~\ref{h nu-ss}.
\end{remark}

We can now describe an \emph{asymptotic Harder--Narasimhan filtration} for unstable objects as well.

\begin{proposition}\label{asymp_HN}
Let $\gamma(t)=(\alpha(t),\beta(t))$ be a path satisfying $\lim_{t\to\infty}\beta(t)=-\infty$. If $B\in\dbx$ is an object for which there is a $t_0>0$ such that $B\in\calb^{\beta(t)}$ for all $t>t_0$, then $B$ admits a filtration in $\calb^{\beta(t)}$ 
$$ 0=B_0\subseteq B_1 \subset B_2 \subset \cdots \subset B_n=B $$
whose factors $G_k:=B_k/B_{k-1}$ are asymptotically $\nuab$-semistable and satisfy 
\begin{enumerate}
\item\label{asymp_HN1} $\nu_{\gamma(t)}(B_1)=+\infty$ for every $t>0$ if $B_1\ne0$;
\item\label{asymp_HN2} for each $k=2,\dots,n$, there is a $t_k>t_0$ such that $\nu_{\gamma(t)}(G_k)-\nu_{\gamma(t)}(G_{k+1})>0$ for every $t>t_k$.
\end{enumerate}
\end{proposition}

\begin{proof}
The first item of Lemma~\ref{asymp_B_is_sheaf} implies that $B$ must be a sheaf, so it admits a filtration as described in Lemma~\ref{hn-filtration}; this is the filtration we are looking for.

Indeed, each factor in the filtration of Lemma~\ref{hn-filtration} belongs to $\cohb$ for $\beta\ll0$, according to Lemma~\ref{h nu-ss} and Remark~\ref{h nu-ss rk0}. The property in item~\eqref{asymp_HN1}  is clear since $B_1\in\coh_1(X)$. Item~\eqref{asymp_HN2} is a consequence of the following claim: given $E,F\in\coh(X)$ with $\ch_{\le1}(E),\ch_{\le1}(F)\ne0$, we have $\nu_{\gamma(t)}(E)>\nu_{\gamma(t)}(F)$ for $t\gg0$ if and only if $\Lambda_{2}(E,F)>0$. This can be explicitly checked for the path $\gamma(t)=(\oalpha,-t)$ using the limits calculated in the proof of Lemma~\ref{asymp_nu_is_mu2}; the verification for more general paths is similar.
\end{proof}

\subsection{Unbounded $\boldsymbol{\Theta^*}$-curves}
It is tempting to think that if $B\in\cohb$ (respectively, $A\in\cohab$) for some $\alpha$, $\beta$, then $B^\vee\in\calb^{-\beta}$ (respectively, $A^\vee\in\mathcal{A}^{\alpha,-\beta}$). The problem is that the duals of objects $B$ in $\mathcal{F}_{\beta}$ are not necessarily in $\mathcal{T}_{-\beta}$ because it might be that $\mu^+(B)=\beta$ and then $\mu^-(B^\vee)=-\beta$. The analogous statement holds for objects $A$ in $\mathcal{F}_{\alpha,\beta}$. However, an \emph{asymptotic} version of this statement does hold. 

First we prove a technical lemma which allows us to turn Proposition~\ref{asymp_HN} about the existence of asymptotic Harder--Narasimhan filtrations into a more precise bound on $\nuab^-$ so long as we constrain the unbounded curve we move along. To this end, we introduce the following definition.

\begin{definition}\label{theta-curve}
A path $\gamma\colon (0,\infty)\to\HH$ with $\gamma(t)=(\alpha(t),\beta(t))$ is called \emph{an unbounded $\Theta^-$-curve} if $\lim_{t\to\infty}\beta(t)=-\infty$ and 
for all $v$ such that $v_{\leq1}\neq(0,0)$, $\nu_{\gamma(t)}(v)>0$ for all $t$ sufficiently large. Equivalently, the curve is asymptotically bounded by $\Theta^-_v$; \textit{i.e.}
\[\lim_{t\to\infty}\frac{\dot\alpha(t)}{\dot\beta(t)}>-1.\]
We define the \emph{dual curve} $\gamma^*$ by $\gamma^*(t)=(\alpha(t),-\beta(t))$.
Similarly, we define $\gamma$ to be \emph{an unbounded $\Theta^+$-curve} if $\lim_{t\to\infty}\beta(t)=+\infty$ and for all $v$ such that $v_{\leq1}\neq(0,0)$, $\nu_{\gamma(t)}(v)<0$ for all $t$ sufficiently large.  
\end{definition}

In particular, $\gamma$ is an unbounded $\Theta^-$-curve if and only if $\gamma^*$ is an unbounded $\Theta^+$-curve.

\begin{remark}
If $v_0=0$ while $v_1\neq0$, then the first condition implies the second; indeed, we have $\nu_{\gamma(t)}(v)=v_2/v_1-\beta(t)$, and the first condition implies that this is positive for all $t\gg0$. 

When  $v_0\neq0$, the condition on $\nu_{\gamma(t)}$ is satisfied, for example, if there is an $\epsilon>0$ such that for all $t\gg0$, $\alpha(t)^2<\beta(t)^2(1-\epsilon)$: we have
\[ \lim_{t\to\infty} \dfrac{\nu_{\gamma(t)}(v)}{-\beta(t)} > \lim_{t\to\infty} \frac{v_2-\beta(t) v_1+\beta(t)^2\epsilon v_0/2}{-v_1\beta(t)+v_0\beta(t)^2} = \dfrac{\epsilon}{2}>0. \]
So again this is positive for all $t\gg0$.
\end{remark}

\begin{example}\label{Gamma_is_gamma}
Note that, since $s>0$, $\alpha^2<\beta^2(1-\epsilon)$ sufficiently far along
$\Gamma^-_{v,s}$, and so $\Gamma^-_{v,s}$ is an unbounded $\Theta^-$-curve.
\end{example}

\begin{lemma}\label{l34}
Let $\gamma$ be an unbounded $\Theta^-$-curve. If $B\in\mathcal{B}^{\beta(t)}$  for all $t\gg0$ and $\ch_{\leq1}(B)\neq(0,0)$, then there exists a $t_0>0$ such that $\nu^-_{\gamma(t)}(B)>0$ for all $t>t_0$. 
\end{lemma}

\begin{proof}
Using Remark~\ref{asymp_HN}, there is a $t_0$ large enough so that there is a fixed Harder--Narasimhan factor of $B$, say $B_0\in\mathcal{B}^{\beta(t)}$, which satisfies $\nu_{\gamma(t)}(B_0)=\nu^-_{\gamma(t)}(B)$ for all $t>t_0$. If $\ch_0(B_0)=\ch_1(B_0)=0$, then $\nu_{\gamma(t)}(B_0)=\infty$ for all $t\gg0$,  and we are done. Otherwise, $\ch_{\leq1}(B_0)\neq(0,0)$. Then $\nu_{\gamma(t)}(B_0)$ is asymptotically positive by the assumption on $\gamma(t)$; thus $\nu_{\gamma(t)}^-(B)>0$ for $t\gg0$, as desired. 
\end{proof}

Combining this with Proposition~\ref{B-cohom-dual}, we deduce the following. 

\begin{proposition}\label{dual_in_A}
Consider an unbounded $\Theta^-$-curve $\gamma$ as above. Then the following are equivalent for an object $E\in D^b(X)$:
\begin{enumerate}
\item\label{dual_in_A1} $E\in\mathcal{A}^{\gamma(t)}$ for all $t\gg0$.
\item\label{dual_in_A2} $E\in\mathcal{B}^{\beta(t)}$ for all $t\gg0$.
\item\label{dual_in_A3} $E\in\coh(X)$.
\end{enumerate}
If $E\in\coh(X)$ contains no subsheaf of dimension $0$, then there is a $t_0>0$ such that for all $t>t_0$, $E^\vee\in\mathcal{A}^{\gamma^*(t)}$.
\end{proposition}

\begin{proof}
Lemma~\ref{l34} implies that~\eqref{dual_in_A2} implies~\eqref{dual_in_A1} and a similar a similar but simpler argument for $\mu^-(E)$ shows that~\eqref{dual_in_A3} implies~\eqref{dual_in_A2}. The idea for the converses is that if $E_1[1]\to E\to E_0$ is short exact in $\mathcal{A}^{\gamma(t_0)}$, then Lemma~\ref{l34} again shows that $\nu^-_{\gamma(t)}(E_1)$ becomes positive as $t$ increases, and then $E_1[1]\in\calb^{\beta(t)}$ for $t$ sufficiently large and so $E\in\calb^{\gamma(t)}$ for the same range of $t$. A similar argument shows~\eqref{dual_in_A2} implies~\eqref{dual_in_A3}.

Now assume $E$ is a coherent sheaf.
When $E$ is torsion-free, the claim follows immediately from Proposition~\ref{B-cohom-dual} and the observation that $\nu^+_{\gamma^*(t)}(E'^\vee)=-\nu^-_{\gamma(t)}(E')\to-\infty$ as $t\to\infty$. Otherwise, let $T\subset E$ be its maximal torsion subsheaf. Then the hypothesis implies that $T^\vee\in\cala^{\gamma^*(t)}$ for all $t\gg0$, and so dualizing $T\to E\to E/T$, we deduce the last part.
\end{proof}

\begin{corollary}\label{ses_in_A}
Suppose $\gamma$ is an unbounded $\Theta^-$-curve $\gamma$ and $E$ is a sheaf in $\mathcal{A}^{\gamma(t)}$ for all $t\geq0$.
A short exact sequence $\,0\to F\to E\to G\to 0$  in $\mathcal{A}^{\gamma(0)}\cap\coh(X)$  remains a short exact sequence in $\mathcal{A}^{\gamma(t)}$ for all $t\geq t_0$ for some $t_0\geq0$.
\end{corollary}

\begin{proof}
Observe that $0\to F\to E\to G\to 0$ is also a short exact sequence in $\mathcal{B}^{\beta(0)}$ by the first part of Proposition~\ref{sheaf_subobjects} and so also in $\mathcal{B}^{\beta(t)}$ for all $t>0$ since $\gamma$ is a $\Theta^-$-curve and so remains to the left of any $\Theta$-curve. But $\nu^-_{\gamma(t)}(G)\geq\nu^-_{\gamma(t)}(E)$. So $G\in\mathcal{A}^{\gamma(t)}$ for all $t>0$. The same is not necessarily true for $F$. But $\nu^-_{\gamma(t)}(F)>0$ for all $t\gg0$.
\end{proof}

\subsection{Asymptotics along $\boldsymbol{\Lambda_{\oalpha}^+}$}

We now move to the proof of the second part of Theorem~\ref{a nu sst}, starting with a characterization of objects lying in $\cohb$ for $\beta\gg0$.

{\samepage
\begin{lemma}\label{b+inf}
Let $B\in\dbx$. There is a $\beta_0>0$ such that $B\in\cohb$ for every $\beta>\beta_0$ if and only if the following conditions hold:
\begin{enumerate}
\item $\calh^p(B)=0$ for $p\ne-1,0$.
\item $\calh^{-1}(B)$ is a torsion-free sheaf.
\item $\calh^{0}(B)$ is a torsion sheaf.
\end{enumerate}
\end{lemma}
}

\begin{proof}
Set $E:=\calh^{-1}(B)$ and $P:=\calh^{0}(B)$. If $B\in\cohb$ for every $\beta\gg0$, then the first two items follow immediately. For the third one, just note that 
$$ \lim_{\beta\to\infty} \dfrac{\ch_1^\beta(P)}{\beta}=-\ch_0(P)\ge0; $$
thus $\ch_0(P)=0$.

Conversely, we have that $E\in\free\beta$ for $\beta\ge\mu^+(E)$ and $P\in\tors\beta$ for every $\beta$; it  follows immediately that $B\in\cohb$ for $\beta\ge\mu^+(E)$. 
\end{proof}

\begin{lemma}\label{nu h+:1}
If $B\in\dbx$ is asymptotically $\nuab$-semistable along $\Lambda_{\oalpha}^+$, then $S:=B^\vee[-1]$ is a $\mu_{\le2}$-semistable sheaf such that $S^{**}/S$ either is empty or has pure dimension $1$.
\end{lemma}

\begin{proof}
For simplicity, set $E:=\calh^{-1}(B)$ and $P:=\calh^{0}(B)$. We start by checking that $E$ is reflexive and $\dim P\le1$.

Indeed, if $E$ is not reflexive, then $Q_E:=E^{**}/E\in\tors\beta\subset\cohb$ for every $\beta$; thus $Q_E$ is a sub-object of $B$ within $\cohb$. But $\nuab(Q_E)=+\infty$ for every $(\alpha,\beta)\in\HH$, so we have a contradiction with the asymptotic $\nuab$-semistability of $B$. 

Lemma~\ref{b+inf} implies that $\ch_0(P)=0$. If $\ch_1(P)\ne0$, then
$$ \lim_{\beta\to\infty}\dfrac{1}{\beta}(\nu_{\oalpha,\beta}(B)-\nu_{\oalpha,\beta}(P)) = \dfrac{1}{2}, $$
also contradicting the asymptotic $\nuab$-semistability of $B$.

Any subsheaf $U$ of $P$ has dimension less than $1$, so $\nuab(U)=+\infty$, and hence $U$
cannot lift to a sub-object of $B$. By Proposition~\ref{a=dual}, we conclude that $S:=B^\vee[-1]$ is a torsion-free sheaf. Note that $\inext^3(S^{**}/S,\ox)\simeq\inext^2(S,\ox)=\calh^1(B)=0$, so $S^{**}/S$ cannot have a $0$-dimensional subsheaf.

Let $G$ be a quotient sheaf of $S$. Then $G^{*}$ is a subsheaf of $S^*=\calh^{-1}(B)$, and $G^*[1]\into B$ is a sub-object within $\cohb$ for every $\beta>\mu^+(G^*)$. Note that $\delta_{10}(G^*[1],B)=\delta_{10}(S,G)$, so
$$ \lim_{\beta\to\infty} (\nu_{\oalpha,\beta}(G^*[1])-\nu_{\oalpha,\beta}(B)) = \dfrac{1}{2}\dfrac{\delta_{10}(S,G)}{\ch_0(G)\ch_0(S)} \le 0 $$
by asymptotic $\nuab$-semistability, so $S$ is $\mu$-semistable.

Next, note that $\delta_{20}(G^*[1],B)=\delta_{20}(G,S)$; thus if $\delta_{10}(S,G)=0$, then
$$ \lim_{\beta\to\infty}(\beta)(\nu_{\oalpha,\beta}(G^*[1])-\nu_{\oalpha,\beta}(B)) = -\dfrac{\delta_{20}(G,S)}{\ch_0(G)\ch_0(S)}\le0, $$
implying that $\delta_{20}(S,G)\le0$, meaning that $S$ is $\mu_{\le2}$-semistable.
\end{proof}

Finally, we provide the converse of the previous lemma, thus concluding the proof of Theorem~\ref{a nu sst}. For $E\in\coh(X)$, note that Lemma~\ref{b+inf} implies that $E^\vee[-1]\in\cohb$ for $\beta\gg0$ if and only if $E$ has no subsheaf of dimension at most $1$, and the cokernel of the canonical morphism $E\to E^{**}$ has pure dimension $1$. 

\begin{lemma}\label{nu h+:2}
If $S$ is a $\mu_{\le2}$-semistable sheaf such that $S^{**}/S$ is either empty or has pure dimension $1$, then there is a $\beta_0>0$ such that $S^\vee[-1]$ is $\nu_{\oalpha,\beta}$-semistable for all $\beta>\beta_0$.
\end{lemma}

\begin{proof}
The observation in the previous paragraph implies that $B:=S^\vee[-1]\in\cohb$ for every $\beta\ge-\mu(S)$.

Let $F\into B$ be a sub-object within $\cohb$ for $\beta>-\mu(S)$; note that $\ch_0(\calh^0(F))=0$ by Lemma~\ref{b+inf}. First assume  $\ch_0(F)=\ch_0(\calh^{-1}(F))\ne0$. It follows that
$$ \mu(F) = \mu(\calh^{-1}(F)) - \dfrac{\ch_1(\calh^0(F))}{\ch_1(\calh^{-1}(F))} \le \mu(\calh^{-1}(F)) \le \mu(S^*) = \mu(B) $$
since $S^*$ is $\mu$-semistable and $\calh^{-1}(F)$ is a subsheaf of $S^*$. We then have 
$$ \lim_{\beta\to\infty} \left( \nu_{\oalpha,\beta}(F)-\nu_{\oalpha,\beta}(B) \right) = \frac{1}{2}\left( \mu(F)-\mu(B) \right)\le0. $$ 
If $\mu(F)=\mu(B)$, then $\mu(\calh^{-1}(F))=\mu(S^*)$ and $\ch_1(\calh^0(F))=0$. Denoting by $G:=B/F$ the corresponding quotient in $\cohb$, we have the following exact sequence in $\coh(X)$:
\begin{equation} \label{sqc-coh-a}
0 \to \calh^{-1}(F) \to S^* \to \calh^{-1}(G) \stackrel{f}{\to} \calh^0(F) \to \inext^1(S,\ox) \to \calh^0(G) \to 0. 
\end{equation}
Clearly, $\ker f$ is torsion-free, and this implies that $\calh^{-1}(F)$ is reflexive. Let $P\into\calh^0(F)$ be a monomorphism of sheaves that lifts to $F$, and let $U$ be the image of the composite morphism of sheaves $P\into\calh^0(F)\to\inext^1(S,\ox)$. Since $U\in\tors\beta\subset\cohb$ for every $\beta$, $U$ coincides with the image of the composite monomorphism $P\into F \into B=S^\vee[-1]$ in $\cohb$, so Proposition~\ref{a=dual} implies that $U=0$, so in fact $P=0$. We therefore conclude, again by Proposition~\ref{a=dual}, that $\overline{F}:=F^\vee[-1]$ is a torsion-free sheaf. Dualizing and shifting the triangle 
$$ F \to B \to G \to F[1] $$
in $\dbx$, we obtain the triangle
$$ \overline{G} \to S \to \overline{F} \to G^\vee,\quad \overline{G}:=G^\vee[-1], $$
which yields the  exact sequence in $\coh(X)$
$$ 0 \to \calh^0(\overline{G}) \to S \to \overline{F} \to \calh^1(\overline{G}) \to 0. $$
Note that $\ch_k(\calh^1(\overline{G}))=0$ for $k=0,1$; it follows that
$$ \delta_{20}(B,G) = \delta_{20}(S,\overline{G}) =
\delta_{20}(S,\calh^0(\overline{G}))+\ch_0(S)\ch_2(\calh^1(\overline{G})) \ge0. $$
We then obtain
$$ \lim_{\beta\to\infty} \beta\left( \nu_{\oalpha,\beta}(B)-\nu_{\oalpha,\beta}(G) \right) = -\dfrac{\delta_{20}(B,G)}{\ch_0(B)\ch_0(G)}\le0. $$
If  equality holds, then both $\ch_{\le2}(F)$ and $\ch_{\le2}(G)$ are multiples of $\ch_{\le2}(B)$, meaning that $\nuab(F)=\nuab(B)=\nuab(G)$ for every $(\alpha,\beta)$. In any case, we conclude that there is a $\beta_0>0$
such that $\nu_{\oalpha,\beta}(F)\le\nu_{\oalpha,\beta}(B)$ for $\beta>\beta_0$.

Next, assume  $\ch_0(F)=\ch_0(\calh^{-1}(F))=0$, so $\calh^{-1}(F)=0$ and $F\in\tors\beta$. If $\ch_1(F)\ne0$, then
$$ \lim_{\beta\to\infty}\dfrac{1}{\beta}(\nu_{\oalpha,\beta}(B)-\nu_{\oalpha,\beta}(F)) = \dfrac{1}{2}, $$
so $\nu_{\oalpha,\beta}(B)>\nu_{\oalpha,\beta}(F)$ for $\beta\gg0$. Therefore, it is enough to consider the case when $\dim F\le1$.

The sequence in display \eqref{sqc-coh-a} simplifies to
$$ 0 \to S^* \to \calh^{-1}(G) \stackrel{f}{\to} F \to \inext^1(S,\ox) \to \calh^0(G) \to 0. $$
Since $\ker f$ is a subsheaf of $F$, we conclude that $\dim\ker f\le1$, so that $(\ker f)^*=\inext^1(\ker f,\ox)=0$. Therefore,
dualizing the sequence
$$ 0\to S^* \to \calh^{-1}(G) \to \ker f\to 0, $$
we get that $S^{**} \simeq \calh^{-1}(G)^*$ and
\begin{equation} \label{new sqc}
\inext^1(S^*,\ox) \to \inext^2(\ker f,\ox) \to \inext^2(\calh^{-1}(G),\ox) \to 0
\end{equation}
since the reflexivity of $S^*$ implies that $\inext^2(S^*,\ox)=0$. In addition, since $\calh^{-1}(G)$ is torsion-free, we can also deduce that $\inext^3(\ker f,\ox)=0$, meaning that either $\ker f$ has pure dimension 1, or $\ker f=0$. However, the latter leads to a contradiction with the exact sequence in display \eqref{new sqc} since both $\inext^1(S^*,\ox)$ and $\inext^2(\calh^{-1}(G),\ox)$ are $0$-dimensional sheaves (since $S^*$ is reflexive and $\calh^{-1}(G)$ is torsion-free), while $\inext^2(\ker f,\ox)$ is a $1$-dimensional sheaf. It follows that $\ker f=0$; hence $F$ is a subsheaf of $\inext^1(S,\ox)$ that lifts to $B=S^\vee[-1]$, contradicting Proposition~\ref{a=dual}.
\end{proof}

\section{The differential geometry of surface walls}\label{sec:suf-walls}

\begin{definition}
Let $u$ and $v$ be real numerical Chern characters such that $v$ satisfies the Bogomolov--Gieseker inequality. We define the surface wall $\Sigma_{u,v}\subset \mathbb{R}^+\times\mathbb{R}\times\mathbb{R}^+$ to be the vanishing locus of
\[ f_{u,v}(\alpha,\beta,s)=\Delta_{32}(\alpha,\beta)-\alpha^2(s-1/3)\Delta_{12}(\alpha,\beta), \]
the numerator of the difference of slopes $\labs(u)-\labs(v)$ now regarded as a function of all three parameters $(\alpha,\beta,s)$. Note that 
\begin{equation}\label{sigmasym}
\Sigma_{u,v}=\Sigma_{v,u}=\Sigma_{\phi v+\psi u,v}
\end{equation}
for any real $\phi$ and $\psi\neq0$. Throughout this section, we will be referring to the function $f_{u,v}$ frequently, and it is will be more readable to abbreviate it when the context is clear to $f(\alpha,\beta,s)$ or just $f$.
\end{definition}

In addition, we will denote by $\Gamma_v$ (without the parameter $s$ in the subscript) the surface $\{\tau_{v,s}(\alpha,\beta)=0\}\subset \mathbb{R}\times\mathbb{R}^+$. In this notation, $\Gamma_{v,s_0}=\Gamma_v\cap\{s=s_0\}$ and $\Upsilon_{u,v,s_0}=\Sigma_{u,v}\cap\{s=s_0\}$ for any $s_0\in\R^+$.

Our aim in this section is to explore some of the differential-geometric properties of $\Sigma_{u,v}$ with a view to understanding finiteness properties of $\lambda$-walls. We will assume  $v$ is a fixed real numerical Chern character. Note that, by \eqref{sigmasym}, we can assume  $u_0=v_0$. It turns out to be best to consider the two cases $u_0=0=v_0$ and $u_0\neq0\neq v_0$. The former is dealt with in Remark~\ref{zerorankcase}.

Generically, $\Sigma_{u,v}$ can have components of dimensions $0$, $1$ or $2$. In fact, each of these does arise at least as a numerical surface wall. For example, for $v=(3,-2,-1/2,1)$ and $u=(0,1,-1/2,0)$,  
\[f_{u,v}(0,\beta,s)=(\beta+1)^2(\beta^2+2)/4\]
and $f_{u,v}(\alpha,\beta,s)\neq0$ whenever $\alpha\neq0$. At special points where the curves $\Gamma_{v,s}$ and $\Theta_{v}$ intersect, $\Sigma_{u,v}$ is guaranteed to be locally a surface.

\begin{lemma}\label{twoD}
At the points of intersection $\Gamma_{u,s}\cap \Gamma_{v,s}$, $\Theta_u\cap\Theta_v$ and $\Gamma_{v,s}\cap\Theta_{v}$, the surface wall $\Sigma_{u,v}$ is $2$-dimensional.
\end{lemma}

In other words, a surface wall $\Sigma_{u,v}$ does have $2$-dimensional components whenever it is not empty.

\begin{proof}
Note that each pair of curves divides a small ball around their intersection into four regions; otherwise, the (algebraic) curves must coincide. They cannot coincide except possibly for the Theta curves, but then $\Sigma_{u,v}=\Theta_u\times\mathbb{R}$. Then for a small arc around the intersection point in two of the regions, the function $f_{u,v}(\alpha,\beta,s)$ is positive on one curve and negative on the other, and so must vanish at some point in the region for each arc sufficiently close to the intersection point. Combining this with Lemma~\ref{indep of s}, we see that $\Sigma_{u,v}$ is $2$-dimensional in a neighbourhood of the point.
\end{proof}

We now look more carefully at the differential geometry of the surface wall $\Sigma_{u,v}$. The normal vector is given by the gradient of $f(\alpha,\beta,s)$.

\begin{lemma}\label{dervs}
\begin{align*}
\partial_\alpha f&=\alpha\bigl(1+2(s-1/3)\bigr)\Delta_{21}
-\alpha\Delta_{30}+\alpha^3(s-1/3)\Delta_{10}, \\
\partial_\beta f&=-\Delta_{31}-\alpha^2(s-1/3)\Delta_{20}, \\
\partial_s f&=\alpha^2\Delta_{21}, \\
\partial^2_\alpha f&=(1+2(s-1/3))\Delta_{21}-\Delta_{30}+\alpha^2(5s+1/3)\Delta_{10}, \\
Hf&=\begin{pmatrix} \partial^2_\alpha f&2\alpha(s-1/3)\Delta_{02}&2\alpha\Delta_{21}-\alpha^3\Delta_{01}, \\
2\alpha(s-1/3)\Delta_{02}&\Delta_{21}+\Delta_{30}+\alpha^2(s-1/3)\Delta_{10}&-\alpha^2\Delta_{20}, \\
2\alpha\Delta_{21}-\alpha^3\Delta_{01}&-\alpha^2\Delta_{20}&0
\end{pmatrix}.
\end{align*}
\end{lemma}

\begin{proposition}
For any $s_1>s_0>0$,
$\Sigma_{u,v}\cap\mathbb{R}\times\mathbb{R}\times\{s_0\leq s\leq s_1\}$ is compact if and only if $\Delta_{01}\neq0$.
\end{proposition}

\begin{proof}
This is just a restatement of Proposition~\ref{bounded walls 1} in terms of $\Sigma_{u,v}$.
\end{proof}

Note that when $s=0$, the surface is unbounded along $\alpha=\pm\beta$.

\begin{remark}\label{zerorankcase}
In the case where $v_0$ and $u_0$ are non-zero, if $u$ gives rise to a $\nu$-wall with respect to $v$, then $\Delta_{01}\neq0$. Then $\Delta_{21}(\alpha,\beta)=0$ has a $1$-dimensional solution set. So if $\Delta_{02}=0$ (everywhere), then, in particular, $\Delta_{20}=0$ along the $\nu$-wall, and it follows that $(u_0,u_1,u_2)\propto (v_0,v_1,v_2)$, which then cannot have a $\nu$-wall as $\Delta_{21}=0$ everywhere. On the other hand, if $\Delta_{01}=0$, then if $\Delta_{20}(\alpha,\beta)=0$ for some $(\alpha,\beta)$, it must vanish identically and again $(u_0,u_1,u_2)\propto(v_0,v_1,v_2)$. But then $\Delta_{21}=0$ identically. On a $\lambda$-wall, we then have $\Delta_{32}=0$ and so $u\propto v$. It follows that $\Delta_{02}$ cannot vanish identically.

If $v_0=0=u_0$, then $\Delta_{i0}=0$ identically for all $i$. Then there are no $\nu$-walls. The numerical $\lambda$-walls are nested ellipses, much as $\nu$-walls are for the truncated Chern characters. It follows that $\Sigma_{u,v}$ is always regular and horizontal exactly on $\Gamma_{v,s}$. The variation in $s$ is just a vertical scaling by $\sqrt{s+1/3}$.
\end{remark}

It is interesting to consider the regularity of numerical walls, and we will use this extensively in our analysis of the asymptotics. The regularity of a general wall for arbitrary $s$ is complicated and hard to describe, but the situation is simpler if we consider the whole surface $\Sigma_{u,v}$ and also for the case $s=1/3$, which we describe first.

\begin{proposition}\label{special_regular}
When $s=1/3$, a numerical wall $\Upsilon_{u,v,1/3}$ is regular everywhere in the upper half plane except where it intersects $\Gamma_{u,1/3}$ and its numerical $\nu$-wall $\Xi_{u,v}$. 
\end{proposition}

\begin{proof}
At a non-regular point $p$ on $\Upsilon_{u,v,1/3}$, we have $\Delta_{32}=0=\Delta_{31}$ and $\Delta_{30}=\Delta_{21}$. If $p\not\in\Gamma_{v,1/3}$, then Lemma~\ref{delta_tech}\eqref{twovanish} implies $\Delta_{12}=0$, and so $\Delta_{30}=0$. But then Lemma~\ref{delta_tech}\eqref{upropv} implies $u\propto v$. So $p\in\Gamma_{v,1/3}\cap\Gamma_{u,1/3}$. But then $\Delta_{30}=0$ and so $\Delta_{21}=0$. But then $p\in\Xi_{u,v}$. Conversely, if $p\in\Gamma_{v,1/3}\cap\Xi_{u,v}\cap\Upsilon_{u,v,1/3}$, then $\partial_\alpha f=0=\partial_{\beta}f$, and so $p$ is not a regular point.
\end{proof}

Looking at the second derivatives, we see that the local model for $\Upsilon_{u,v,1/3}$ at its singular point is $(\alpha-\alpha_0)^2+\text{higher order terms}=0$, and so the singular point is a cusp. We see this more generally in case~\eqref{reg_tilt_gamma} in Theorem~\ref{regular} below.

\begin{theorem}\label{regular}
Suppose $u$ and $v$ are real numerical Chern characters with $v_0\neq0$.
Any $2$-dimensional component of the surface wall $\Sigma_{u,v}$ is regular everywhere except in one of the following situations:
\begin{enumerate}
\item\label{regular1} It intersects a $\nu$-wall away from $\Gamma_v$ at $\alpha=0$, in which case it is locally $\alpha^2-\beta^2$.
\item\label{reg_tilt_gamma} It intersects a $\nu$-wall and $\Gamma_v$ at $s=1/3$ for $\alpha>0$, in which case it is locally $$(\alpha-\alpha_0)^2+(\alpha-\alpha_0)(s-1/3)+\dfrac{\Delta_{02}(\alpha_0,\beta_0)}{\Delta_{01}(\alpha_0,\beta_0)}(\beta-\beta_0)(s-1/3)=0$$ up to scaling. 
\item\label{reg_tilt_gamma_alpha} It intersects a $\nu$-wall and $\Gamma_v$ at  $\alpha=0$ and $s=1/3$,
  in which case the surface is smooth $($locally given by $(\beta-\beta_0)^3=0)$. 
\item\label{reducible} The characters $u$ and $v$ are special vectors satisfying $\mu(v)=\mu(u)$, $\Delta_{03}=\Delta_{21}$, $\Delta_{20}\neq0$ and $\ch^{\alpha,\mu(v)}_3(v)=0=\ch^{\alpha,\mu(u)}_3(u)$. Then $\Sigma_{u,v}$ is regular except along the line $\Mu_v\times\{s=1/3\}\cup\{\alpha=0,\ s<1/3\}$,  and it is locally $(\beta-\beta_0)\bigl((s_0-1/3)(\alpha-\alpha_0)+(\beta-\beta_0)^2\bigr)$ at $(\alpha_0,\mu(v),1/3)$ and at $(0,\beta_0,s_0)$.
\end{enumerate}
\end{theorem}

Note that in~\eqref{regular1}, we also allow the degenerate case where the $\nu$-wall has radius $0$. An example of this can be seen for $v=(2,0,-1,0)$, illustrated in Figure~\ref{fig (2,0,-1,0)} in Section~\ref{sec:examples}.

\begin{proof}
Suppose $(\alpha,\beta,s)$ is a non-regular point of $\Sigma_{u,v}$. Then $\nabla f_{u,v}(\alpha,\beta)=0$. From $\partial_s f_{u,v}=0$, we have either $\alpha=0$ or $\Delta_{12}(\alpha,\beta)=0$. We will treat these two cases separately.

First suppose $\alpha\neq0$. Then $\Delta_{12}=0$. From $f_{u,v}=0$, we also have $\Delta_{32}=0$. From $\partial_\alpha f_{u,v}=0$, we have $\Delta_{30}=\alpha^2(s-1/3)\Delta_{10}$, and from $\partial_\beta f_{u,v}=0$, we have $\Delta_{31}=-\alpha^2(s-1/3)\Delta_{20}$. Assume  $\ch_1^\beta(v)\neq0$ and $s\neq1/3$. Then Lemma~\ref{delta_tech}\eqref{twovanish} implies that $\Delta_{31}=0$, and then $\Delta_{20}=0$ so that also $\Delta_{10}=0$. Since $v_0\neq0$, Lemma~\ref{delta_tech}\eqref{upropv} implies that $u\propto v$, which gives a contradiction. So we must have $\ch_1^\beta(v)=0=\ch_1^\beta(u)$ or $s=1/3$. In the case $s\neq1/3$, it follows that $\Delta_{01}=0$. From $\partial_\beta f_{u,v}=0$, we have $\Delta_{02}=0$, but this contradicts $\Delta_{03}=0$. So we must  have $s=1/3$ and either (a) $\Delta_{01}=0=\Delta_{30}$ along $\ch^\beta_1(v)=0$ or (b) the $\lambda$-curve passes through both a $\nu$-wall and $\Gamma_{v,s}$.

We consider case (a). From $\Delta_{31}=0$, we then have that $u\propto v$ unless $\ch_3^{\alpha,\beta}(u)=0=\ch_3^{\alpha,\beta}(v)$ along $\Mu_v$. Then we have that the conditions are equivalent to
\begin{gather}
\Delta_{01}=0\label{v1}, \\
v_3v_0^2 - v_1v_2v_3 + v_1^3/3=0\label{v3}, \\
\Delta_{03}=\Delta_{12}\label{v2},\quad\text{and}\\
\Delta_{02}\neq0.
\end{gather}
The first equality for $u$ also follows (from the second).
Note that in this situation, $\Sigma_{u,v}$ is singular all along $\Mu_v$ and so also at $\alpha=0$. Along $\Mu_v$ the Hessian of $f$ vanishes. The only third derivatives to be non-zero (up to symmetry) are $\partial_\beta^3 f_{u,v}=2\Delta_{02}$ and $\partial_\alpha\partial_\beta\partial_s f_{u,v}=2\alpha\Delta_{02}$. So the local model for $f_{u,v}$ is 
\[2\Delta_{02}(\beta-\mu(v))\bigl((\beta-\mu)^2+(\alpha-\alpha_0)(s-1/3)\bigr).\]
This is a triple zero at $\alpha=0$ and looks like $y=0\cup y^2=xz$ for $\alpha\neq0$. In fact, since $f_{u,v}(\alpha,v_1/v_0)=0$, we have that $(\beta-v_1/v_0)$ is a linear factor of $f_{u,v}$.

Now suppose $\alpha=0$. Then $\partial_\beta f_{u,v}=0$ implies that $\Delta_{31}(\alpha,\beta)=0$. But from $f_{u,v}=0$, we have $\Delta_{32}(\alpha,\beta)=0$. If the point is not on $\Gamma$, then $\Delta_{21}(\alpha,\beta)=0$, and so $(\alpha,\beta)$ is on a $\nu$-wall. Then the Hessian at this point is the diagonal $\operatorname{diag}(-\Delta_{30}(0,\beta),\Delta_{30}(0,\beta),0)$. Then $f_{u,v}$ is locally of the form $\Delta_{03}(\alpha^2-(\beta-\beta_0)^2)$ to lowest order, as required. Conversely, such a point (where a numerical $\lambda$-wall intersects its associated $\nu$-wall away from $\Gamma_v$ on $\alpha=0$) is always singular of this form.

Alternatively, if the point lies on $\Gamma_{v,s}\cap \Gamma_{u,s}$ and $\alpha=0$, then the Hessian is equal to the diagonal  $\operatorname{diag}(\Delta_{21}(0,\beta)(s+1/3),\Delta_{21}(0,\beta),0)$. So $\Delta_{21}=0$. But this can only happen if either $\ch^\beta_1(v)=0=\ch^\beta_1(u)$ or we are also on a $\nu$-wall. In the latter case, we are in case (b); see below. Otherwise, $\Delta_{01}=0$. But then we either have the special situation described before, or $\Delta_{20}=0$ and $\Delta_{30}=0$, so we have $u\propto v$. This is impossible. In the special case above, we have $Hf_{u,v}=0$ again, and the same third derivatives  are non-zero as before.

In case (b) above, we have $\Delta_{21}=0$ and $\ch^{\alpha,\beta}_3(u)=0=\ch^{\alpha,\beta}_3(u)$, so that $\Delta_{3i}=0$ for all $i$. Since $\Sigma_{u,v}$
also crosses a $\nu$-wall, we must have $\Delta_{20}\neq0$. When $\alpha\neq0$, the local form in~\eqref{reg_tilt_gamma} can then be read off the Hessian. Note that at $s=1/3$, the wall has a standard cusp singularity given by $(\alpha-\alpha_0)^2\propto (\beta-\beta_0)^3$. Otherwise, the Hessian vanishes at the point. The third derivatives are then all zero except for $\partial^3_\beta f_{u,v}=-2\Delta_{02}$. This gives the required form in~\eqref{reg_tilt_gamma_alpha}.
\end{proof}

\begin{remark}
The special case where the surface is not regular for $\alpha>0$ splits into two types: $v_1=0$ and $v_1\neq0$. The former case gives $v_3=0=u_3$ and $u_1=0$. Then, so long as $\Delta_{02}\neq0$, the surface has the local form as stated in Theorem~\ref{regular}. It follows that $f(\alpha,\beta,s)=((3s-1)\alpha^2+\beta^2)\beta\Delta_{02}$, so $\Sigma_{u,v}$ is reducible. As a concrete (actually generic) example, consider the Chern characters $v=(2,0,-3,0)$ and $u=\ch(\ox)$. On $\p3$, $v$ is the Chern character of a rank $2$ locally free sheaf $E$ with $c_1(E)=0$ and $c_2(E)=3$; see Figure~\ref{fig:non-reg-example}.

\begin{figure} \centering
\includegraphics[scale=0.7]{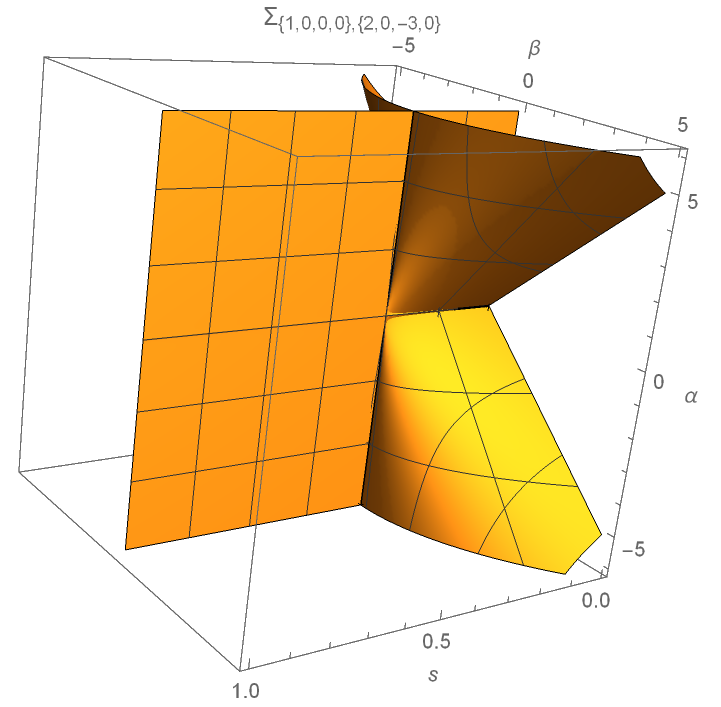}
\caption{An example of a reducible surface $\Sigma_{u,v}$ with $u=(1,0,0,0)$ and $v=(2,0,-3,0)$, illustrating case \eqref{reducible} of Theorem~\ref{regular}. }
\label{fig:non-reg-example}
\end{figure}

On the other hand, if $v_1\neq0$, the picture will be the same translated by $\mu(v)$ along $\beta$. Note that \eqref{v1} will determine $u_1$ for arbitrary non-zero $\ch(v)$ and $u_0$. Then \eqref{v3} will determine $v_3$, while \eqref{v2} will determine $u_2$: one root is also a root of $\Delta_{02}=0$ and so must be dismissed, and then $u_2$ is uniquely determined. This gives a $3$-parameter family of possible rational examples, but to correspond to actual objects, there are strong diophantine constraints which will depend on the threefold $X$.
\end{remark}

\begin{corollary}
If $\Delta_{01}\neq0$, then $\Sigma_{u,v}\cap\{\alpha>0,s>0,s\neq 1/3\}$ is regular and $\Upsilon_{u,v,s_0}=\Sigma_{u,v}\cap\{s=s_0\}$ is bounded for each $s_0>0$.
\end{corollary}

A numerical $\lambda$-wall $\Upsilon_{\alpha,\beta,s}$ is said to be \emph{horizontal} at a point $(\alpha_0,\beta_0)\in\HH$ if $\partial_\beta f_{u,v}(\alpha_0,\beta_0)=0$. 

\begin{proposition}\label{horizontal wall}
Assume $u$ and $v$ are as in Theorem~\ref{regular}.
Let $s=1/3$. A numerical $\lambda$-wall is horizontal at a point $(\alpha,\beta)\not\in\Theta_v$ if and only if $(\alpha,\beta)\in\Gamma_{v,s}$ or $(\alpha,\beta)\in\Xi_{u,v}$ away from $\Theta_v$. It is a local maximum on $\Gamma_{v,s}$ and a local minimum of $\,\Xi_{u,v}$ except for the special case where $\delta_{01}=0$ and $(\alpha,\beta)\in\Gamma_{v,s}\cap\Xi_{u,v}$, where it is a point of inflection.
\end{proposition}

\begin{proof}
By Lemma~\ref{dervs}, $\Upsilon_{\alpha,\beta,s}$ is horizontal if and only if $\partial_\beta f_{u,v,s}=\Delta_{13}(\alpha,\beta)=0$. First assume  the point is on $\Gamma_{v,s}$ (but not on $\Theta_v$), so that $\ch^{\alpha,\beta}_3(v)=0$. Then also $\ch^{\alpha,\beta}_3(u)=0$ because $\ch_2^{\alpha,\beta}(v)\neq0$ and $f_{u,v,s}(\alpha,\beta)=0$. Hence, $\Delta_{31}(\alpha,\beta)=0$, and so $\partial_\beta f_{u,v,s}=0$, and the wall is horizontal. In a local chart away from a point where the wall is vertical, we can view $\alpha$ as a function of $\beta$. Then the second derivative is given by
\[\frac{d^2\alpha}{d\beta^2}=-\frac{\partial^2_\beta f_{u,v,s}}{\partial_\alpha f_{u,v,s}}=\frac{\Delta_{21}+\Delta_{30}}{\alpha\Delta_{21}-\alpha\Delta_{30}}=-\frac{1}{\alpha}.\]

If the point is on the associated $\nu$-wall $\Xi_{u,v}$, then since $\ch_1^\beta(v)\neq0$, we have
\[\Delta_{32}=\frac{\ch_2^{\alpha,\beta}(v)}{\ch_1^\beta(v)}\Delta_{31},\]
and so if also $\Delta_{32}=0$, then the $\Upsilon_{u,v,s}$ must be horizontal there, unless the point lies on $\Theta_v$. The second derivative (if the point is not on $\Gamma_{v,s}$) is $1/\alpha$. Note that when the point is also on $\Gamma_{v,s}$, the wall is singular by Theorem~\ref{regular}\eqref{reg_tilt_gamma}; otherwise, $\Delta_{21}$ and $\Delta_{30}$ cannot vanish simultaneously except in the special case where~$\delta_{01}=0$. Consequently, there is never a point of inflection on a numerical $\lambda$-wall except at $\Gamma_{v,s}\cap \Xi_{u,v}$ in this special case.

Assume  $\Upsilon_{u,v,s}$ is horizontal at a point $(\alpha,\beta)$ not on $\Gamma_{v,s}$. Then $\Delta_{13}(\alpha,\beta)=0$, and from $\Delta_{23}(\alpha,\beta)=0$, we also have $\Delta_{12}(\alpha,\beta)=0$ by Lemma~\ref{delta_tech}\eqref{twovanish} since $\ch_3^{\alpha,\beta}\neq0$ by assumption. So $(\alpha,\beta)$ is on the associated $\nu$-wall.
\end{proof}

\begin{remark}\label{horiz_theta}
A numerical $\lambda$-wall might be horizontal as it crosses $\Theta_v$, though it generally is not. In that case, we can use $\beta$ and $s$ as local coordinates on the surface, and the second fundamental form at that point is $\operatorname{II}=\frac{1}{\alpha}d\beta^2$. Then the Gauss curvature is zero, and the mean curvature is $1/2\alpha$. 
\end{remark}

Now we consider how numerical $\lambda$-walls vary with $s$. Note that $\Sigma_{u,v}\cap\{s=s_0\}=\Upsilon_{u,v,s_0}$.

\begin{theorem}\label{create wall}
Let $\Sigma'$ be a non-empty connected component of $\,\Sigma_{u,v}$.
\begin{enumerate}
\item Suppose $s_0\geq 1/3$. If $\,\Sigma'\cap\{s=s_0\} \neq \emptyset$, then $\Sigma'\cap \{s=s_1\} \neq \emptyset$ for every $s_1\geq 1/3$.
\item Suppose $s_0<1/3$. If $\,\Sigma'\cap\{s=s_0\} \neq \emptyset$, then $\Sigma'\cap \{s=s_1\} \neq \emptyset$ for every $0\leq s_1 < s_0$.
\end{enumerate}
If $\,\Sigma'\cap\{s=s_0\} \neq \emptyset$, then $\Sigma'\cap\{s=s_1\} \neq \emptyset$ for every $0<s_1\leq s_0$.
\end{theorem}

In other words, if a numerical $\lambda$-wall exists for one $s_0\geq 1/3$, then it must also exist for all $s\geq 1/3$, whereas if $s_0<1/3$, then it need only exist for $s<s_0$. This means that numerical $\lambda$-walls can only be ``created'' for $s<1/3$ and as $s$ decreases. We shall see below that walls cannot be ``created'' in $R^{\pm}_{v,s}$ even when $s<1/3$, but they can be created in $R^0_{v,s}$.

\begin{proof}
First observe  that we may assume boundedness away from $s=0$ because if $\Delta_{01}=0$, then $\Delta_{20}\neq0$ and so $f_{u,v,s}(\alpha,\beta)=0$ is a cubic in $\beta$. Then it must have a solution for all $\alpha$. But if a wall crosses $\alpha=0$ for one value of $s$, then it crosses it for all $s$.

By Theorem~\ref{regular}, the surface is regular except at special points on $\alpha=0$ where there are multiple tangent planes which include the $s$-direction or we are in case \eqref{reg_tilt_gamma} of the theorem. But this latter case cannot arise in the present situation. This means that if $\Sigma_{u,v}\cap\{s=s_1\}$ is empty for some $s_1>0$, then there is some $s_0$ such that $\partial_\alpha f_{u,v,s_0}=0=\partial_\beta f_{u,v,s_0}$ and $\partial_s f_{u,v}|_{s=s_0}\neq0$. It follows that $\alpha\neq0$ and $\Delta_{12}(\alpha,\beta)\neq0$. 

Then $\Sigma_{u,v}|_{s_0}$ is a union of closed curves, unbounded curves or one or more distinct points. Suppose $(\alpha,\beta)$ is a point on a closed curve component. Since $f$ is a quadratic function of $\alpha^2$, it follows that nearby $(\alpha,\beta)$ there are two distinct solutions for a fixed $\beta$. If the curve does not cross $\alpha=0$, then there are four distinct solutions sufficiently close to two points in $\alpha>0$. But then there are also four points for $\alpha<0$, which is impossible. So the curve must cross $\alpha=0$. But this gives a contradiction as $\alpha\neq0$. We can also eliminate the possibility of unbounded curves, as follows. Each unbounded curve must have two distinct unbounded branches. Then nearby $s_0$, there will be four unbounded branches. But the implicit function defining the $\lambda$-wall $\Upsilon_{u,v,s_0}$ is only asymptotically cubic (see the proof of Proposition~\ref{bounded walls 1}), and so this is impossible.

Consequently, $\Sigma_{u,v}\cap \{s=s_0\}$ consists only of isolated points. Now consider one of these points, with coordinates $(\alpha_0,\beta_0)$. By regularity, we can use $\alpha$ and $\beta$ as local coordinates on $\Sigma_{u,v}$ at this point. Note that the first fundamental form at $(\alpha_0,\beta_0)$ is $d\alpha^2+d\beta^2$, from the vanishing of $-\partial_\alpha f_{u,v}/\partial_s f_{u,v}$ and $-\partial_\beta f_{u,v}/\partial_s f_{u,v}$.

From $f_{u,v}(\alpha_0,\beta_0,s_0)=0$, $\partial_\alpha f_{u,v}(\alpha_0,\beta_0,s_0)=0$ and $\partial_\beta f_{u,v}(\alpha_0,\beta_0,s_0)=0$, we have
\begin{gather}
\Delta_{32}=\alpha_0^2\left(s_0-1/3\right)\Delta_{12},\label{f_vanish}\\
\Delta_{30}=\left(1+2\left(s_0-1/3\right)\right)\Delta_{21}+\alpha_0^2\left(s_0-1/3\right)\Delta_{10},\label{from_d_alpha}\\
\Delta_{31}=\alpha_0^2\left(s_0-1/3\right)\Delta_{02},\label{from_d_beta}
\end{gather}
where we abbreviate $\Delta_{ij}(\alpha_0,\beta_0)$ to $\Delta_{ij}$ here and in the rest of the proof. Together with Lemma~\ref{delta_tech}\eqref{delta_identity}, these identities imply that
\begin{equation}\alpha_0^2(s_0-1/3)\Delta_{02}^2=-(2s_0+1/3)\Delta_{12}^2.\label{from_cyclic}\end{equation}
Since $\alpha_0\neq0$ and $\Delta_{12}\neq0$, we deduce that $\Delta_{02}\neq0$ and $s<1/3$. This establishes the first part of the theorem.

For the second part, we show that the mean curvature at our point is negative, and so the surface must curve towards $s<s_0$. To this end, observe that the Gauss curvature $K$ is a positive multiple of  $\partial_{\alpha}^2f_{u,v}\partial_{\beta}^2f_{u,v}-(\partial_{\alpha}\partial_{\beta}f_{u,v})^2$. Using \eqref{from_d_alpha}, we have 
\begin{align*}
\partial_{\alpha}^2f_{u,v} & = 2\alpha_0^2(2s_0+1/3)\Delta_{10}, \\
\partial_{\alpha}\partial_{\beta}f_{u,v} & = 2\alpha_0(s_0-1/3)\Delta_{02}, \\
\partial_{\beta}^2f_{u,v} & = \alpha_0^2(s-1/3)\Delta_{10}+2(s_0+2/3)\Delta_{21}.
\end{align*}
Using \eqref{from_cyclic}, we have $(\partial_{\alpha}\partial_{\beta}f_{u,v})^2=-(2s_0+1/3)(s_0-1/3)\Delta_{21}^2$. Then $K$ has the same sign as 
\[\begin{split}
    &\alpha_0^4(s_0-1/3)\Delta_{10}^2+\alpha_0^2(s_0+2/3)\Delta_{10}\Delta_{21}+(s_0-1/3)\Delta_{21}^2 \\
    &=(s_0-1/3)(\alpha_0\Delta_{10}+\Delta_{12})^2+3s_0\alpha_0^2\Delta_{10}\Delta_{21}
\end{split}
\]
But the tangent plane at $(\alpha_0,\beta_0)$ only intersects $\Sigma_{u,v}$ at that point and so $K\geq0$. The first term is negative, and so we must have $\Delta_{10}\Delta_{21}>0$. 

The mean curvature is then
\[-\frac{\partial_{\alpha}^2f_{u,v}}{\partial_s f_{u,v}}-\frac{\partial_{\beta}^2f_{u,v}}{\partial_s f_{u,v}}=
-3s_0\alpha_0\frac{\Delta_{10}}{\Delta_{21}}-\frac{2}{\alpha_0}\left(s_0+\frac23\right)<0
\]
as required.
\end{proof}

We observe that numerical $\lambda$-walls do admit isolated points; see Figure~\ref{pointexample} below.

\begin{figure}[ht]\centering
\includegraphics[scale=0.7]{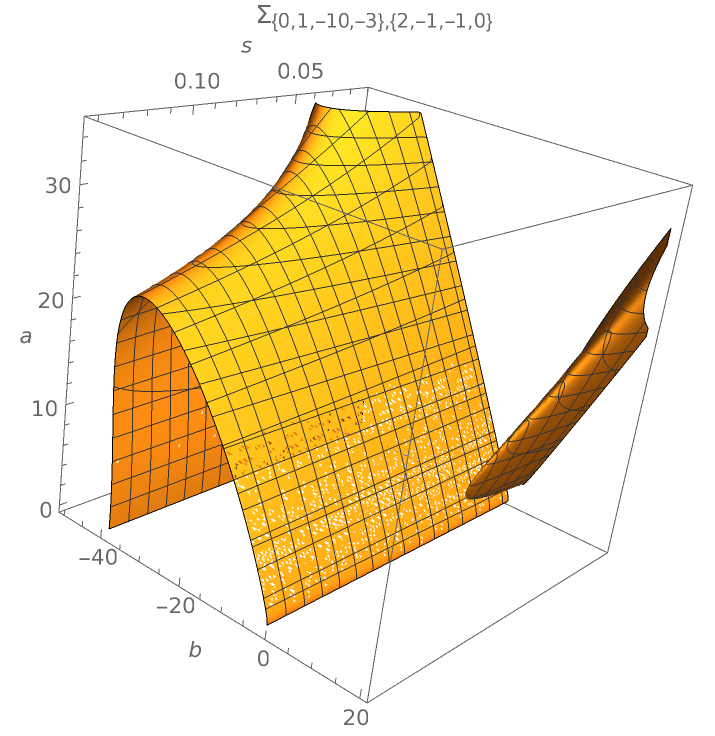}    
\caption{This is the surface wall $\Sigma_{u,v}$ for $v=(2,-1,-1,0)$ and $u=(0,1,-10,-3)$. Taking $s_0\simeq0.0569$, the point $(\alpha,\beta)\simeq(6.24,3.52)$ is an isolated point of the numerical $\lambda$-wall $\Upsilon_{u,v,s_0}=\Sigma_{u,v}|_{s_0}$.} \label{pointexample}
\end{figure}

As a consequence of Theorem~\ref{regular} and the proof of Theorem~\ref{create wall}, we can deduce that $1$-dimensional components of numerical $\lambda$-walls for a fixed $s$ are regular away from $\nu$-walls.

\begin{corollary}\label{regular_wall}
For any $s>0$ and any real numerical Chern characters $u$ and $v$, any connected component of a numerical $\lambda$-wall $\Upsilon_{u,v,s}$ in $\mathbb{H}$ is regular as a real curve away from $\Xi_{u,v}$. 
In particular, if there is no associated $\nu$-wall for the pair $u$, $v$, then $\Upsilon_{u,v,s}$ is always regular.
\end{corollary}

\begin{proof}
Away from the $\nu$-wall $\Xi_{u,v}$ (if non-empty) when $s=1/3$, $\Sigma_{u,v}$ is regular by Proposition~\ref{special_regular} or  Theorem~\ref{regular}. But by the proof above, the only place where the plane tangent to $\Sigma_{u,v}$ is parallel to the $(\alpha,\beta)$-plane occurs at isolated points and not in a $1$-dimensional portion of the surface wall.  
\end{proof}

We complete this section by returning to the issue of the intersection of numerical $\lambda$-walls for $v$ and the curves $\Theta_v$ and $\Gamma_{v,s}$.

Unravelling the equality in \eqref{eqn for b_0} yields a cubic polynomial equation for $\beta_0$. This means that the intersection $\Upsilon_{u,v,s}\cap\Gamma_{v,s}$ consists of at most three points away from $\Theta_v$; in addition, according to Lemma~\ref{indep of s}, the number of intersection points does not depend on the parameter $s$. Notice that the total number of intersection points of a $\lambda$-wall for $v$ with $\Gamma_{v,s}$ will increase by $1$ if $q(v)<0$.

Let us now examine one situation in which $\Upsilon_{u,v,s}\cap\Gamma_{v,s}$ contains two points. 

\begin{lemma}\label{twice} 
If a connected component of a numerical $\lambda$-wall $\Upsilon_{u,v,s}$ crosses a connected component of $\Gamma_{v,s}$ twice away from the hyperbola $\Theta_v$, then the associated numerical $\nu$-wall $\Xi_{u,v}$ is non-empty. 
\end{lemma}

\begin{proof}
By  Lemma~\ref{indep of s}, we can let $s=1/3$. Then the two intersection points are local maxima of $\Upsilon_{u,v,1/3}$. By Corollary~\ref{regular_wall}, the component is regular or has a tacnode on a $\nu$-wall. If it is regular, then it must have a minimum between the two maxima. Proposition~\ref{horizontal wall} implies that such a minimum must lie on the associated $\nu$-wall.
\end{proof}

Figure~\ref{bubble} illustrates the typical situation described in Lemma~\ref{twice} and shows that it does arise. It is easy to see that the intersection with $\alpha=0$ must happen to the right of $\Gamma^-_{v,s}$ because $\Gamma^-_{v,s}$ is monotonic.

\begin{figure}[ht]
\includegraphics{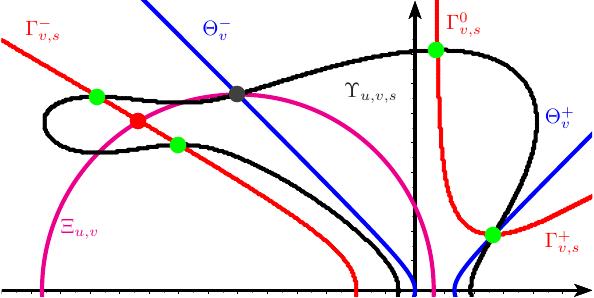}
\caption{This example contains several interesting features. We set $s=1/3$, $v=(3,1,0,-1)$ and $u=(0,1,-3,7)$. First, the curve $\Gamma_{v,s}$ (in red) intersects the positive branch of the hyperbola $\Theta_v$ (in blue); this intersection is marked with a black bullet. Second, the numerical $\lambda$-wall $\Upsilon_{u,v,s}$ (in black) crosses the curve $\Gamma_{v,s}$ four times (marked with green bullets), twice along $\Gamma_{v,s}^-$. Finally, we can see that the associated numerical $\nu$-wall $\Xi_{u,v}$ (in magenta) cuts $\Gamma_{v,s}^-$ (red bullet) between the two intersection points of $\Upsilon_{u,v,s}\cap\Gamma_{v,s}^-$, illustrating the phenomenon described in Lemma~\ref{twice}. Both $\Upsilon_{u,v,s}$ and $\Xi_{u,v}$ cut $\Theta_v^-$ at the same point.}
\label{bubble}
\end{figure}

The situation in Figure~\ref{bubble} also demonstrates another phenomenon in which a numerical $\lambda$-wall intersects horizontal lines four times. When this happens, it must be that $f(\alpha,\beta)$ has three turning points along this horizontal line. But note that $\partial_\beta f=-\Delta_{31}$ when $s=1/3$. On the other hand, $\partial_\alpha\Delta_{31}=0$, and so the solutions of $\Delta_{31}=0$ are vertical lines. But these intersect the wall at its horizontal turning points (the green points in Figure~\ref{bubble}). The middle one also intersects $\Upsilon_{u,v,1/3}$ again at a minimum or on $\Theta_v$, which must therefore also be on $\Xi_{u,v}$. For a connected component, we could already deduce this because such a component must have at least one minimum, but the same will follow even when the geometry of the wall does not require there to be a minimum. This more precise reasoning allows us to refine Lemma~\ref{twice} as follows. 

\begin{lemma}\label{double_to_left}
If a numerical $\lambda$-wall $\Upsilon_{u,v,s}$ intersects $\Gamma^-_{v,s}$ twice, then $\Upsilon_{u,v,s}$ must intersect $\Theta_v$ and $\Xi_{u,v}$ intersects $\Gamma^-_{v,s}$ in between the intersection points.
\end{lemma}

\begin{proof}
Note that Lemma~\ref{twice} shows that if the wall is a connected component, then it must intersect $\Xi_{u,v}$. Since the solutions of $\Delta_{31}=0$ are vertical lines, there must be minima above or below the maxima (the green dots on $\Gamma^-_{v,1/3}$ in Figure~\ref{bubble}). Since the wall cannot cross $\Gamma^-_{v,s}$ for a third time because it would have to double back on itself and there can only be at most two solutions of $\Delta_{32}=0$ along any vertical line, it must cross $\Xi_{u,v}$ a third time at a point which is not a minimum, which must therefore be on $\Theta_v$. Note that an alternative picture to Figure~\ref{bubble} has the wall cross $\alpha=0$ twice to the left of $\Gamma^-_{v,s}$. In that case, it only crosses $\Xi_{u,v}$ once, and then $\Xi_{u,v}$ does not intersect $\Gamma^-_{u,v}$ between the two intersection points. 

Now we assume the wall is two nested components.
First we observe that this hypothesis is independent of $s$. To see that, note that if a component were created at $(\overline\alpha,\overline\beta)$, then there would be six distinct solutions of $f(\alpha,\beta)=0$ along $\beta=\overline\beta$, which is impossible. The two bounded components must have maxima at $s=1/3$ which must intersect $\Gamma^-_{v,1/3}$.

Then there are three vertical line components of $\Delta_{31}=0$ which intersect $\Gamma^-_{v,1/3}$ at the maxima of the wall at distinct $\beta$-values.
Since the inner bounded component must have a maximum (on $\Gamma^-_{v,1/3}$), the middle component of $\Delta_{31}=0$ intersects the wall at that point and so intersects the outer component away from $\Gamma^-_{v,1/3}$. But this point cannot be a point of inflection (see Remark~\ref{horiz_theta}) and so must be a minimum, which must therefore be on a $\nu$-wall. If the wall is otherwise in the exterior of the $\nu$-wall, then that point must be at the maximum of the $\nu$-wall and so is on $\Theta_{v}$. Otherwise, it must cross the $\nu$-wall again, which is not a turning point, and so must also cross $\Theta_v$. In fact this minimum of $\Upsilon_{u,v,1/3}$ in the outer component must have another maximum, which must be on $\Gamma_{v}$ (see Figure~\ref{line-fig-2} for a concrete example of this case).
\end{proof}

We can argue similarly to Lemma~\ref{double_to_left} in $R^0_{v}$. 

\begin{lemma}\label{double_R0}
If a numerical $\lambda$-wall $\Upsilon_{u,v,s}$ intersects $\Gamma^0_{v,s}$ twice, then $\Upsilon_{u,v,s}$ intersects $\Xi_{u,v}$ and $\Theta_v$.
\end{lemma}

\begin{proof}
First observe that by Lemma~\ref{indep of s}, the hypothesis and conclusion are independent of $s$, and so we may set $s=1/3$.

If  $\Upsilon_{u,v,s}$ has a single component, then there must be a minimum between the maxima, and so it crosses $\Xi_{u,v}$ at that point. If there are two components, then the geometry does not require there to be a minimum. But then there are three vertical components of $\Delta_{31}=0$. The middle one intersects the inner component at its maximum, and then it must intersect the other component at a minimum (by Proposition~\ref{horizontal wall}) which must be on $\Xi_{u,v}$.  

In the last case, either the outer component is unbounded and the minimum occurs on $\Theta_v$, or it is bounded and there is a further maximum which must be on $\Gamma^{\pm}_{v,1/3}$, and so the wall again crosses $\Theta_v$. The shape is illustrated in Figure~\ref{line-fig-2}.
\end{proof}

We can now state the main theorem of this section. 

\begin{theorem}\label{RLvs}
Suppose a real numerical Chern character $v$ satisfies the Bogomolov--Gieseker inequality and $v_0\neq0$. Any connected bounded component of a numerical $\lambda$-wall in $R^{-}_{v,s}$ for some $s\geq1/3$ intersects $\Gamma^-_{v,s}$. 
\end{theorem}

\begin{proof}
Let $\Upsilon_{u,v,s}$ be the component, and suppose for a
contradiction that it does not cross $\Gamma^-_{v,s}$.  By
Theorem~\ref{create wall}, this component must exist for all $s$, and
so we may assume $s=1/3$. By Theorem~\ref{regular},
$\Upsilon_{u,v,1/3}$ is regular except possibly in case
\eqref{reg_tilt_gamma} of that theorem, in which case the singular
point is on $\Gamma^-_{v,s}$, as required. Otherwise,
$\Upsilon_{u,v,1/3}$ is regular, and so it must have a maximum. By
Proposition~\ref{horizontal wall}, it must intersect
$\Gamma^-_{v,1/3}$ at this point.
\end{proof}

\begin{remark}
In other words, if we want to classify all of the actual $\lambda$-walls in the region $R^{-}_{v,s}$ then we only need to locate the ones which cross $\Gamma^-_{v,s}$. But note that the wall may not be actual as it crosses $\Gamma^-_{v,s}$. 
\end{remark} 

On the other hand, this fails for $s<1/3$ because isolated walls can appear as $s$ decreases. Some will be unbounded as $s\to0$, but others may persist as isolated closed curves.

Finally, observe that analogous statements also hold for $\Gamma^+_{v,s}$ and the region to the right of $\Theta^+_v$.

\section{Asymptotic \texorpdfstring{$\boldsymbol{\labs}$}{lambda\textunderscore\{alpha,beta,s\}}-stability in \texorpdfstring{$\boldsymbol{R_{v,s}^{\pm}}$}{R\textasciicircum\{+/-\}\textunderscore\{v,s\}}} \label{labs sst -}

Similarly to Definition~\ref{asym nu s-st}, we introduced the following definition, where $\gamma\colon [0,\infty)\to\HH$ is  an unbounded path.

\begin{definition}\label{asym s-st}
An object $A\in\dbx$ is \emph{asymptotically $\labs$-$($semi$)$stable along $\gamma$} if the following two conditions hold for a given $s>0$: 
\begin{enumerate}[(i)]
\item There is a $t_0>0$ such that $A\in\mathcal{A}^{\gamma(t)}$ for every $t>t_0$.
\item There is a $t_1>t_0$ such that for all $t>t_1$, every sub-object $F\into A$ in $\mathcal{A}^{\gamma(t)}$ satisfies $\lambda_{\gamma(t),s}(F) <$ $(\le)$ $\lambda_{\gamma(t),s}(A)$.
\end{enumerate}
\end{definition}

Our first goal is to characterize asymptotically $\labs$-semistable objects with numerical Chern character $v$ satisfying $v_0>0$ and the Bogomolov--Gieseker inequality along two families of paths contained in the region $R_{v}^-$, namely the paths $\Gamma^{-}_{v,s}$ for each $s>0$ and $\Lambda^-_{\oalpha}$ for each $\oalpha>0$; see the notation introduced in display
\eqref{h-lines}. 

We then invoke Proposition~\ref{dual_in_A} to characterize asymptotically $\labs$-semistable objects along the paths $\Gamma^{+}_{v,s}$ and $\Lambda^+_{\oalpha}$. More precisely, we show that $\labs$-semistable objects along $\Gamma^{+}_{v,s}$ and $\Lambda^+_{\oalpha}$ are duals of Gieseker semistable sheaves.

First, we provide a simple consequence of asymptotic $\labs$-semistability and the support
property along unbounded paths.

\begin{lemma}
Let $A\in\dbx$ be an asymptotically $\labs$-semistable object along an unbounded path
$\gamma(t)$. Then $Q^{\mathrm{tilt}}(A)\geq0$. 
\end{lemma}

\begin{proof}
The support property in Proposition~\ref{support_property} implies that $Q_{\gamma(t)}(A)\geq0$ for all large enough $t$.
  
First suppose $\lim_{t\to\infty}\beta(t)=\pm\infty$;  then
\[0\leq\lim_{t\to\infty}\frac{Q_{\gamma(t)}(A)}{\beta(t)^2}=\left(\lim_{t\to\infty}\frac{\alpha(t)^2}{\beta(t)^2}+1\right)Q^{\mathrm{tilt}}(A).\]
As the first factor is positive for all $t$, it follows that $Q^{\mathrm{tilt}}(A)\geq0$.

If $\lim_{t\to\infty}\alpha(t)=\infty$, then
\[0\leq\lim_{t\to\infty}\frac{Q_{\gamma(t)}}{\alpha(t)^2}=Q^{\mathrm{tilt}}(A),\]
and the same conclusion follows.
\end{proof}

\subsection{Asymptotics along $\boldsymbol{\Gamma^{-}_{v,s}}$}\label{sec:gamma-}

We now turn to asymptotically $\labs$-semistable objects with Chern character $v$ along the curve $\Gamma^{-}_{v,s}$; we assume from now on  $v$ satisfies $v_0\ne0$ and the Bogomolov--Gieseker inequality \eqref{B-ineq}.

Since $\ch_1^\beta(E)\neq0$ along $\Gamma^{\pm}_{v,s}$, we have
\begin{equation} \label{subst}
\labs(E) = 0 ~~ \Leftrightarrow ~~ \alpha^2 = \dfrac{\ch^\beta_3(E)}{(s+1/6)\ch_1^\beta(E)},
\end{equation}
at least away from the point where $\Gamma^{\pm}_{v,s}$ meets $\Theta_v$ (it meets it exactly once if $q(E)<0$ and not at all otherwise), so we can use $\beta<0$ as a parameter.

Substituting $\alpha^2$ from equation \eqref{subst} into the expression for $\labs(u)$, we define the function
$\olambda_{v,\beta,s}(u)$; to be precise,
$$ \olambda_{v,\beta,s}(u) = \dfrac{\ch_3^{\beta}(u)\ch_1^{\beta}(v)-\ch_3^{\beta}(v)\ch_1^{\beta}(u)}{\ch_2^{\beta}(u)\ch_1^{\beta}(v)-u_0\ch_3^{\beta}(v)/(2s+1/3)} $$ 
expresses the $\lambda$-slope of an object $F$ with $\ch(F)=u$ along $\Gamma^{\pm}_{v,s}$. If $\ch_0(F)\ne0$, the previous expression yields
\begin{equation} \label{asymp lambda}
\olambda_{v,\beta,s}(F) = - \dfrac{6s+1}{9s\ch_0(E)\ch_0(F)}\cdot
\dfrac{ \delta_{01}\beta^3 - 3\delta_{02}\beta^2 + 3(\delta_{03}+\delta_{12})\beta - 3\delta_{13} }{ \beta^3 + \textrm{ lower order terms} },
\end{equation}
where $\delta_{ij}=\delta_{ij}(F,E)$, as defined in equation \eqref{Delta}. When $\ch_0(F)=0$ and $\ch_1(F)\ne0$, we obtain
\begin{equation} \label{asymp lambda rk0}
\olambda_{v,\beta,s}(F) = -\dfrac{1}{3}  \dfrac{\ch_0(E)\ch_1(F)\beta^3+\textrm{ lower order terms}}{\ch_0(E)\ch_1(F)\beta^2+\textrm{ lower order terms}}.
\end{equation}

Note that a numerical $\lambda$-wall $\Upsilon_{u,v,s}$ crosses $\Gamma_{v,s}^\pm$ precisely at the zeros of the function $\olambda_{v,\beta,s}(u)$.

The behaviour of the $\nu$-slope of an object $F\in\dbx$ along $\Gamma_{v,s}^\pm$ can be analysed in a similar way. Substituting $\alpha^2$ from equation \eqref{subst} into the expression for $\nuab(F)$, we obtain
\begin{equation}\label{eq:limitnu}
\lim_{\beta\to-\infty}\dfrac{-1}{\beta}\nuab(F) =
\begin{cases}
\dfrac{s}{2s+1/3} &\textrm{if } \ch_0(F)\ne0, \\ 
1 &\textrm{if } \ch_0(F)=0,~\ch_1(F)\ne0.
\end{cases}
\end{equation}

We are now in position to state the main result of this section.

\begin{theorem}\label{t:asymtotegamma}
Let $v$ be a numerical Chern character with $v_0\ne0$ satisfying the Bogomolov--Gieseker inequality \eqref{B-ineq}. For each $s>0$, an object $A\in\dbx$ with $\ch(A)=v$ is asymptotically $\labs$-$($semi\/$)$stable along $\Gamma_{v,s}^-$ if and only if $A$ is a Gieseker $($semi\/$)$stable sheaf.
\end{theorem}

The proof will be done in a series of lemmas. 

\begin{lemma}\label{gammaasymp}
For every $s>0$, if there is a $\beta_0<0$ such that $E\in\cohab$ for $(\alpha,\beta)\in\Gamma^-_{v,s}$ with $\beta<\beta_0$, then $E\in\coh(X)$.
\end{lemma}

This claim actually follows directly from Proposition~\ref{dual_in_A}, but we give an alternative hands-on proof.

\begin{proof}
Using the notation of Section~\ref{2nd tilt}, note that $\nuab^+(E_1)\leq0$, since $E_1\in\free{\alpha,\beta}$, for every $(\alpha,\beta)\in\Gamma^-_{v,s}$ with $\beta<\beta_0$. However, equation \eqref{eq:limitnu} implies that for every sub-object $F\into E_1$, there exists a $\beta_0'<0$, depending only on $\ch_{\le2}(F)$, such that $\nuab(F)>0$ for every $(\alpha,\beta)\in\Gamma^-_{v,s}$ and $\beta<\beta_0'$, leading to a contradiction. It follows that $E_1=0$.

On the other hand, the inequality
$$ \ch_1^\beta(E_{01})=\ch_1(E_{01})-\beta\ch_0(E_{01})\leq0 $$ for all $\beta\ll 0$ implies that $\ch_0(E_{01})=0$; thus $E_{01}=0$ since $E_{01}$ is torsion-free.
\end{proof}

\begin{lemma}\label{asymtogs}
For every $s>0$, if $E$ is an asymptotically $\labs$-$($semi\/$)$stable object along $\Gamma^-_{v,s}$, then $E$ is a Gieseker $($semi\/$)$stable sheaf.
\end{lemma}

\begin{proof}
If $E$ is asymptotically $\labs$-semistable along $\Gamma^-_{v,s}$, then $E$ is a sheaf by Lemma~\ref{gammaasymp} or Proposition~\ref{dual_in_A}.

If $E$ is not torsion-free, let $F\into E$ be its maximal torsion subsheaf; if $T\into F$ is a subsheaf of dimension at most $1$, then $T\into E$ is a morphism in $\cala^{\alpha,\beta}$ for $(\alpha,\beta)\in\Gamma^-_{v,s}$ and $\beta\ll0$ since $T\in\tors{\alpha,\beta}$ for every $(\alpha,\beta)$. We have that $\ch_0(T)=\ch_1(T)=0$; then
\begin{equation} \label{eq:limitzerorank-2}
\olambda_{v,\beta,s}(T) = \begin{cases}
-\beta +  \dfrac{\ch_3(T)}{\ch_2(T)} &\textrm{if } \ch_2(T)\ne0, \\ 
+\infty &\textrm{if } \ch_2(T)=0, 
\end{cases} 
\end{equation}
so $T$ would destabilize $E$. Therefore, we can assume  $F$ has pure dimension 2; let $T$ be its maximal $\hat{\mu}$-semistable subsheaf. Remark~\ref{h nu-ss rk0} implies that $T\in\cala^{\alpha,\beta}$ for $(\alpha,\beta)\in\Gamma^-_{v,s}$ and $\beta\ll0$; thus $T\into E$ is a morphism in $\cala^{\alpha,\beta}$ in the same range. Since $\ch_0(T)=0$ and $\ch_1(T)\ne0$, we have
$$ \lim_{\beta\to-\infty} \dfrac{-1}{\beta} \olambda_{v,\beta,s}(T) = \dfrac{1}{3}, $$ again contradicting asymptotic $\lambda$-semistability. We therefore conclude that $E$ must be torsion-free.

If $F\into E$ is a proper (torsion-free) subsheaf with $\delta_{01}(F,E)\ne0$, then equation \eqref{asymp lambda} implies that
$$ \lim_{\beta\to-\infty} \olambda_{v,\beta,s}(F) = - \dfrac{6s+1}{9s\ch_0(E)\ch_0(F)} \delta_{01}(F,E) \le 0 $$
by hypothesis; thus $\delta_{01}(F,E)\ge0$.

If $\delta_{01}(F,E)=0$, then equation \eqref{asymp lambda} implies that
$$ \lim_{\beta\to-\infty} (-\beta)\olambda_{v,\beta,s}(F) = - \dfrac{6s+1}{3s\ch_0(E)\ch_0(F)} \delta_{02}(F,E) \le 0; $$ thus $\delta_{02}(F,E)\ge0$.

Finally, if $\delta_{01}(F,E)=\delta_{02}(F,E)=0$, then
$$ \lim_{\beta\to-\infty} \beta^2 \olambda_{v,\beta,s}(F) = -\dfrac{6s+1}{3s\ch_0(E)\ch_0(F)} \delta_{03}(F,E) \le 0; $$ thus $\delta_{03}(F,E)\ge0$, with equality holding if and only of $\delta_{03}(F,E)=0$ as well. In other words, $E$ is Gieseker (semi)stable.
\end{proof}

We now prove the converse of Lemma~\ref{asymtogs}. The difficulty is that the strong definition of asymptotic stability includes showing that a given object has only finitely many walls along $\Gamma^-_{v,s}$. In fact, we can show that there is at most one, at least outside its actual $\nu$-wall.

Assume $E$ is a Gieseker stable sheaf, and suppose there is an actual $\lambda$-wall given by the short exact sequence $0\to F\to E\to G\to0$ crossing $\Gamma^-_{v,s}$. Consider the case where $F$ and $G$ are sheaves.

\begin{lemma}\label{asymp_gamma}
Fix $s>0$ and $\obeta$.
Suppose $E$ is a Gieseker stable sheaf with $\ch(E)=v$ and $F\into E$ is a sub-object in both $\coh(X)$ and $\cohab$ such that $E/F\in\cohab\cap\coh(X)$ for all $\beta<\obeta$ and $(\alpha,\beta)\in\Gamma^-_{v,s}$. Then there is some $\beta_0<\obeta$ such that 
$\lambda_{\alpha,\beta,s}(F)<\lambda_{\alpha,\beta,s}(E)$ for all $\beta<\beta_0$.
\end{lemma}

\begin{proof} 
We can compute the same limits as in the proof of Lemma~\ref{asymtogs}; we first have
$$ \lim_{\beta\to-\infty} \olambda_{\alpha,\beta,s}(F) = - \dfrac{6s+1}{9s\ch_0(E)\ch_0(F)} \delta_{01}(F,E) \le 0 $$ since $\delta_{01}(F,E)\ge0$ by hypothesis.  If $\delta_{01}(F,E)=0$, then
$$ \lim_{\beta\to-\infty} (-\beta)\olambda_{\alpha,\beta,s}(F) = - \dfrac{6s+1}{3s\ch_0(E)\ch_0(F)} \delta_{02}(F,E) \le 0 $$ because $\delta_{02}(F,E)\ge0$. Finally, if $\delta_{02}(F,E)$ also vanishes, then
$$ \lim_{\beta\to-\infty} \beta^2 \olambda_{\alpha,\beta,s}(F) = -\dfrac{6s+1}{3s\ch_0(E)\ch_0(F)} \delta_{03}(F,E) \le 0 $$
since $\delta_{03}(F,E)\ge0$, with equality holding if and only if $\delta_{03}(F,E)=0$ as well, which implies that 
$\olambda_{\alpha,\beta,s}(F)=0$ for every $(\alpha,\beta)\in\Gamma^-_{v,s}$ for $\beta\ll0$.
\end{proof}

This enables us to prove that Gieseker stable sheaves can only be destabilized as we move down along~$\Gamma^-_{v,s}$.

\begin{lemma}\label{sheaf_no_destab}
Suppose $E$ is a Gieseker stable sheaf and $F\into E$ is a subsheaf which also corresponds to an actual $\lambda$-wall $W_{u,v,s}$ crossing $\Gamma^-_{v,s}$ at a point $P$ beyond any $\nu$-wall for $E$. If $E/G\in\cohab\cap\coh(X)$, then $E$ must be stable above the point $P$.
\end{lemma}

\begin{proof}
Parameterize the curve $\Gamma^-_{v,s}$ by $\gamma(t)$ in the decreasing $\beta$-direction with $\gamma(0)=P$. By assumption and Proposition~\ref{dual_in_A}, $E\in\cala^{\gamma(t)}$ for all $t\geq0$.  Suppose $E$ is unstable in $\mathcal{A}^{\gamma(t)}$ for some $t>0$. Observe that for $t\gg0$, we have that $F\into E$ is a monomorphism in $\cala^{\gamma(t)}$. This is because if $\lambda_{\gamma(t),s}(F)\geq\lambda_{\gamma(t),s}(E)$ for all $t>0$, then Corollary~\ref{ses_in_A} and Lemma~\ref{asymp_gamma} give a contradiction. It follows that at some point $Q=\gamma(t_0)$, for $t_0>0$, we have $\lambda_{Q,s}(F)=\lambda_{Q,s}(E)$, and so the numerical $\lambda$-wall $\Upsilon_{u,v,s}$ crosses $\Gamma^-_{v,s}$ twice. Then Lemma~\ref{double_to_left} tells us that there is a numerical $\nu$-wall $\Xi_{u,v}$ intersecting $\Gamma^-_{v,s}$ between $P$ and $Q$. But $0\to F\to E\to G\to0$ is a short exact sequence in $\mathcal{B}^{\beta(t)}$ for all $t\geq0$, and so this is an actual $\nu$-wall, contradicting the assumption.
\end{proof}

We now show that above an actual $\nu$-wall for $E$, the $\lambda$-wall equivalent of Lemma~\ref{asymp_B_is_sheaf}\eqref{one_nu_wall} holds.

\begin{theorem}\label{one_lambda_wall}
Suppose $E$ is a Gieseker stable sheaf with Chern character $v$. Then there is at most one actual wall $W_{u,v,s}$ intersecting $\Gamma^-_{v,s}$ at a point $P$ above the actual $\nu$-wall for $E$. In particular, such a wall destabilizes downwards along $\Gamma^-_{v,s}$.
\end{theorem}

\begin{proof}
We prove the last part first. Suppose  a numerical wall $\Upsilon_{u,v,s}$ intersects $\Gamma^-_{v,s}$ at $P$ above the $\nu$-wall for $E$. Note that $\Upsilon_{u,v,s}$ is a slice of $\Sigma_{u,v}$, and since this is an orientable surface, to show that the $E$ can only be destabilized downwards, it suffices to consider the case $s=1/3$. We show that $f_{u,v,1/3}(\alpha,\beta)$ is increasing as we cross $\Upsilon_{u,v,1/3}$ moving  down $\Gamma^-_{v,1/3}$. This is equivalent to showing 
\[\nabla f\cdot (-\partial_\beta \tau,\partial_\alpha\tau)\bigr|_{P}>0.\]
 Using \eqref{dervch} and Lemma~\ref{dervs}, we have $\nabla
 f=(\alpha\Delta_{21}-\alpha\Delta_{30}, -\Delta_{31})$, and the
 tangent vector down along $\Gamma^-_{v,s}$ is
 $(\ch^{\alpha,\beta}_2,-\alpha\ch^{\alpha,\beta}_1)$. Note that
 $\ch_3^P(v)=0=\ch_2^P(u)$ since $\ch_2^P(v)\neq0$. Hence,
 $\Delta_{3i}(P)=0$. So
 \[\nabla f\cdot (-\partial_\beta \tau,\partial_\alpha\tau)\bigr|_{P}=\alpha\ch_2^P(v)\Delta_{21}(P).\]
But $\Delta_{21}(P)>0$ as we are outside the $\nu$-wall (where $\Delta_{21}$ vanishes), and $\ch_2^P(v)>0$ as $P\in R^-_v$. 

So if we have two walls $W_{u,v,s}$ and $W_{u',v,s}$ corresponding to $F_1\hookrightarrow E$ and $F_2\hookrightarrow E$ crossing $\Gamma^-_{v,s}$ at $P=\gamma(0)$ and $Q=\gamma(1)$, respectively, then $F_2$ destabilizes $E$ below  $Q$. If $F_2\to E$ remains an injection in $\mathcal{A}^{\gamma(t)}$ for $0\leq t\leq1$, $W_{u',v,s}$ must cross $\Gamma^-_{v,s}$ again, and then there would be an actual $\nu$-wall between $P$ and $Q$, which is not permitted by hypothesis. So it must be that $F_2\to E$ ceases to be an injection at some point. Let that point be $R$.  We show that $E$ remains unstable as we cross $R$. In fact, we will prove a stronger result in the following lemma.
\end{proof}

\begin{lemma}
Suppose $E$ is a Gieseker stable sheaf with $\ch(E)=v$, and let $\gamma$ be a curve segment of $\,\Gamma^-_{v,s}$ from a point $Q=\gamma(0)$ to $P=\gamma(1)$ which is outside an actual $\nu$-wall. Suppose $0\to A\to E\to B\to 0$ is a short exact sequence in $\mathcal{A}^{P}$ corresponding to an actual $\lambda$-wall through the point $Q$. Then $E$ is $\lambda_{\gamma(t),s}$-unstable for all $t\in (0,1]$.
\end{lemma}

\begin{proof}
By the second statement of the previous lemma, we know that $0\to A\to E\to B\to 0$ destabilizes for $t\in(0,\epsilon)$ for some $\epsilon>0$. We also know from that proof that $\lambda_{\gamma(t),s}(A)>\lambda_{\gamma(t),s}(E)$ for all $t\in(0,1]$. In the case where $A$ is not a sheaf, we have that $0\to A\to E\to B\to 0$ remains short exact to the end of $\Gamma^-_{v,s}$ as $E=A_{00}$ and $B=A_{01}[2]$ and $\Gamma^-_{v,s}$ ends on either the $\beta$-axis or $\Theta^-_v$. So we may assume $A$ is a sheaf.

Then $A$ remains in the category until its Harder--Narasimhan factor with smallest $\nu$, $A^-$, say, goes out of the category. But then just beyond that point $\lambda(A^-)>0$, and so $E\to \ker_{\mathcal{A}}(B\to A^-[1])$ still destabilizes $E$. We can continue until we have exhausted all of the Harder--Narasimhan factors $A'$ of $A$. Finally, as we approach the last $\Theta^-_{\ch(A')}$-curve for the filtration of $A$, we have $\lambda(A')<0$. But then there would be a wall corresponding to $A'$ destabilizing $E$ above, which is impossible. It follows that $A$ remains in $\mathcal{A}^{\gamma(t)}$ for all $t\in[0,1]$.

Similarly, $B_0$ remains in $\mathcal{A}^{\gamma(t)}$ until $\Theta^-$ of a factor, but this cannot happen as $E$ remains in the category. But then $B$ remains in the category by Theorem~\ref{t:stays_in_A}. This completes the proof.
\end{proof}

We can now complete the proof of Theorem~\ref{t:asymtotegamma}.

\begin{lemma}\label{gstoasym}
If $E$ is a Gieseker $($semi\/$)$stable sheaf, then $E$ is asymptotically $\labs$-$($semi\/$)$stable along $\Gamma^-_{v,s}$.
\end{lemma}

\begin{proof}
If $E$ is Gieseker stable, then the first part of Definition~\ref{asym s-st} follows from the fact that $E\in\coh(X)$ and Proposition~\ref{dual_in_A}, and the second now follows from Theorem~\ref{one_lambda_wall}. The statement for semistability follows by inducting on the length of the Jordan--H\"older filtration of $E$.
\end{proof}

\begin{remark}
Just as for Proposition~\ref{asymp_HN}, we can deduce from Lemma~\ref{gstoasym} that any $E\in\cohab$ for all $(\alpha,\beta)$ along an unbounded $\Theta^-$-curve $\gamma(t)$ in $R^-_v$ has an asymptotic Harder--Harasimhan filtration for $\lambda_{\alpha\,\beta,s}$-stability.
\end{remark}

The following will be useful in Section~\ref{sec:instantons}. 

\begin{proposition}\label{tilt_wall_above}
Let $E$ be a Gieseker stable sheaf with $\ch(E)=v$, and let $W_{u,v,s}$ be an actual $\lambda$-wall in $R^-_v$ crossing $\Gamma^-_{v,s}$. Then there is an actual $\lambda$-wall which either crosses the $\beta$-axis between $\Gamma^-_{v,s}$ and $\Theta^-_{v}$ or cuts $\Theta^-_{v}$. In particular, in the latter case, there is an actual $\nu$-wall for $E$.
\end{proposition}

\begin{proof}
Let $P$ denote the point where $W_{u,v,s}$ cuts $\Gamma^-_{v,s}$.  By Lemma~\ref{lem:actual wall}, $W_{u,v,s}$ ends either on another actual $\lambda$-wall which must remain above the original wall or on $\alpha=0$. So we have a piecewise path of actual $\lambda$-walls in the region between $\Gamma^-_{v,s}$ and $\Theta^-_v$. By Theorem~\ref{one_lambda_wall}, this path cannot cross $\Gamma^-_{v,s}$ again except at $P$, but then there is a loop in $R^-_{v}$ intersecting $\Gamma^-_{v,s}$ at $P$ outside of which $E$ is unstable. But then $E$ would be unstable on both sides of $P$, contradicting the fact that $W_{u,v,s}$ is an actual wall. The path cannot be unbounded in this region because unbounded curves are only unbounded in $R^0_{v}$. So it must cross either $\Theta^-_{v}$ or the $\beta$-axis in this region.

For the last part, observe that the final segment of the path crossing $\Theta^-_{v}$ is an actual $\lambda$-wall, and so by Theorem~\ref{schmidt}, there is an actual $\nu$-wall crossing $\Theta^-_{v}$ at the same point.  
\end{proof}

\subsection{Asymptotics along $\boldsymbol{\Lambda_{\oalpha}^-}$}

Next, we study objects that are asymptotically $\labs$-semistable along a horizontal line $\{\alpha=\oalpha\}$, or \emph{asymptotically $\lambda_{\oalpha,\beta,s}$-semistable}. The key is to observe that $\Lambda^-_{\oalpha}$ is eventually in $\leftgamma_{v,s}$, which is the region to the left of $\Gamma^-_{v,s}$ (see Definition~\ref{RL}).  Then we can use Theorem~\ref{RLvs}. We establish the following result.

\begin{theorem}\label{t:asymp-h}
Let $v$ be a Chern character of an object of $\dbx$ satisfying $v_0\ne0$ and the Bogomolov--Gieseker inequality \eqref{B-ineq}. For each $s\geq 1/3$ and each $\oalpha>0$, an object $A\in\dbx$ is asymptotically $\labs$-semistable along $\Lambda_{\oalpha}^-$ if and only if $A$ is a Gieseker semistable sheaf.
\end{theorem}

The strategy of the proof is very similar to the one used in Theorem~\ref{t:asymtotegamma}, though with different calculations; we include them here for the sake of completeness. Indeed, note that
\begin{equation}\label{lim-nu-h}
\lim_{\beta\to-\infty}\frac{-1}{\beta}\nu_{\oalpha,\beta}(F) = 
\begin{cases}
1/2 &\text{if } \ch_0(F)\ne0, \\
1 &\text{if } \ch_0(F)=0\text{ and }\ch_1(F)\ne0.
\end{cases}
\end{equation}

As before, the claim follows from Proposition~\ref{dual_in_A}\eqref{dual_in_A1}; alternatively, it can also be proved in the same manner as Lemma~\ref{gammaasymp} above.

\begin{lemma}\label{l1:asymp-h}
Fix $\oalpha>0$. If there is a $\beta_0<0$ such that $E\in\mathcal{A}^{\oalpha,\beta}$ for every $\beta<\beta_0$, then $E\in\coh(X)$.
\end{lemma}

The next step is to understand the difference in $\lambda$-slopes; note that
\begin{equation}\label{lambda-h}
\labs(F)-\labs(E) = \dfrac{f_{u,v}(\alpha,\beta)}{\rho_u(\alpha,\beta)\rho_v(\alpha,\beta)}
\end{equation}
with $u=\ch(F)$ and $v=\ch(E)$. The full expression of the numerator is given by equation \eqref{num-walls}; the numerator is given by
\begin{align*}
& \rho_u(\alpha,\beta)\rho_v(\alpha,\beta) = \dfrac{1}{4}\ch_0(F)\ch_0(E)\beta^4 - \dfrac{1}{2}(\ch_1(F)\ch_0(E)+\ch_1(E)\ch_0(F))\beta^3 \\
& + \left( \dfrac{1}{2}\ch_0(F)\ch_0(E)\alpha^2 + \ch_1(F)\ch_1(E) + \dfrac{1}{2}(\ch_0(F)\ch_2(E)+\ch_2(F)\ch_0(E))\right)\beta^2 \\
  & + (\textrm{ terms of lower order in } \beta).
  \end{align*}

With these formulas at hand, we can establish the \emph{if} part of Theorem~\ref{t:asymp-h}. First we prove a version of Lemma~\ref{asymp_gamma} along $\Lambda^-_{\oalpha}$.

\begin{lemma}\label{Lambda_asymptotic}
For any fixed $s>0$ and $\oalpha>0$, if $E$ is a Gieseker stable sheaf, then there is no monomorphism $F\into E$ in $\mathcal{A}^{\oalpha,\beta}$ with $\lambda_{\oalpha,\beta,s}(F)\geq\lambda_{\oalpha,\beta,s}(E)$ for all $\beta<\obeta$.
\end{lemma}

\begin{proof}
If $F\into E\onto G$ is a short exact sequence in $\mathcal{A}^{\oalpha,\beta}$ for every $\beta<\beta_0$, then Lemma~\ref{l1:asymp-h} implies that both $F$ and $G$ are sheaves. We first have that
$$
\lim_{\beta\to-\infty} \left( \lambda_{\oalpha,\beta,s}(F)-\lambda_{\oalpha,\beta,s}(E) \right) = \dfrac{1}{3}\dfrac{\delta_{10}(F,E)}{\ch_0(F)\ch_0(E)}\le0
$$
since $\delta_{10}(F,E)\le0$ by hypothesis. If $\delta_{10}(F,E)=0$, then also
$$ \lim_{\beta\to-\infty} (-\beta)(\lambda_{\oalpha,\beta,s}(F)-\lambda_{\oalpha,\beta,s}(E)) = \dfrac{4}{3}\dfrac{\delta_{20}(F,E)}{\ch_0(F)\ch_0(E)}\le0$$
because $\delta_{20}(F,E)\le0$. Finally, if $\delta_{20}(F,E)$ also vanishes, then
$$ \lim_{\beta\to-\infty} \beta^2(\lambda_{\oalpha,\beta,s}(F)-\lambda_{\oalpha,\beta,s}(E)) = 4\dfrac{\delta_{30}(F,E)}{\ch_0(F)\ch_0(E)}\le0$$ since $\delta_{30}(F,E)\le0$, with equality holding if and only if $\delta_{30}(F,E)=0$ as well, which implies that $\lambda_{\oalpha,\beta,s}(F)=\lambda_{\oalpha,\beta,s}(E)$ for every $\oalpha$, $\beta$ and $s$.  It follows that $E$ is asymptotically $\lambda_{\oalpha,\beta,s}$-(semi)stable, as desired.
\end{proof}

Now we extend Theorem~\ref{one_lambda_wall} to $\Lambda^-_{\oalpha}$. The trick is to reduce to the case of $\Gamma^-_{v,s}$.

\begin{lemma}\label{Lambda_finite}
Suppose $s\geq1/3$, and suppose $E$ is a Gieseker semistable sheaf with $\ch(E)=v$. Then there is a $\obeta$ such that for all $\beta<\obeta$, there are no actual $\lambda$-walls intersecting $\Lambda^-_{\oalpha}$. Furthermore, there are no more actual $\lambda$-walls to the left of the actual $\nu$-wall for $E$ in $R^-_{v}$.
\end{lemma}

\begin{proof}
Suppose $W_{u,v,s}$ is an actual $\lambda$-wall intersecting $\Lambda^-_{\oalpha}$ at a point $P=(\oalpha,\beta_0)$. By the choice of $\obeta$, we can assume $P\in R^L_{v,s}$.  By Theorem~\ref{create wall}. any numerical $\lambda$-wall $W_{u,v,s}$ for $s\ge1/3$ remains a numerical $\lambda$-wall at $s=1/3$, and, since it is bounded, each component must intersect $\Gamma^-_{v,1/3}$ at its maxima. Theorem~\ref{one_lambda_wall} tells us that this wall must destabilize to the right. If $W_{u,v,1/3}$ is not actual as it crosses $\Gamma^-_{v,1/3}$, then by Lemma~\ref{lem:actual wall}, it ends on another actual $\lambda$-wall. Repeating, we have a piecewise path of walls $W^{(i)}_{u,v,1/3}$ which must all destabilize to the right, and so the path has increasing $\alpha$. Then it must cross $\Gamma^-_{v,1/3}$. But the actual $\lambda$-wall component $W^{N}_{v,1/3}$ crossing $\Gamma^-_{v,1/3}$ must be the unique wall by Theorem~\ref{one_lambda_wall}. So there cannot be another such wall to the left of $P$. Hence, we can let $\obeta=\beta_0$.

For the last sentence, observe that if the wall crosses $\Lambda^-_{\oalpha}$ again, then there must be a piecewise smooth curve of actual $\lambda$-walls which intersects $\Gamma^-_{v,s}$ at a point $Q$, say. This can only happen if there is a $\nu$-wall above $Q$.
\end{proof}

These two lemmas then immediately imply the following. 

\begin{lemma}\label{l3:asymp-h}
For any $s\ge1/3$ and $\oalpha>0$, if $E$ is a Gieseker $($semi\/$)$stable sheaf, then $E$ is asymptotically $\lambda_{\oalpha,\beta,s}$-$($semi\/$)$stable along $\Lambda^-_{\oalpha}$.
\end{lemma}

We now consider the converse.

\begin{lemma}\label{l2:asymp-h}
For any $s>0$ and $\oalpha>0$, if $E$ is asymptotically $\lambda_{\oalpha,\beta,s}$-$($semi\/$)$stable along $\Lambda^-_{\oalpha}$, then $E$ is a Gieseker $($semi\/$)$stable sheaf.
\end{lemma}

\begin{proof}
If $E$ is asymptotically $\labs$-(semi)stable along $\alpha=\overline{\alpha}$, then Lemma~\ref{l1:asymp-h} implies that $E$ is a sheaf.

If $E$ is not torsion-free, let $F\into E$ be its maximal torsion subsheaf. If $T\into F$ is a subsheaf of dimension $1$, then $T\into E$ is a morphism in $\cala^{\oalpha,\beta}$ for $\beta\ll0$ since $T\in\tors{\alpha,\beta}$ for every $(\alpha,\beta)$. We have that $\ch_0(T)=\ch_1(T)=0$; thus
$$ \lim_{\beta\to-\infty} \dfrac{-1}{\beta} \left( \lambda_{\oalpha,\beta,s}(T)-\lambda_{\oalpha,\beta,s}(E) \right) = \dfrac{2}{3}>0 $$
provided $\ch_2(T)\ne0$, contradicting the asymptotic $\lambda$-semistability. If $\ch(T)=(0,0,0,e)$, then $T$
clearly destabilizes $E$ as well.

So now assume $F$ has pure dimension 2, and let $T$ be its maximal $\hat{\mu}$-semistable subsheaf. Remark~\ref{h nu-ss rk0} implies that $T\in\cala^{\oalpha,\beta}$ for $\beta\ll0$; thus $T\into E$ is a morphism in $\cala^{\oalpha,\beta}$ in the same range. Since  $\ch_0(T)=0$ and $\ch_1(T)\ne0$, then
$$ \lim_{\beta\to-\infty} \dfrac{-1}{\beta} \left( \lambda_{\oalpha,\beta,s}(T)-\lambda_{\oalpha,\beta,s}(E) \right) = 
\dfrac{1}{6}>0, $$
again contradicting the asymptotic $\lambda$-semistability. We therefore conclude that $E$ is torsion-free.

If $E$ is not Gieseker semistable, let $F\into E$ be its maximal destabilizing subsheaf. As in the first part of the proof of Lemma~\ref{l3:asymp-h}, we can conclude that $F\in\cala^{\oalpha,\beta}$ for $\beta\ll0$, so $F\into E$ is a morphism in $\cala^{\oalpha,\beta}$ in the same range.

It follows that
$$ \lim_{\beta\to-\infty} \left( \lambda_{\oalpha,\beta,s}(F)-\lambda_{\oalpha,\beta,s}(E) \right) = \dfrac{1}{3}\dfrac{\delta_{10}(F,E)}{\ch_0(F)\ch_0(E)}\ge0; $$
thus $\delta_{10}(F,E)=0$ because $E$ is asymptotically $\lambda_{\oalpha,\beta,s}$-semistable. We then have that
$$ \lim_{\beta\to-\infty} (-\beta)(\lambda_{\oalpha,\beta,s}(F)-\lambda_{\oalpha,\beta,s}(E)) = \dfrac{4}{3}\dfrac{\delta_{20}(F,E)}{\ch_0(F)\ch_0(E)}\ge0;$$
thus again $\delta_{20}(F,E)=0$. Finally, we have that 
$$ \lim_{\beta\to-\infty} \beta^2(\lambda_{\oalpha,\beta,s}(F)-\lambda_{\oalpha,\beta,s}(E)) = 4\dfrac{\delta_{30}(F,E)}{\ch_0(F)\ch_0(E)}\ge0;$$
thus $\delta_{30}(F,E)=0$, meaning that $\ch(F)=\ch(E)$, contradicting the fact that $F$ is a proper subsheaf of $E$. We therefore conclude that $E$ must be Gieseker (semi)stable, as desired.
\end{proof}

\subsection{Asymptotics along general unbounded $\boldsymbol{\Theta^-}$-curves}
We can now extend our results to any unbounded $\Theta^-$-curve (recall Definition~\ref{theta-curve}).

\begin{theorem}\label{t:asymp theta -}
Let $\gamma$ be an unbounded $\Theta^-$-curve, and fix $s\geq1/3$. An object $E\in D^b(X)$ is asymptotically $\lambda_{\alpha,\beta,s}$-$($semi\/$)$stable along $\gamma$ if and only if $E$ is Gieseker $($semi\/$)$stable.
\end{theorem}

\begin{proof}
We consider the stable case first, and we deduce the semistable case by inducting on the lengths of the Jordan--H\"older filtrations.

Suppose $E$ is Gieseker stable. Then Lemma~\ref{Lambda_finite} as $\oalpha$ varies implies that there is at most a single destabilizing piecewise smooth curve of actual $\lambda$-walls from the $\beta$-axis, and it cannot be unbounded in $R^-_{v}$. So the piecewise curve must either cross the $\beta$-axis again or cross $\Theta^-_v$. In either case, there is a $\obeta$ such that $E$ is $\lambda_{\beta,\alpha,s}$-stable in $\{\beta<\obeta\}\cap R^-_v$ and so asymptotically $\lambda$-stable along any curve in this region.

Conversely, suppose that $E$ is asymptotically $\lambda$-stable along $\gamma$ and not Gieseker stable. Then for all $t\gg0$, $E$ is not $\lambda$-stable along $\Lambda^-_{\alpha(t)}$. Let $t_0$ be the least $t\in\mathbb{R}$ such that $E$ is $\lambda$-semistable in $\mathcal{A}^{\gamma(t)}$ for all $t<t_0$. By increasing $t_0$ if necessary, we may assume that for each $t>t_0$, there is some $F_t$ such that $F_t$ destabilizes $E$ at some $\beta=\beta_1<\beta(t)$ along $\Lambda^-_{\alpha(t)}$ and such that $\nu^-_{\gamma(t)}(F_t)>0$ (by the assumption that $\gamma$ is a $\Theta^-$-curve and Lemma~\ref{l34}). So there must be a wall $W_{\ch(E),\ch(F_t),s}$ which crosses the vertical line $\beta=\beta_1$. But that wall must cross either $\Lambda^-_{\alpha(t)}$ at $\beta<\beta_1$ or $\Im(\gamma)$ at $\gamma(t')$ for some $t'>t$. Since it remains an actual $\lambda$-wall, the latter possibility cannot happen by the assumption on $t_0$. The former possibility means that there must be another $F'_t$ which destabilizes at some $(\alpha(t),\beta_2)$ with $\beta_2<\beta_1$. Repeating, we have an infinite sequence of such destabilizers, which contradicts Lemma~\ref{Lambda_finite}.
\end{proof}

\begin{remark}
We expect that the assumption $s\geq1/3$ is unnecessary. But there are instances of walls which vanish as $s$ increases. These could occur to the left of our piecewise destabilizing wall. For any $s$, we would conjecture that there are only finitely many of these, and then there would be a Gieseker chamber containing a finite number of small regions where $E$ is $\lambda$-unstable. 
\end{remark}

We can now deduce asymptotic conditions along unbounded curves in $R^+_{v}$. Using the notion of unbounded $\Theta^-$-curve $\gamma$ and of its dual curve $\gamma^*$ proposed in Definition~\ref{theta-curve}, we state the following. 

\begin{proposition}\label{asymptotic_dual}
Let $\gamma$ be an unbounded $\Theta^-$-curve. Then $E$ is asymptotically $\labs$-$($semi\/$)$stable along $\gamma$ if and only if $E^\vee$ is asymptotically $\labs$-$($semi\/$)$stable along $\gamma^*$.
\end{proposition}

\begin{proof}
By Proposition~\ref{dual_in_A}, we know that there is a $t_1$ such that $E^\vee\in\mathcal{A}^{\gamma^*(t)}$ for all $t>t_1$. If $F\into E$ destabilizes $E$ for all $t>t_0$, then the proposition also shows that there is some $t_2$ such that $F^\vee$ and $E^\vee$ are in $\mathcal{A}^{\gamma^*(t)}$ for all $t>t_2$ and also $E^\vee\to F^\vee$ surjects in $\mathcal{A}^{\gamma^*(t)}$. But $\lambda_{\gamma(t),s}(F)=-\lambda_{\gamma^*(t),s}(F^\vee)$.
\end{proof}

Putting Theorem~\ref{t:asymp theta -} and Proposition~\ref{asymptotic_dual} together, we immediately deduce the following statement.

\begin{theorem}\label{t:asymp-h+}
Let $v$ be a numerical Chern character satisfying $v_0\ne0$ and $Q^{\rm tilt}(v)\ge0$, and fix $\oalpha>0$. For each $s\geq1/3$, an object $A\in\dbx$ with $\ch(A)=v$ is asymptotically $\labs$-$($semi\/$)$stable along an unbounded $\Theta^+$-curve if and only if $A^\vee$ is a Gieseker $($semi\/$)$stable sheaf.
\end{theorem}

Since Theorem~\ref{t:asymtotegamma} holds for every $s>0$, we have a stronger statement for the asymptotics along $\Gamma^+_{v,s}$.

\begin{theorem}\label{t:asymtotegammaplus}
Let $v$ be a numerical Chern character satisfying $v_0\ne0$ and $Q^{\rm tilt}(v)\ge0$. For each $s>0$, an object $A\in\dbx$ is asymptotically $\labs$-$($semi\/$)$stable along $\Gamma^+_{v,s}$ if and only if $A^\vee$ is a Gieseker $($semi\/$)$stable sheaf.
\end{theorem}

Note that Proposition~\ref{a=dual} provides an explicit characterization of objects $A\in\dbx$ that are dual to torsion-free sheaves.

\section{Examples of classifying walls: ideal sheaves and null correlation sheaves}\label{sec:examples}

We complete this paper by studying some concrete examples of actual $\lambda$-walls and asymptotically $\labs$-stable objects for $X=\p3$ which illustrate many of the results established above.

\subsection{Ideal sheaves of lines in $\boldsymbol{\p3}$} \label{ideal line}

Let $L\subset\p3$ be a line, and let $I_{L}$ denote its ideal sheaf; recall that $v:=\ch(I_{L})=(1,0,-1,1)$. The curves $\Theta_v$ and $\Gamma_{v,s}$ are given by
\begin{align*}
\Theta_v &:~~ \beta^2 - \alpha^2 = 2, \\
\Gamma_{v,s} &:~~ \left( s+\dfrac{1}{6} \right) \beta\alpha^2 -\dfrac{\beta^3}{6} + \beta = - 1.
\end{align*}
Note that for every $s>0$, the curve $\Gamma_{v,s}$ intersects $\Theta_v^-$  at a single point, call it $P_s$, whose $\beta$-coordinate is the real root of the polynomial $3s\beta^3+(2-6s)\beta+3=0$; the case $s=1/3$ is pictured in Figure~\ref{line-fig-1}. 

Theorem~\ref{thm for q} guarantees the existence of vanishing $\nu$- and $\lambda$-walls, which are precisely given by the triangle
$$ \op3(-1)^{\oplus 2} \to I_{L} \to \op3(-2)[1]. $$
We simplify the notation and use $\Xi:=\Xi_{u,v}$ and $\Upsilon_s:=\Upsilon_{u,v,s}$, where $u=\ch(\op3(-1))$, for the vanishing $\nu$- and $\lambda$-walls just described, respectively.

Both curves cut $\Theta_v^-$ at the point $R:=\Theta_v\cap\Xi\cap\Upsilon_s=(\alpha=1/2,\beta=-3/2)$. In addition, the vanishing $\lambda$-wall crosses $\Gamma_{v,s}^-$ at the point $Q_s:=\Gamma_{v,s}^-\cap\Upsilon_s=(1/\sqrt{6s+1},-2)$.  

There are no other actual $\nu$-walls for the numerical Chern vector $v=(1,0,-1,1)$. To see that, suppose $A\to I_{L}$ is a $\nuab$-stable destabilizing object in $\cohb$; then it must destabilize along $\beta=-\sqrt{2}$ and also on $\Theta_v^-$. Let $\ch_{\le2}(A)=(r,-c,d)$; then $r>0$, and the condition that both $A$ and $I_{L}/A$ are in $\cohb$ gives us
\begin{align*}
-c-\beta r&<0,\\
c-\beta(1-r)&\leq0,
\end{align*}
and so $-\beta r<c\leq\beta(1-r)$. For $\beta=-\sqrt{2}$, we see that $c=r$ is the only solution. Then the first inequality implies $r<\beta/(1+\beta)=\sqrt{2}/(\sqrt{2}-1)$ and so $0<r\leq 3$. Now choosing $\beta$ on $\Theta_v^-$, we have $d+\beta c+r=0$ in order for $A$ to destabilize $I_{L}$. So $d=-r(\beta+1)>r(\sqrt{2}-1)$. On the other hand, the Bogomolov--Gieseker inequality for $A$ gives $r^2\geq 2rd$ and so $d\leq r/2$. Then $d=r/2$ as it must be in $\mathbb{Z}[1/2]$.

\begin{figure}[ht] \centering
\includegraphics{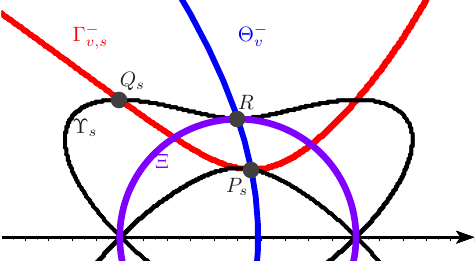}
\caption{The vanishing $\nu$- and $\lambda$-walls (in purple and black, respectively) for the ideal sheaf of a line. Note that the $\lambda$-wall is connected but possesses two irreducible components. Both curves intersect $\Theta_v$ (in blue) and $\Gamma_{v,s}^-$ (in red).} \label{line-fig-2}
\end{figure}

Since $\op3(-2)$ and $\op3(-1)$ are $\labs$-stable for all $(\alpha,\beta,s)$, it follows that this is the only actual $\lambda$-wall destabilizing ideal sheaves of lines.

\subsection{Ideal sheaves of points in $\boldsymbol{\p3}$} \label{ideal point}

Let $p\in\p3$ be a point, and let $I_{p}$ denote its ideal sheaf; recall that $v:=\ch(I_{p})=(1,0,0,-1)$. The curves $\Theta_v$ and $\Gamma_{v,s}$ are given by
\begin{align*}
\Theta_v &:~~ \beta^2 - \alpha^2 = 0, \\
\Gamma_{v,s} &:~~ \left( s+\dfrac{1}{6} \right)\beta\alpha^2 - \frac{\beta^3}{6}=1.
\end{align*}
Note that $\Gamma_{v,s}$ does not intersect $\Theta_v^-$. There are no actual $\nu$-walls
because $\ch_{\leq2}(I_p)=\ch_{\leq2}(\op3)$ and $I_p\to\op3$ injects in $\cohab\cap\cohb\cap\coh(\p3)$.

The resolution of $I_{p}$
\[0\to\op3(-3)\to \op3(-2)^{\oplus 3}\to \op3(-1)^{\oplus 3}\to I_{p}\to 0\]
induces two triangles. The first one defines a reflexive sheaf $S_p$ with
$u:=\ch(S_p[1])=(-2,3,-3/2,-1/2)$ whose only singularity is precisely the point $p$ by splitting the
resolution in the middle. Then we have triangles
\begin{equation}\label{S_p}
\op3(-2)^{\oplus3}[1] \to S_p[1] \to \op3(-3)[2].
\end{equation} 
The second one relates the sheaf $S_p$ with $I_{p}$:  
\begin{equation}\label{I_p}
\op3(-1)^{\oplus 3} \to I_{p} \to S_p[1].
\end{equation}
This provides one wall which cuts $\Gamma_{v,s}$ at $(1/\sqrt{6s+1},-2)$. There is another wall which crosses at the same point, given by an object $D$ which is not a sheaf (or a shift of a sheaf) with $u':=\ch(D)=(0,3,-9/2,7/2)$, 
\[D\to I_p\to \op3(-3)[2],\]
where $D$ is the cone on the morphism $\op3(-2)^{\oplus3}\to\op3(-1)^{\oplus3}$ in the
resolution of $I_p$. This gives rise to a fourth triangle:
\[\op3(-1)^{\oplus 3}\to D\to \op3(-2)^{\oplus 3}[1].\]
Note that this provides an example in Proposition~\ref{sheaf_subobjects} where $D\not\in\coh(X)$, and note that we have $E/D\in\coh(\op3)[2]\cap\cohb[1]$.  The objects $\op3(-3)[2]$, $\op3(-2)[1]$ and $\op3(-1)$ all belong to $\cohab$ for every $(\alpha,\beta)$ within a square with vertices at the points $(\alpha=0,\beta=-2)$, $(\alpha=1/2,\beta=-3/2)$, $(\alpha=1,\beta=-2)$, $(\alpha=1/2,\beta=-5/2)$. Consequently, this also follows for $S_p$ and $D$.

Note that at the special point $P$, $\labs(\op3(-1))=\labs(\op3(-2))=\labs(\op3(-3))=\labs(D)=\labs(S_p)=\labs(I_p)$.

Within this region, the two triangles induce actual $\lambda$-walls $\Upsilon_{w,u,s}$, where $w:=\ch(\op3(-2)[1])$, and $\Upsilon_{u,v,s}$.  One can check that $\Upsilon_{w,u,s}$ and $\Upsilon_{u,v,s}$ intersect at the point $(\alpha=1/\sqrt{6s+1},\beta=-2)$, which also belongs to $\Gamma_{v,s}^-$, for every $s$, implying that the sheaf $I_{p}$ is a $\labs$-semistable object at that point; see Figure~\ref{fig pt}. It is easy to see that $\Ext^1(D,\op3(-3)[2])=0=\Ext^1(\op3(-1)^{\oplus3},S_p)$, and so the walls are vanishing.

\begin{figure}[t] \centering
\includegraphics{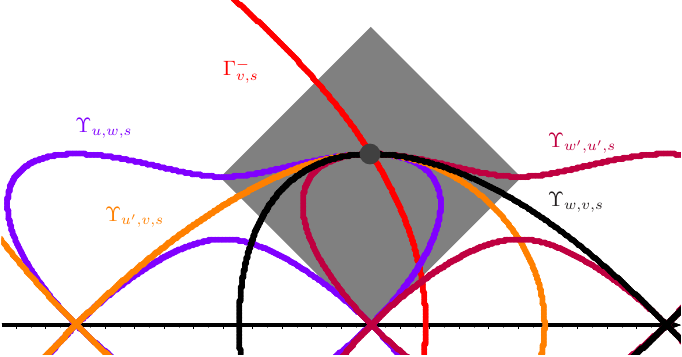}
\caption{This graph shows the numerical $\lambda$-walls $\Upsilon_{u,w,s}$ (in blue), $\Upsilon_{w',u',s}$ (in purple, where $w'=\ch(\op3(-1))$), $\Upsilon_{w,v,s}$ (in black) and  $\Upsilon_{u',v,s}$ (in orange), and the curve $\Gamma_{v,s}^-$ (in red) crossing at the point $(\alpha=1/\sqrt{3},\beta=-2)$; we set $s=1/3$. The shaded region marks where the objects $O(-3)[2]$, $O(-2)[1]$ and $O(-1)$, and consequently also $S_p[1]$, $D$ and $I_{p}$, belong to $\cohab$.}\label{fig pt}
\end{figure}

It follows that there exists an open set $V\subset\HH$ containing the point $(\alpha=1/\sqrt{6s+1},\beta=-2)$ in its boundary such that $I_{p}$ is a $\labs$-stable object for every $(\alpha,\beta)\in V$. Moreover, $\Upsilon_{u,v,s}$ is a vanishing $\lambda$-wall for $I_{p}$ for $\beta>-2$, and $\Upsilon_{u'v,s}$ is the vanishing wall for $\beta<-2$. The resulting wall is $C^1$ on $\Gamma_{v,1/3}^-$ but only $C^0$ for other values of $s$.

Since $I_{p}$ is also asymptotically $\labs$-stable along $\Gamma_{v,s}^-$, it follow that $I_{p}$ is $\labs$-stable for every $(\alpha,\beta)\in\Gamma_{v,s}^-$ with $\beta<-2$ (that is, no actual $\lambda$-wall crosses $\Gamma_{v,s}^-$ for $\beta<-2$).

\subsection{Torsion-free sheaves with Chern character $\mathbf{(2,0,-1,0)}$} \label{sec:(2,0,-1,0)}

Our first step towards the classification of $\labs$-semistable objects with Chern character $v=(2,0,-1,0)$ on $\p3$ near the curve $\Gamma_{v,s}^-$ is the classification of $\mu$-semistable torsion-free sheaves with this Chern character.

Below, $p$ and $L$, respectively, denote a point and a line in $\p3$. Recall also that a torsion-free sheaf $E$ on $\p3$ is called a \emph{null correlation sheaf} if it sits in the exact sequence
$$ 0 \to \op3(-1) \to \Omega^1_{\p3}(1) \to E \to 0; $$
see \cite{Ein}. Note that $\ch(E)=(2,0,-1,0)$. A locally free null correlation sheaf is called a null correlation bundle; non--locally free
null correlation sheaves satisfy a sequence of the form 
$$ 0\to E \to \op3^{\oplus 2} \to \mathcal{O}_{L}(1) \to 0 $$
for some line $L\subset\p3$.

\begin{proposition}\label{(2,0,-1,0)}
Let $E$ be a $\mu$-semistable torsion-free sheaf on $X=\p3$ with $\ch(E)=(2,0,-1,0)$.
\begin{enumerate}
\item If $E$ is locally free, then $E$ is a null correlation bundle; in particular, $E$ is $\mu$-stable.
\item\label{P8.1-2} If $E$ is properly torsion-free, then $E$ is strictly $\mu$-semistable, and it is given by one of the following extensions:
  \begin{enumerate}[label={\rm(\alph*)},ref=\alph*]
\item\label{P8.1-21} $0\to I_{L} \to E \to I_{p} \to 0$ for $p\in L$ with non-trivial extension (in particular, $E$ is a null correlation sheaf, and it is $\mu_{\le2}$-stable), 
\item\label{P8.1-22}  $0\to I_{p} \to E \to I_{L} \to 0$ for arbitrary $p$ and $L$ (in particular, $E$ is not $\mu_{\le2}$-semistable, and it has no global sections),
\item \label{P8.1-23} $0\to \op3 \to E \to I_{\tilde{L}} \to 0$, where $\tilde{L}$ is a $1$-dimensional scheme satisfying the sequence $0\to\mathcal{O}_p\to\mathcal{O}_{\tilde{L}}\to\mathcal{O}_L\to0$ for arbitrary $p$ and $L$ (in particular, $E$ is not $\mu_{\le2}$-semistable, and it has a global section).
\end{enumerate}
\end{enumerate}
\end{proposition}

Note that ${\rm Ext}^1(I_{L},I_{p})=0$ when $p\notin L$.

\begin{proof}
  If $E$ is locally free, then \cite[Lemma 2.1]{Chang} implies that $E$ is $\mu$-stable; the fact that every $\mu$-stable rank $2$ bundle $E$ with $c_2(E)=1$ on $\p3$ is a null correlation bundle is proved in \cite[4.3.2, p.~363]{OSS}.

If $E$ is not locally free, then $E^{**}$ is a $\mu$-semistable rank $2$ reflexive sheaf with $\ch_1(E^{**})=0$, and either $\ch_2(E^{**})=0$ or $\ch_2(E^{**})=1$; in both cases, $E^{**}$ is strictly $\mu$-semistable (\textit{cf.} \cite[Lemma 2.1]{Chang}), so $E$ is strictly $\mu$-semistable. The first case forces $E^{**}=\op3^{\oplus 2}$, while, in the second case, $E^{**}$ must be a properly reflexive sheaf $S_L$ given by the sequence
\begin{equation}\label{defn-s}
0 \to \op3 \to S_L \to I_{L} \to 0 ;
\end{equation}
note that $\ch(S_L)=(2,0,-1,1)$.

We start by analysing the first case, that is, $E^{**}=\op3^{\oplus 2}$. We have that $\ch(Q_E)=(0,0,1,0)$, where $Q_E:=E^{**}/E$; note that $h^0(E)=0,1$. Again, two possibilities follow: either $Q_E$ has pure dimension 1, in which case $Q_E=\mathcal{O}_L(1)$, or $Q_E$ satisfies a sequence of the form
$$ 0 \to Z \to Q_E \to O_L(k) \to 0, $$
where $Z$ is a $0$-dimensional sheaf and $k-1=-h^0(Z)<0$ because $\ch_3(O_L(k))=-\ch_3(Z)$; since we have $h^0(O_L(k))>0$, we must have $k=0$ and thus $Z=\mathcal{O}_p$.

The first possibility, namely $Q_E=\mathcal{O}_L(1)$, leads to the sequence
$$ 0\to E \to \op3^{\oplus 2} \to \mathcal{O}_{L}(1) \to 0 $$
and therefore yields the sheaves described in item~\eqref{P8.1-2}\eqref{P8.1-21} since $I_{L}$ is the kernel of the composition \mbox{$\op3\to\op3^{\oplus 2}\onto\mathcal{O}_{L}(1)$.} In addition, notice that these are the null correlation sheaves. Checking that these sheaves are $\mu_{\ge2}$-stable is a simple exercise.

The second possibility leads to the sequence
$$ 0\to E \to \op3^{\oplus 2} \to \mathcal{O}_{\tilde{L}} \to 0 $$
and therefore yields the sheaves described in item~\eqref{P8.1-2}\eqref{P8.1-23}.

Finally, if $E^{**}=S_L$, then $Q_E=\mathcal{O}_p$, and the sequence $0\to E \to S_L \to \mathcal{O}_p\to0$ together with the sequence \eqref{defn-s} lead to the sequence in item~\eqref{P8.1-2}\eqref{P8.1-22}.
\end{proof}

It is important to notice that non-trivial extensions like the ones in items~\eqref{P8.1-2}\eqref{P8.1-22} and~\eqref{P8.1-2}\eqref{P8.1-23} of Proposition~\ref{(2,0,-1,0)} do exist. Let us first consider the extension of an ideal sheaf of a line by an ideal sheaf of a point.

\begin{lemma}\label{ext1_point_line}
If $p$ is a point and $L$ is a line in $\p3$, then
$$ \dim\Ext^1(I_L,I_p) = \left\{
\begin{array}{l}
  3 \quad\text{if}~ p\notin L, \\
  5 \quad\text{if}~ p\in L. \end{array} \right. $$
\end{lemma}

\begin{proof}
First, apply the functor $\Hom(-,I_p)$ to the exact sequence $0\to I_L\to\op3\to\mathcal{O}_L\to0$ to conclude that $\Ext^1(I_L,I_p)\simeq \Ext^2(\mathcal{O}_L,I_p)$. Applying the functor $\Hom(\mathcal{O}_L,-)$ to the exact sequence $0\to I_p\to\op3\to\mathcal{O}_p\to0$, we obtain
\begin{equation}\label{ext-lp}
0\to \Ext^1(\mathcal{O}_L,\mathcal{O}_p) \to \Ext^2(\mathcal{O}_L,I_p) \to
H^1(\mathcal{O}_L(-4)) \to \Ext^1(\mathcal{O}_L,\mathcal{O}_p) \to 0.   
\end{equation}

Note that $\Ext^i(\mathcal{O}_L,\mathcal{O}_p)=H^0(\inext^i(\mathcal{O}_L,\mathcal{O}_p))$. If $p\notin L$, then $\inext^i(\mathcal{O}_L,\mathcal{O}_p)=0$ for every $i\ge0$; thus $\Ext^1(I_L,I_p)\simeq H^1(\mathcal{O}_L(-4))$, which completes the proof of the first claim.

If $p\in L$, then one can check that
$$ h^0(\inext^i(\mathcal{O}_L,\mathcal{O}_p))= \left\{
\begin{array}{l} 2 \quad\text{for}~ i=1, \\ 1 \quad\text{\rm for}~ i=2. \end{array} \right. $$
Comparing with the exact sequence in display \eqref{ext-lp}, we obtain the second part of the statement.
\end{proof}

Next we consider the sheaves in item~\eqref{P8.1-2}\eqref{P8.1-23} of Proposition~\ref{(2,0,-1,0)}.

\begin{lemma}\label{ext1_Ltilde}
If $\,\tilde{L}$ is the $1$-dimensional scheme described in item~\eqref{P8.1-2}\eqref{P8.1-23} of Proposition~\ref{(2,0,-1,0)}, then we have \mbox{$\dim\Ext^1(I_{\tilde{L}},\op3)=1$.}
\end{lemma}

\begin{proof}
Note that
$$ \Ext^1(I_{\tilde{L}},\op3)\simeq H^2(I_{\tilde{L}}(-4))^*\simeq H^2(I_L(-4))^*, $$
where the first isomorphism is given by Serre duality, and the second follows from the sequence $0\to I_{\tilde{L}} \to I_L \to \mathcal{O}_p\to0$.

The resolution of the ideal sheaf $I_L$ yields the cohomology sequence
$$ 0\to H^2(I_L(-4)) \to H^3(\op3(-6)) \to H^3(\op3(-5))^{\oplus 2} \to H^3(I_L(-4)) \to 0. $$
However, $H^3(I_L(-4))\simeq\Hom(I_L,\op3)^*$; thus $h^2(I_L(-4))=1$.
\end{proof}

\bigskip

Now let $\mathcal{G}$ denote the Grassmanian of lines in $\p3$. Recall that a line $L\in\mathcal{G}$ is called a \emph{jumping line} for a $\mu$-semistable torsion-free sheaf $E$ with $\ch_1(E)=0$ on $\p3$ if $E\otimes\mathcal{O}_L=\mathcal{O}_L(-a)\oplus\mathcal{O}_L(a)$ for some $a>0$; we denote by $\mathcal{J}(E)$ the set of jumping lines for $E$, as a subset of $\mathcal{G}$.

If $E$ is a null correlation sheaf, then $\mathcal{J}(E)$ is a divisor of degree 1 in $\mathcal{G}$. Therefore, for each null correlation sheaf $E$ and $L\in\mathcal{J}(E)$, there exists an epimorphism $E\onto\mathcal{O}_L(-1)$ whose kernel is a Gieseker semistable sheaf~$K$ with $\ch(K)=(2,0,-2,2)$.

Such sheaves have been previously considered by Mir\'o-Roig and Trautmann in \cite[Section 1.5]{MRT}; they showed that the family of sheaves $K$ defined by exact sequences of the form
$$ 0 \to K \to E \to \mathcal{O}_L(-1) \to 0, $$
where $E$ is a null correlation sheaf and $L\in\mathcal{J}(E)$, define a locally closed $8$-dimensional subscheme of the full moduli space of Gieseker semistable sheaves on $\p3$ with Chern character equal to $(2,0,-2,2)$. In addition, this moduli space, which we will simply denote by $\mathcal{Z}$, is an irreducible projective variety of dimension 13, and every such sheaf satisfies an exact sequence of the form 
\begin{equation} \label{k-sqc}
0 \to \op3(-2)^{\oplus 2} \stackrel{A}{\to} \op3^{\oplus 4}(-1) \to K \to 0,
\end{equation}
with $A$  a $4\times2$ matrix of linear forms in four variables. Combining this with the presentation of $\mathcal{O}_L(-1)$
\[0\to\op3(-3)\to\op3(-2)^{\oplus 2}\to\op3(-1)\to\mathcal{O}_L(-1)\to0, \]
we have a presentation for any Gieseker stable sheaf of Chern character $(2,0,-1,0)$:
\begin{equation}\label{Epres}
0\to\op3(-3)\to\op3(-2)^{\oplus 4}\to\op3(-1)^{\oplus 5}\to E\to 0.
\end{equation}

Our next goal in this section is the classification of $\nuab$-stable objects for the Chern character $v=(-2,0,1,0)$. First, we show that there are no $\nu$-walls.

\begin{lemma}\label{N1_no_tilt}
For $v=\pm(2,0,-1,k)$, there are no actual $\nu$-walls for any $k$.
\end{lemma}

\begin{proof}
  We first suppose $\beta<0$. Then we may assume $E$ is a Gieseker semistable sheaf (to find the biggest wall). Suppose $F$ is a sub-object of $E$ in $\cohb$ with $\nu_\beta(F)=\nu_\beta(E)$ and $\ch(F)=(r,x,y/2,z/6)$. Then $r\geq1$ and $x\leq0$. By \cite[Corollary~2.8]{M}, every $\nu$-wall must cross $\beta=\mu(E)-\sqrt{\Delta(E)/r(E)^2}=-1$ (at the bottom of $\Theta_v$), and so $0<x+r<2$ and then $x=1-r$.  We look for $\nu$-walls along $\Theta^-_v$. Then $\alpha^2=\beta^2-1$. Now $\ch_2^{\alpha,\beta}(F)=y/2-\beta (1-r)+r/2=0$. So $y=2\beta(1-r)-r$. If $r\geq2$, then $y\geq r-1$ as $\beta<-1$. From the Bogomolov--Gieseker inequality  for $F$, we have $(r-1)^2\geq ry\geq r(r-1)$, which is impossible. So $r=1$. Then $y=-1$, but $x=0$, and $(1,0,-1/2,z/6)$ is not the Chern character of a sheaf.

Now we assume $\beta>0$. Then $E$ is the dual of a Gieseker semistable sheaf. Suppose $0\to F\to E\to G\to 0$ is a short exact sequence in $\cohb$ with $\nuab(F)=\nuab(E)=\nuab(G)$ for some $\alpha$. If $E$ is a sheaf, then so is $G$, and $G^\vee\in\cohb$ with $\nu_{\alpha,-\beta}(G^\vee)=-\nu_{\alpha,-\beta}(G)$. But then $F^\vee$ is also a sheaf, and so $E^\vee$ is not $\nu_{\alpha,-\beta}$-stable, but this contradicts the last paragraph. If $E$ is not a sheaf, then $\mathcal{H}^0(G)$ is a quotient of $\mathcal{O}_L(1)$, and since we may assume  $G$ is $\nuab$-semistable, it follows that condition~\eqref{a=dual6} of Proposition~\ref{a=dual} holds and $\mathcal{H}^{-1}(G)$ is reflexive (any torsion sheaf $T$ with a map $T[1]\to \mathcal{H}^{-1}(G)\to G$ injects). Since $S=0$ in Proposition~\ref{a=dual}, it follows that $G$ is the dual of a torsion-free sheaf, and again $F^\vee$ is a sheaf. So we still have a contradiction.   
\end{proof}

Lemma~\ref{N1_no_tilt} has two interesting applications. First we recover the following stronger version of a well-known fact.

\begin{proposition}\label{no_mu2}
If $k>0$ is an integer, then there are no $\mu_{\le2}$-semistable torsion-free sheaves of Chern character $(2,0,-1,k)$ on $\p3$.
\end{proposition}

\begin{proof}
Note that $q(v)=8(2-9k^2)<0$. Then if there are $\mu_{\le2}$-semistable torsion-free sheaves, there must exist an actual vanishing $\nu$-wall by Theorem~\ref{thm for q}. But there are no $\nu$-walls.
\end{proof}

Note that $S_L$ in the proof of Proposition~\ref{(2,0,-1,0)} is an example of a $\mu$-semistable reflexive sheaf of Chern character $(2,0,-1,1)$. In addition, let $E$ be a null correlation bundle, and let $\{p_1,\dots,p_k\}$ be distinct points in $\p3$; the kernel of an epimorphism $E\onto \oplus_{i=1}^k\mathcal{O}_{p_i}$ provides an example of a $\mu$-stable torsion-free sheaf with Chern character $(2,0,-1,-k)$ for any $k>0$.

One can also provide a complete description of $\nuab$-stable objects with Chern character $(2,0,-1,0)$.

\begin{proposition}
Given any $\alpha>0$, an object $B\in\cohb(\p3)$ with $\ch(B)=(2,0,-1,0)$ is $\nuab$-stable if and only if
\begin{enumerate}
\item $B$ is a null correlation sheaf, when $\beta<0$;
\item $B^\vee[-1]$ is a null correlation sheaf, when $\beta>0$.
\end{enumerate}
\end{proposition}

\begin{proof}
Theorem~\ref{a nu sst} and Proposition~\ref{(2,0,-1,0)} tell us that the asymptotically $\nuab$-stable objects on each side of the $\alpha$-axis are precisely the ones described above. Since there are no actual $\nu$-walls, such objects are $\nuab$-stable everywhere.
\end{proof}

Note that it also follows from Proposition~\ref{no_mu2} that if $k>0$, then there are no $\nu_{\alpha,\beta}$-objects with Chern character $(2,0,-1,k)$. We can also deduce that from the fact that $Q_{\alpha,\beta}({2,0,-1,k})=4+4\alpha^2+4\beta^2+12\beta k$, which vanishes for $k>0$ in the region $\beta<0$; since there are no non-vertical $\nu$-walls, this is impossible.

\subsection{$\boldsymbol{\lambda}$-walls and stability for the Chern character $\mathbf{(2,0,-1,0)}$}
\label{sec:instantons}

We now turn to the classification of actual $\lambda$-walls corresponding to Gieseker stable
objects and the description of the $\labs$-semistable objects with Chern character
$v=(2,0,-1,0)$ near the curve $\Gamma^-_{v,s}$. 

Note that the fact that both $I_{L}$ and $I_{p}$ are $\lambda_{1/\sqrt{6s+1},-2,s}$-semistable for every $s>0$, checked in Sections~\ref{ideal line} and~\ref{ideal point}, implies that the sheaves of described in items~\eqref{P8.1-2}\eqref{P8.1-21} and~\eqref{P8.1-2}\eqref{P8.1-22} of Proposition~\ref{(2,0,-1,0)} are also $\lambda_{1/\sqrt{6s+1},-2,s}$-semistable for every $s>0$. For this case, we will set $s=1/3$ for simplicity.

We first observe that there is a vanishing wall through the point $P$ given by $\beta=-2$ and $\alpha^2=1/3$ on $\Gamma^-_{v,1/3}$ given by $F=\op3(-1)$. Theorem~\ref{one_lambda_wall} and Lemma~\ref{N1_no_tilt} then imply that this is the largest actual $\lambda$-wall for such Gieseker stable sheaves, and it suffices to find all of the walls through $P$. So in what follows, we set $\beta=-2$ and $\alpha^2=1/3$. One strategy to find all of the walls would be to guess some based on prior knowledge, and then check whether both the sub-object $F$ and the quotient object $G$ are semistable, thus guaranteeing that the suspected wall is an actual one. However, we shall try to deduce the walls by solving the inequalities to locate solutions directly to illustrate the techniques one might use to do this.

For $E$ with $\ch(E)=(2,0,-1,0)$, we consider a short exact sequence in $\cohab$
\[0\to F\to E\to G\to 0\]
with $\lambda_{\alpha,\beta}(F)=\lambda_{\alpha,\beta}(E)$ along $\Gamma^-_{v,s}$. We let $\ch(F)=(r,x,y/2,z/6)$ for integers $r,x,y,z$. Since $\chi(F)=r+11x/6+y+z/6$ is an integer, we must have $6|(z-x)$. In fact, at our point $P$, $\chi(F)\propto\ch_3^P(u)=0$, so this holds automatically.

We have $\ch^{\alpha,\beta}_2(E)=\frac{2}{3}\beta^2$,
\[\ch^{\alpha,\beta}_2(F)=\frac{y}{2}-\beta x+\frac{1}{3}\beta^2 r+\frac{1}{2} r\]
and
\[\ch^{\alpha,\beta}_2(G)=-\frac{y}{2}+\beta x-\frac{1}{3}\beta^2(r-2)-\frac{1}{2}r.\]
These provide the constraints
\begin{equation}0<\frac{y}{2}+2x+\frac{11}{6}r<\frac{8}{3}.\label{ch2_constraint}\end{equation}
Since $F\in\mathcal{T}_{\beta}$, we have $x\geq -2r$. In addition, we have 
$\lambda_{1/\sqrt{3},-2,1/3}(F)=0$, which gives 
\[z=-11x-6(r+y).\]
Finally, we also have
\begin{equation}\label{Q_constraint}
\begin{split}
Q_P(F)&=12 r^2+28 r x+\frac{23}{3} r y+\frac{46}{3} x^2+8 x y+y^2
\qquad\text{and}\\
Q_P(v-u)&=12 r^2+28 r x+\frac{23}{3} r y-\frac{98}{3} r+\frac{46}{3} x^2+8 x y-40 x+y^2-\frac{34}{3}
   y+\frac{64}{3}. 
\end{split}
\end{equation}
From Proposition~\ref{sheaf_subobjects}, there are some constraints on $F$ and $G$. We will focus on cases~\eqref{sheaf_subobjects1} and~\eqref{sheaf_subobjects2} of that proposition. We first eliminate case~\eqref{sheaf_subobjects2}.  

\begin{lemma}\label{F_G_sheaf}
Given $E$, $F$ and $G$ as above,  $F\not\in\coh(X)$ and $G\not\in\cohb[1]\cap\coh(X)[1]$.
\end{lemma}

\begin{proof}
The assumption $G[-1]\in\cohb$ implies that $x\geq2(r-2)$, and $r\geq2$, so the constraints from $\ch_2^P(F)$ imply that $y/2<32/3-35r/6\leq-1$ and so $y\leq -3$.  We also have
\[2x<\frac{8}{3}-\frac{y}{2}-\frac{11}{6}r.\]
Then from \eqref{Q_constraint}, we have 
\[Q_P(F)<\frac{1601}{12}r^2-\frac{3440}{9} r+\frac{7648}{27}-3\left(y-\frac{32}{9}+\frac{7}{6}r\right)^2.\]

Now $y<64/3-35r/3$, and when $r>3$, it is easy to see that the above bound on $Q_P(F)$ is negative. When $r=2$, we have $-22<3y+12x<-6$ from the constrains, while $Q_P(u)<0$ in the range $-23-\sqrt{f(x)}<3y+12x<-23+\sqrt{f(x)}$ for some (quadratic) function $f(x)$ and $Q_P(v-u)<0$ in the range $-6-\sqrt{6}x<3y+12x<-6+\sqrt{6}x$; hence one of these must be negative so long as the ranges overlap. This happens when $x\geq 2$. When $x=0$, we have $G[-1]\in\cohab$, which gives a contradiction. If $x=1$, then the constraints give $-17/3<y/2<-2$. But $G$ is a sheaf, and so $y\cong 0\pmod 2$. Hence $y=-6,-8,-10$ are the only possibilities. For $y=-6$, $Q_P(v-u)=-2/3$. For $y=-8$, $Q_P(u)=-10/3$. For $y=-10$, $Q_P(u)=-14$.  When $r=3$, the constraints give $-66<6y+24x<-34$, while $Q_P(u)<0$ in the range $-69-\sqrt{g(x)}<6y+24x<-69+\sqrt{g(x)}$ and $Q_P(v-u)<0$ in the range $-35-\sqrt{h(x)}<6y+24x<-35+\sqrt{h(x)}$, and again one or the other must be negative for any value of $y$ or $x$, and this time the ranges overlap for all $x\geq0$.

In any case, $F$ or $G$ is not $\lambda_P$-semistable, which is impossible.
\end{proof}

Given Lemma~\ref{N1_no_tilt}, it is reasonable to assume   $\Delta_{21}(\ch(E),\ch(F))\geq0$ (with equality only if it holds for all $\beta$), which gives the inequality
\begin{equation}\label{tilt_inequal}
y+\frac{8}{3}x+r<0.
\end{equation}
We take this inequality as an ansatz to eliminate possibilities rather than a necessary condition for a numerical wall to be actual. 

\begin{proposition}\label{gamma+inst}
Consider $v=(2,0,-1,0)$. Then on $X=\p3$, there are five pseudo $\lambda$-walls for a Gieseker stable sheaf $E$ intersecting $\Gamma^-_{v,1/3}$, given by
\begin{align}
\label{ideal_wall}  0\to\mathcal{I}_L\to \hbox{}&E\to\mathcal{I}_p\to 0,\\
\label{line_wall}   0\to K\to \hbox{}&E\to \mathcal{O}_L(-1)\to 0,\\
  \label{killer_wall} 0\to\op3(-1)\to\hbox{}&E\to\mathcal{I}_{C}(1)\to0,\\
  \label{killer_wall2} 0\to A\to & E\to\op3(-3)[2]\to0,\\
\label{weird_wall}  0\to S_p(1)\to \hbox{}& E\to \mathcal{I}_{p/H}\to 0,
\end{align}
where $L$ is a line, $C$ is the union of two lines on a quadric surface in $\p3$ and $H$ is a hyperplane. Moreover, $K$ is a Gieseker semistable sheaf with $\ch(K)=(2,0,-2,2)$, as described in display \eqref{k-sqc}, and $A$ is quasi-isomorphic to a complex $\op3(-2)^{\oplus 4}\to\op3(-1)^{\oplus 5}$.  The first two are coincident, and only \eqref{killer_wall2} is an actual $\lambda$-wall below $\Gamma^-_{v,1/3}$, while only \eqref{killer_wall} is actual above $\Gamma^-_{v,1/3}$.
\end{proposition}

\begin{proof}
We assume $E$ is a Gieseker stable torsion-free
sheaf. These are classified by Proposition~\ref{(2,0,-1,0)}. 
We start by looking at an exceptional case not described in Proposition~\ref{sheaf_subobjects}. This is where $F$ is not a sheaf and the destabilizing sequence is of the form $F\to E\to F_{11}[2]$. This arises from the presentation of $E$ in \eqref{Epres}, 
\[0\to\op3(-3)\to\op3(-2)^{\oplus 4}\to\op3(-1)^{\oplus 5}\to E\to 0.\]
Viewing the middle map as a 2-step complex $A$ in $D^b(X)$, near $P$ we have that $A\to E\to A_{11}[2]$ is short exact in $\mathcal{A}^P$. It is easy to see that this is a wall and $A_{11}=\op3(-3)$, giving \eqref{killer_wall2}. From the definition of $A$, it follows that
$$ \Ext^1(A,\calo(-3)[2])=\Ext^3(A,\calo(-3))=0; $$
thus where this wall is actual, it is a vanishing wall. The presentation of $A$ shows that it has a vanishing $\lambda$-wall given by $\calo(-2)$ (as a sub-object) which goes through the point $P$; denote this wall by $\mathcal{W}$. One can check that $\mathcal{W}$ is above \eqref{killer_wall2} to the right of
$\Gamma^-_{v,1/3}$, but it is below \eqref{killer_wall2} to the left of
$\Gamma^-_{v,1/3}$. This shows that \eqref{killer_wall2} is actual to the left but not to the right of $\Gamma^-_{v,1/3}$.

Once we have found the walls for $\mu$-semistable torsion-free sheaves $E$, we need to also consider new objects which become $\lambda$-stable as we cross each wall. Lemma~\ref{F_G_sheaf} and Proposition~\ref{sheaf_subobjects} show that it is reasonable to assume our destabilizing objects $F$ and $G$ are in $\cohab\cap\cohb\cap\coh(X)$. Then we have additional constraints to add to \eqref{ch2_constraint} and \eqref{Q_constraint}:
\begin{equation}
\label{ch1_constraint}-2r <x\leq0\\
\end{equation}
since $E$ is $\mu$-semistable and $F\in\cohb$. We also add the ansatz \eqref{tilt_inequal} and consider various possible cases for the Chern character $(r,x,y/2,z/6)$ of $F$.

First consider the case $r=1$. 
\subsubsection*{Case $(1,0,y/2,z/6)$}
Suppose $x=0$. Then the $\ch_2^P$-constraint gives $5/3>y>-11/3$. Then $1\geq y\geq -3$. From \eqref{tilt_inequal}, we have $y\leq-1$, and so we have the possibilities $y=-1,-2,-3$. But $F$ is of rank $1$ and torsion-free, and so $y\cong0\pmod 2$; so $y=-2$ is the only possibility. This gives the first wall \eqref{ideal_wall}.

\subsubsection*{Case $(1,-1,y/2,z/6)$}
Now suppose $x=-1$. The constraints on $y$ give $17/6> y\geq 1$. But \eqref{tilt_inequal} implies $y<8/3-r=5/3$, and so $y=1$ is the only possibility. This gives wall \eqref{killer_wall}.

\medskip
Finally, we return to the case $r=2$. Then $G$ is a torsion sheaf. We have $-3\leq x\leq0$. We shall consider these, case by case as before.

\subsubsection*{Case $(2,0,y/2,z/6)$}
We now consider the special case $x=0$.  The constraints give $-6\leq y<-2$.  But $y$ must be even for $G$ to be a sheaf, and so $y\leq-4$. When $y=-4$, we have the same wall as \eqref{ideal_wall} but given by the sequence \eqref{line_wall}. When $y=-6$, we have $Q_P(F)=-8$, and so $F$ cannot be semistable.

\subsubsection*{Case $(2,-1,y/2,z/6)$}
Now we let $x=-1$. Then the constraints give $0\geq y>-10/3$, so $-3\leq y\leq0$.  But $2\ch_2(G)\cong 1\pmod 2$. Hence $y=-1$ or $y=-3$. In the former case, $\ch(F)=(2,-1,-1/2,5/6)$, which corresponds to sequence \eqref{weird_wall}.  But one can check that $Q_P(2,-1,-1/2,5/6)=1>0$, $Q_P(0,1,-1/2,-5/6)=25/3>0$ and $Q_P(v)-Q_P(2,-1,-1/2,5/6)-Q_P(0,1,-1/2,-5/6)=12>0$, so this is a pseudo-wall.  In the latter case, $Q_P(u)=-17/3$, so $F$ cannot be $\lambda$-semistable.

\subsubsection*{Case $(2,-2,y/2,z/6)$}
When $x=-2$, we have $1\leq y\leq 5$ from the $\ch_2^P$-constraints and $y+2\leq 4$ from \eqref{tilt_inequal}. So $1\leq y\leq 3$. But $2\ch_2(G)\cong 0\pmod 2$, and so $y=2$ is the only possible solution. Then $F=\op3(-1)^{\oplus2}$ and $G=\cali_C(1)/\op3(-1)$, but this gives the same wall as \eqref{killer_wall}.

\subsubsection*{Case $(2,-3,y/2,z/6)$}
When $x=-3$, the $\ch_2^P$-constraint gives $y>14/3$ and so $y\geq 5$. But \eqref{tilt_inequal} gives $y+2<8$ and so $y=5$. In this case, $z=-6$. But then $\chi(F)=5/2$, and so this is impossible.

\medskip
This completes all of the possible numerical cases satisfying our ansatzes.

Note that the walls for \eqref{ideal_wall} and \eqref{line_wall} coincide. For $s\geq1/3$, the walls below $\Gamma^-_{v,s}$ are in the order \eqref{killer_wall2}$>$\eqref{ideal_wall}=\eqref{line_wall}$>$\eqref{weird_wall}$>$\eqref{killer_wall}, and this is reversed above $\Gamma^-_{v,s}$ (see Figure~\ref{fig (2,0,-1,0)}).

Finally, we observe that $\cali_C(1)$ is destabilized by $\calo(-1)^{\oplus 4}$. In fact, the quotient is $\Omega^2(1)[1]$, and this is stable to the right of $\Gamma^-_{v,1/3}$ and to the left of $\Theta^-_{\calo(-2)}$. The wall for this is below \eqref{killer_wall} to the right of $P$ and so \eqref{killer_wall} is actual in this region.
\end{proof}

\begin{remark}
Note that the problem of eliminating walls when $x\leq -2$ and $r=1$ is actually a quadratic programming problem. To see this, note that
\[z= -2\beta^2 x+3\beta y+3\beta-3x,\]
and so
\[Q_P(F)=3 \beta ^2+\frac{10}{3} \beta ^2 x^2+2 x^2-2 \beta ^3 x-6 \beta  x-4 \beta  x y+y^2+\frac{5}{3} \beta ^2y+y,\]
which we can rewrite as
\[F_t(\mathbf{v})=\mathbf{v}A_t\mathbf{v}^T-\mathbf{c}\mathbf{v}^T\]
subject to the inequalities 
\begin{gather*}
4\leq\gamma<\lambda/2<\mu/4,\\ 
-2-\gamma\leq y\leq\gamma,\\
\lambda-\mu/3<({y+1})/2<\lambda+\mu/3,
\end{gather*}
where $\mathbf{v}=(\lambda,\mu,\gamma,y)=(\beta x,\beta^2,x^2,y)$, 
\[A_t=\begin{pmatrix}10(1-t)/3&-1&0&-2\\-1&0&5t/3&5/6\\0&5t/3&0&0\\-2&5/6&0&1\end{pmatrix}
\]
and $c=(6,-3,-2,-1)$. Then we need to show there is a real value of $t$ for which $F_t(\mathbf{v})<0$ given the constraints.  In fact, numerical methods show that the ``local'' maximum value of $F(\mathbf{v})$ when $t=1$ is negative even without the additional non-linear constraint $\gamma\mu=\lambda^2$. But $A_t$ is not definite, and so this need not be a global maximum.
\end{remark}

\begin{figure}[ht] \centering
\includegraphics{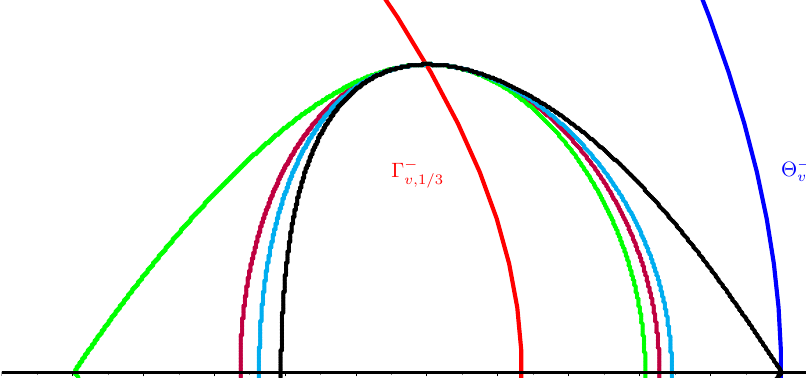}
\caption{This graph shows the actual and pseudo $\lambda$-walls for the Chern character $v=(2,0,-1,0)$ intersecting at the point $(\alpha=1/\sqrt{3},\beta=-2)$, as described in Proposition~\ref{gamma+inst}. The green curve is \eqref{killer_wall2}, the black one is \eqref{killer_wall}, the cyan one is \eqref{weird_wall}, and the purple one is \eqref{ideal_wall}. The outermost curve (black and green) is the actual $\lambda$-wall, while the remaining curves are only pseudo $\lambda$-walls. The stability chambers in the region $R^-_v$ are described in Theorem~\ref{chamber-}.}
\label{fig (2,0,-1,0)}
\end{figure}

It follows from Theorem~\ref{RLvs} that no other $\lambda$-walls in the region $R^-_v$ exist when $s=1/3$. This observation allows us to give a complete chamber decomposition of this region, summarized in the following statement.

  \begin{theorem}\label{chamber-}
{\samepage Let $X=\p3$, and consider the numerical Chern character $v=(2,0,-1,0)$; fix $s=1/3$.
The region $R^-_v$ is divided into two stability chambers $C_i$ within which the $\labs$-stable objects are described as follows:
\begin{enumerate}[\rm{(C\arabic*)},ref={C\arabic*}]
\item\label{c1} null correlation sheaves,
\item\label{c2} no stable objects.
\end{enumerate}}
\end{theorem}

The two stability chambers just described are pictured in Figure~\ref{fig (2,0,-1,0)} as follows: the chamber~\eqref{c1} lies above the green and black curves (which correspond to the walls \eqref{killer_wall2} and \eqref{killer_wall}, respectively).

\medskip

As a final remark, we observe that Proposition~\ref{gamma+inst} provides concrete examples of intersecting actual $\lambda$-walls for the Chern character $(2,0,-1,0)$.

\newcommand{\etalchar}[1]{$^{#1}$}


\begin{thebibliography}{BMS{\etalchar{+}}17+++}

\bibitem[AM16]{AM}
W.~Alagal and A.~Maciocia,
\emph{Critical $k$-very ampleness for abelian surfaces}, 
Kyoto J.\ Math.\ {\bf 56} (2016), 33--47.

\bibitem[AB13]{AB}
D.~Arcara and A.~Bertram,
\emph{Bridgeland-stable moduli spaces for K-trivial surfaces}, 
With an appendix by Max Lieblich.
J.\ Eur.\ Math.\ Soc.\ \textbf{15} (2013), 1--38. 

\bibitem[Bal11]{Balmer}
P.~Balmer,
\emph{Separability and triangulated categories},
Adv.\ Math.\ \textbf{226} (2011), 4352--4372.

\bibitem[BM14a]{BM14a}
A.~Bayer and E.~Macr\`i,
\emph{MMP for moduli of sheaves on K3s via wall-crossing: nef and movable cones, Lagrangian fibrations},
Inv.\ Math.\ \textbf{198} (2014), 505--590.

\bibitem[BM14b]{BM14b}
\bysame,
\emph{Projectivity and birational geometry of Bridgeland moduli spaces}, 
J.\ Amer.\ Math.\ Soc.\ \textbf{27} (2014), 707--752.

\bibitem[BMS16]{BMS}
A.~Bayer, E.~Macr\`i and P.~Stellari,
\emph{The space of stability conditions on abelian 3-folds, and on some Calabi--Yau threefolds}, 
Inv.\ Math.\ {\bf 206} (2016), 869--933.

\bibitem[BMT14]{BMT}
A.~Bayer, E.~Macr\`i and Y.~Toda,
\emph{Bridgeland stability conditions on threefolds I: Bogomolov--Gieseker type inequalities}, 
J.\ Algebraic Geom.\ {\bf 23} (2014), 117--163.

\bibitem[BMS{\etalchar{+}}17]{BMSZ}
M.~Bernardara, E.~Macr\`i, B.~Schmidt and X.~Zhao,
\emph{Bridgeland stability conditions on Fano threefolds}, 
\'Epijournal G\'eom.\ Alg\'ebrique  \textbf{1}  (2017), Art.~2, 24 pp.

\bibitem[Bri07]{B07}
T.~Bridgeland,
\emph{Stability conditions on triangulated categories}, 
Ann.\ of Math.\ \textbf{166} (2007), 317--345.

\bibitem[Bri08]{B08}
\bysame,
\emph{Stability conditions on K3 surfaces}, 
Duke Math.\ J.\ {\bf 141} (2008), 141--291.

\bibitem[Cha84]{Chang}
M.\,C.~Chang,
\emph{Stable rank $2$ reflexive sheaves on $\p3$ with small $c_2$ and applications}, 
Trans.\ Amer.\ Math.\ Soc.\ {\bf 284} (1984), 57--84.

\bibitem[Ein82]{Ein}
L.~Ein,
\emph{Some stable vector bundles on $\p4$ and $\p5$}, 
J.~reine angew.\ Math.\ \textbf{337} (1982), 142--153.

\bibitem[Fey16]{SF16}
S.~Feyzbakhsh,
\emph{Stability of restrictions of Lazarsfeld-Mukai bundles via wall-crossing, and Mercat's conjecture}, 
preprint, \arXiv{1608.07825}.

\bibitem[Fey17]{SF17}
\bysame,
\emph{Mukai's program (reconstructing a K3 surface from a curve) via wall-crossing}, J.~reine angew.\ Math.\ \textbf{765} (2020).


\bibitem[FL21]{FL}
S.~Feyzbakhsh and C.~Li,
\emph{Higher rank Clifford indices of curves on a K3 surface},  Selecta Math.\ Vol 27 (2021). 

\bibitem[GHS18]{GHS}
P.~Gallardo, C.\,L.~Huerta and B.~Schmidt,
\emph{Families of elliptic curves in $\p3$ and Bridgeland stability}, 
Michigan Math.~J.\ \textbf{67} (2018), 787--813.

\bibitem[HL10]{HL}
D.~Huybrechts and M.~Lehn, 
\emph{The geometry of moduli spaces of sheaves}, 2nd ed.,  
Cambridge University Press (2010).

\bibitem[Li19a]{Li2}
C.~Li,
\emph{On stability conditions for the quintic threefold}, Inv.\ Math.\ \textbf{218} (2019) 301--340.


\bibitem[Li19b]{Li}
\bysame,
\emph{Stability conditions on Fano threefolds of Picard number one}, 
J.\ Eur.\ Math.\ Soc.\ {\bf 21} (2019), 709--726.


\bibitem[Mac14]{M}
A.~Maciocia,
\emph{Computing the walls associated to Bridgeland stability conditions on projective surfaces}, 
Asian J.\ Math.\ {\bf 18} (2014), 263--279.

\bibitem[MM13]{MM}
A.~Maciocia and C.~Meachan,
\emph{Rank one Bridgeland stable moduli spaces on a principally polarized abelian surface}, 
Int.\ Math.\ Res.\ Not. \textbf{9} (2013), 2054--2077. 

\bibitem[MP16]{MP}
A.~Maciocia and D.~Piyaratne,
\emph{Fourier--Mukai transforms and Bridgeland stability conditions on abelian threefolds, II}, 
Internat.\ J.\ Math.\ {\bf 27} (2016), article id.~1650007.

\bibitem[Mac14]{M-p3}
E.~Macr\`i,
\emph{A generalized Bogomolov--Gieseker inequality for the three-dimensional projective space}, 
Algebra Number Theory {\bf 8} (2014), 173--190

\bibitem[MS17]{MS}
E.~Macr\`i and B.~Schmidt,
\emph{Lectures on Bridgeland stability}, 
In: \emph{Moduli of curves}, 139--211,
Lect.\ Notes Unione Mat.\ Ital.\ {\bf 21}, Springer, Cham, 2017.

\bibitem[MS20]{MS2}
\bysame,
\emph{Derived categories and the genus of space curves}, Alg.~Geom.\ {\bf 7} (2020), 153--191,

\bibitem[MS19]{MSD}
C.~Martinez and B.~Schmidt,
\emph{Bridgeland Stability on Blow Ups and Counterexamples}, Math.~Z.\ \textbf{292} (2019) 1495--1510. 



\bibitem[Mea12]{Mea}
  C.~Meachan,
\emph{Moduli of Bridgeland-Stable objects}, 
  PhD Thesis, Edinburgh, 2012. Available from \url{https://era.ed.ac.uk/handle/1842/6230}. 

\bibitem[MRT94]{MRT}
R.\,M.~Miro-Roig and G.~Trautmann,
\emph{The moduli scheme $M(0,2,4)$ over $\mathbb{P}_3$}, 
Math.\ Z.\ {\bf 216} (1994), 283--315.

\bibitem[OSS80]{OSS}
C.~Okonek, M.~Schneider and H.~Spindler, 
\emph{Vector bundles on complex projective spaces}, 
Progress in Math.\ \textbf{3}, Birkhauser, Boston, MA, 1980.

\bibitem[Piy17]{Piy}
D.~Piyaratne,
\emph{Stability conditions, Bogomolov--Gieseker type inequalities and Fano 3-folds}, 
preprint, \arXiv{1705.04011}. 

\bibitem[PT19]{PT}
D.~Piyaratne and Y.~Toda,
\emph{Moduli of Bridgeland semistable objects on 3-folds and Donaldson--Thomas invariants}, 
J.\ reine angew.\ Math.\ \textbf{3} (2019), 175--219.

\bibitem[Rud97]{R}
A.~Rudakov,
\emph{Stability for an abelian category}, 
J.~Algebra \textbf{197} (1997), 231--245.

\bibitem[Sch14]{S-q3}
B.~Schmidt,
\emph{A generalized Bogomolov--Gieseker inequality for the smooth quadric theefold}, 
Bull.\ London Math.\ Soc.\ \textbf{46} (2014), 915--923.

\bibitem[Sch17]{S17}
\bysame
\emph{Counterexample to the generalized Bogomolov--Gieseker inequality for threefolds}, 
Int.\ Math.\ Res.\ Not.\ \textbf{8} (2017), 2562--2566.

\bibitem[Sch20a]{S}
\bysame
\emph{Bridgeland stability conditions on threefolds - some wall crossings}, J.~Alg.\ Geom., \textbf{29} (2020), 247--283.

\bibitem[Schi20b]{S18}
\bysame
\emph{Rank two sheaves with maximal third Chern character in three-dimensional projective space},  Mat.\ Contemp.\ {\bf 47} (2020), 228--270.

\bibitem[SS21]{SS}
B.~Schmidt and B.~Sung,
\emph{Discriminants of stable rank two sheaves on some general type surfaces},  Res.\ Lett., {\bf 28} (2021), no.~1, 245--270.

\bibitem[Tod10]{Tod10}
Y.~Toda,
\emph{Curve counting theories via stable objects I. DT/PT correspondence}, 
J.~Amer.\ Math.\ Soc.\ {\bf 23} (2010), 1119--1157.

\bibitem[YY14]{YY12}
S.~Yanagida and K.~Yoshioka,
\emph{Bridgeland's stabilities on abelian surfaces}, 
Math.~Z.\ \textbf{276} (2014), 571--610.

\end{thebibliography}
\end{document}